%% file: squotientopes.tex
\documentclass{amsart}

\usepackage[T1]{fontenc}
\usepackage{enumerate, amsmath, amsfonts, amssymb, amsthm, thmtools, mathrsfs, wasysym, graphics, graphicx, xcolor, url, hyperref, hypcap, shuffle, xargs, multicol, overpic, pdflscape, multirow, hvfloat, minibox, accents, array, multido, xifthen, a4wide, ae, aecompl, blkarray, pifont, mathtools, etoolbox, dsfont, stmaryrd}
\usepackage{marginnote}
\hypersetup{colorlinks=true, citecolor=darkblue, linkcolor=darkblue}
\usepackage[all]{xy}
\usepackage[bottom]{footmisc}
\usepackage{tikz}
\usepackage{tikz-qtree}
\usetikzlibrary{trees, decorations, decorations.markings, shapes, arrows, matrix, calc, fit, intersections, patterns, angles}
\usepackage[external]{forest}
\usepackage{varwidth}
\graphicspath{{figures/}{figures/nodes/}}
\makeatletter\def\input@path{{figures/}}\makeatother
\usepackage{caption}
\usepackage{placeins} 
\usepackage{afterpage} 
\usepackage[noabbrev,capitalise]{cleveref}
\usepackage[export]{adjustbox}
\usepackage{ulem}
\usepackage{picins/picins}

\newtheorem{theorem}{Theorem}
\newtheorem{corollary}[theorem]{Corollary}
\newtheorem{proposition}[theorem]{Proposition}
\newtheorem{lemma}[theorem]{Lemma}
\newtheorem{conjecture}[theorem]{Conjecture}
\crefname{conjecture}{Conjecture}{Conjectures}
\newtheorem*{theorem*}{Theorem}
\newtheorem*{corollary*}{Corollary}

\theoremstyle{definition}
\newtheorem{definition}[theorem]{Definition}

\newtheorem{example}[theorem]{Example}
\newtheorem{remark}[theorem]{Remark}

\crefname{notation}{Notation}{Notations}
\crefname{problem}{Problem}{Problems}

\newcommand{\R}{\mathbb{R}} 
\newcommand{\N}{\mathbb{N}} 
\newcommand{\Z}{\mathbb{Z}} 
\renewcommand{\b}[1]{{\boldsymbol{#1}}} 
\renewcommand{\c}[1]{\mathcal{#1}} 

\newcommand{\set}[2]{\left\{ #1 \;\middle|\; #2 \right\}} 
\newcommand{\bigset}[2]{\big\{ #1 \;\big|\; #2 \big\}} 
\newcommand{\ssm}{\smallsetminus} 
\newcommand{\dotprod}[2]{\left\langle \, #1 \; \middle| \; #2 \, \right\rangle} 
\newcommand{\one}{\b{1}} 
\newcommand{\eqdef}{\mbox{\,\raisebox{0.2ex}{\scriptsize\ensuremath{\mathrm:}}\ensuremath{=}\,}} 
\newcommand{\defeq}{\mbox{~\ensuremath{=}\raisebox{0.2ex}{\scriptsize\ensuremath{\mathrm:}} }} 
\renewcommand{\implies}{\Rightarrow} 
\newcommand{\iscovered}{\lessdot} 

\DeclareMathOperator{\conv}{conv} 
\DeclareMathOperator{\trees}{\mathsf{T}} 
\DeclareMathOperator{\leftTree}{\mathsf{L}} 
\DeclareMathOperator{\rightTree}{\mathsf{R}} 
\DeclareMathOperator{\inv}{inv} 
\DeclareMathOperator{\ninv}{ninv} 

\newcommand{\ie}{\textit{i.e.}~} 
\newcommand{\eg}{\textit{e.g.}~} 
\newcommand{\viceversa}{\textit{vice versa}} 
\definecolor{darkblue}{rgb}{0,0,0.7} 
\definecolor{green}{RGB}{57,181,74} 
\definecolor{violet}{RGB}{147,39,143} 
\newcommand{\darkblue}{\color{darkblue}} 
\newcommand{\defn}[1]{\textsl{\darkblue #1}} 

\usepackage{todonotes}

\newcommand{\meet}{\wedge} 
\newcommand{\join}{\vee} 
\newcommand{\bigMeet}{\bigwedge} 
\newcommand{\bigJoin}{\bigvee} 
\newcommandx{\projDown}[1][1={}]{\smash{\pi_\downarrow^{#1}}} 
\newcommandx{\projUp}[1][1={}]{\smash{\pi^\uparrow_{#1}}} 
\newcommand{\con}{\mathrm{con}} 

\newcommand{\s}{\b{s}} 
\newcommandx{\Fan}[1][1=n]{\mathcal{F}(#1)} 
\newcommand{\polytope}[1]{\mathds{#1}} 
\newcommandx{\Perm}[1][1=d]{\polytope{P}\mathrm{erm}(#1)} 
\newcommandx{\Asso}[1][1=d]{\polytope{A}\mathrm{sso}(#1)} 
\newcommandx{\Zono}[1][1=\s]{\polytope{Z}\mathrm{ono}(#1)} 
\DeclareMathOperator{\pos}{pos} 
\DeclareMathOperator{\lpos}{lpos} 
\DeclareMathOperator{\rpos}{rpos} 
\DeclareMathOperator{\ins}{\mathsf{B}} 
\newcommandx{\sFoam}[1][1=\s]{\c{F}_{#1}} 
\newcommandx{\quotientFan}[1][1=\equiv]{\c{F}_{#1}} 
\newcommandx{\quotientFoam}[1][1=\equiv]{\c{F}_{#1}} 
\newcommandx{\shard}[1][1=\alpha]{\Sigma_{#1}} 
\newcommandx{\shardPolytope}[1][1=\alpha]{\polytope{SP}_{#1}} 
\newcommandx{\localShardPolytope}[2][1=\alpha, 2=\b{q}]{\polytope{S}_{#1}^{#2}} 
\newcommandx{\quotientope}[1][1=\equiv]{\polytope{Q}_{#1}} 
\newcommandx{\shardoplex}[1][1=\alpha]{\polytope{S}_{#1}} 
\newcommandx{\quotientoplex}[1][1=\equiv]{\polytope{Q}_{#1}} 
\newcommandx{\sTrees}[1][1=\s]{\b{\Gamma}_{#1}} 
\newcommandx{\sTrunks}[1][1=\s]{\b{\Lambda}_{#1}} 
\newcommand{\altmatch}{\mu} 
\newcommandx{\altmatchset}[1][1=\alpha]{\mathcal{M}_{#1}} 
\newcommand{\charvect}{\point \chi} 
\newcommandx{\tree}[1][1=T]{{\mathsf{#1}}} 
\newcommand{\bush}{\tree[B]} 
\newcommand{\trunk}{\tree[T]} 
\newcommand{\fiber}[1]{\polytope{F}_{#1}} 
\newcommand{\closedFiber}[1]{\bar{\polytope{F}}_{#1}} 

\newcommand{\decoration}{{\b{\delta}}} 
\newcommand{\includeSymbol}[1]{\ensuremath{%
	\mathchoice
		{\raisebox{-.7mm}{\includegraphics[height=2.2ex]{#1}}}	
		{\raisebox{-.7mm}{\includegraphics[height=2.2ex]{#1}}}
		{\raisebox{-.6mm}{\includegraphics[height=1.6ex]{#1}}}
		{\raisebox{-.5mm}{\includegraphics[height=1ex]{#1}}}
}}
\robustify{\includeSymbol}
\newcommand{\noneCirc}{\includeSymbol{none}}
\newcommand{\upCirc}{\includeSymbol{up}}
\newcommand{\downCirc}{\includeSymbol{down}}
\newcommand{\upDownCirc}{\includeSymbol{updown}}
\newcommand{\Decorations}{\{\noneCirc{}, \downCirc{}, \upCirc{}, \upDownCirc{}\}} 

\def\tropring{\mathbb{T}} 
\newcommandx{\troppol}[1][1=F]{#1} 
\newcommand{\trophyp}[1][\troppol]{\mathcal{T}(#1)} 
\newcommand{\tropdome}[1][\troppol]{\mathcal{D}(#1)} 

\newcommand{\sprod}[2]{\langle {#1} , {#2} \rangle} 
\newcommand{\point}[1]{\b{#1}} 
\newcommand{\subdiv}[1][S]{\b{\mathcal{#1}}} 
\newcommand{\lift}{\ell}
\newcommandx{\pc}[1][1=A]{\b{#1}} 

\makeatletter
\def\l@part{\@tocline{1}{8pt}{0pc}{}{}}
\def\l@section{\@tocline{1}{4pt}{0pc}{}{}}
\makeatother
\let\oldtocpart=\tocpart
\renewcommand{\tocpart}[2]{\sc\large\oldtocpart{#1}{#2}}
\let\oldtocsection=\tocsection
\renewcommand{\tocsection}[2]{\bf\oldtocsection{#1}{#2}}
\let\oldtocsubsubsection=\tocsubsubsection
\renewcommand{\tocsubsubsection}[2]{\quad\oldtocsubsubsection{#1}{#2}}

\title{\mbox{Geometric realizations of the $\b{s}$-weak order and its lattice quotients}}

\thanks{
VP was partially supported by the Spanish project PID2022-137283NB-C21 of MCIN/AEI/10.13039/501100011033 / FEDER, UE, by the Spanish--German project COMPOTE (AEI PCI2024-155081-2 \& DFG 541393733), by the Severo Ochoa and María de Maeztu Program for Centers and Units of Excellence in R\&D (CEX2020-001084-M), and by the Departament de Recerca i Universitats de la Generalitat de Catalunya (2021 SGR 00697).
VP and EP were partially supported by the French--Austrian project PAGCAP (ANR-21-CE48-0020 \& FWF I 5788).
}

\author{Eva Philippe}
\address[Eva Philippe]{Sorbonne Université and Universitat de Barcelona.}
\email{eva.philippe@imj-prg.fr}
\urladdr{\url{https://perso.imj-prg.fr/eva-philippe/}}

\author{Vincent Pilaud}
\address[Vincent Pilaud]{Universitat de Barcelona \& Centre de Recerca Matemàtica, Barcelona}
\email{vincent.pilaud@ub.edu}
\urladdr{\url{https://www.ub.edu/comb/vincentpilaud/}}

\begin{document}

\begin{abstract}
For an $n$-tuple~$\s$ of non-negative integers, the $\s$-weak order is a lattice structure on $\s$-trees, generalizing the weak order on permutations.
We first describe the join irreducible elements, the canonical join representations, and the forcing order of the $\s$-weak order in terms of combinatorial objects, generalizing the arcs, the non-crossing arc diagrams, and the subarc order for the weak order.
We then extend the theory of shards and shard polytopes to construct geometric realizations of the $\s$-weak order and all its lattice quotients as polyhedral complexes, generalizing the quotient fans and quotientopes of the weak order.
\end{abstract}

\vspace*{-1.8cm}
\maketitle

\vspace*{-.8cm}
\tableofcontents


\renewcommand{\thetheorem}{\Alph{theorem}}%

\section*{Introduction}
\label{sec:introduction}

The structure of permutations and associations of an $n$-element set is a classical topic of algebraic and geometric combinatorics.
In combinatorics, it is encoded by the Cayley graph of permutations under simple transpositions and the rotation graph on binary trees.
In lattice theory, it materializes in the lattice morphism from the weak order on permutations to the Tamari lattice on binary trees.
In polyhedral geometry, it appears in the braid arrangement and the sylvester fan, and their polar permutahedron and associahedron.
See \cite{PilaudSantosZiegler} for a recent survey on these connections and their influence in mathematics.

This prototype has motivated the study of all lattice congruences of the weak order, pioneered by N.~Reading~\cite{Reading-LatticeCongruences}.
Combinatorially, he provided an elegant combinatorial model for the lattice theory of the weak order in~\cite{Reading-arcDiagrams}.
Namely, the join irreducible permutations are encoded as certain arcs wiggling around the vertical axis, the canonical join representations of the permutations are encoded by non-crossing arc diagrams, and the forcing order on join irreducible permutations is encoded by the subarc order on these arcs.
Geometrically, he showed in~\cite{Reading-HopfAlgebras} that coarsening the braid fan according to the equivalence classes of any congruence of the weak order always yields a complete polyhedral fan.
These quotient fans were shown to be normal fans of so-called quotientopes by V.~Pilaud and F.~Santos~\cite{PilaudSantos-quotientopes}.
Later, A.~Padrol, V.~Pilaud and J.~Ritter~\cite{PadrolPilaudRitter} revisited the problem using the theory of shard polytopes, a family of polytopes indexed by the arcs (or, equivalently, by the join irreducible permutations).

\begin{figure}[b]
	\capstart
	\centerline{\includegraphics[scale=.6]{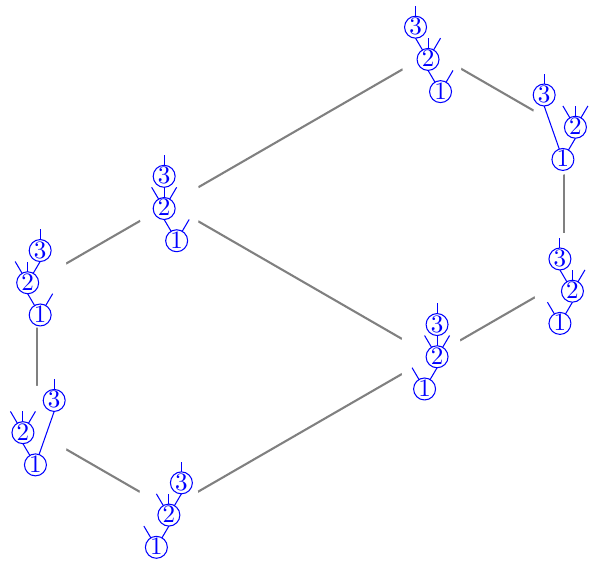} \qquad \includegraphics[scale=.6]{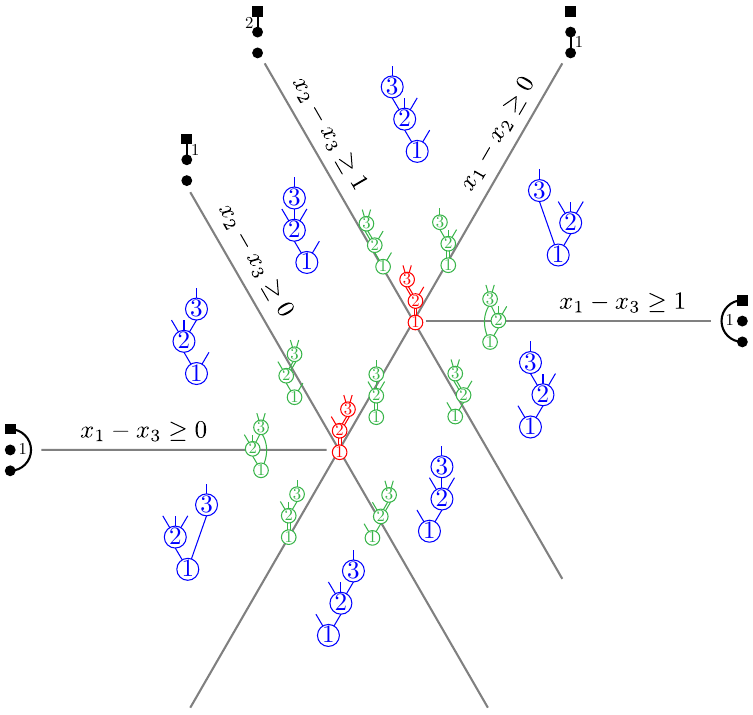}}
	\caption{The $(1,2,0)$-weak order (left) and the $(1,2,0)$-foam (right).}
	\label{fig:mainSWeakOrder}
\end{figure}

C.~Ceballos and V.~Pons introduced in~\cite{CeballosPons-sWeakOrderI,CeballosPons-sWeakOrderII} the $\s$-weak order on $\s$-trees for an $n$-tuple~$\s$ of non-negative integers, generalizing the classical weak order (the case~${\s = (1, \dots, 1)}$).
An $\s$-tree is a rooted plane tree on~$[n]$, where the node~$i$ has $s_i+1$ children that are all either leaves or nodes~$j > i$ (note that we have changed the conventions of~\cite{CeballosPons-sWeakOrderI} to make inductive arguments more transparent).
The order among these $\s$-trees is defined by inequalities between the inversion numbers in~$\s$-trees, generalizing the definition of the weak order by inclusion of inversion sets of permutations.
See \cref{fig:mainSWeakOrder}\,(left).
They proved in~\cite[Sect.~1]{CeballosPons-sWeakOrderI} that the $\s$-weak order is a lattice, described its meet and join operations, and established lattice properties of the $\s$-weak order, in particular congruence uniformity.
They also introduced in~\cite[Sect.~2]{CeballosPons-sWeakOrderI} the $\s$-Tamari lattice, a sublattice of the $\s$-weak order (and also a lattice quotient of the $\s$-weak order when~$\s$ contains no~$0$), which is isomorphic to the $\nu$-Tamari lattice of~\cite{PrevilleRatelleViennot, CeballosPadrolSarmiento-geometryNuTamari}  (for well-chosen $\s$ and~$\nu$).
They conjectured in~\cite[Conj~3.1.2]{CeballosPons-sWeakOrderII} that the Hasse diagram of the $\s$-weak order can be realized as an orientation of the skeleton of a polyhedral subdivision of the permutahedron, in the same spirit C.~Ceballos, A.~Padrol and C.~Sarmiento realized the $\nu$-Tamari lattice as an orientation of the skeleton of a polyhedral complex in~\cite{CeballosPadrolSarmiento-geometryNuTamari}.
In the situation when~$\s$ contains no~$0$, this conjecture was settled by R.~González D'Léon, A.~Morales, E.~Philippe, D.~Tamayo Jiménez, and M.~Yip in~\cite{DLeonMoralesPhilippeTamayoYip} using the theory of flow polytopes.
In this paper, we study the $\s$-weak order combining perspectives from combinatorics, lattice theory and polyhedral geometry.

\subsection*{Combinatorics}

Our first contribution is an algorithmic perspective on the $\s$-weak order.
We consider a natural \defn{insertion algorithm} from $\R^n$ to $\s$-trees, generalizing the insertion of a permutation in an increasing binary tree.
The ambiguities of this algorithm on certain non-generic points of~$\R^n$ force us to define \defn{$\s$-bushes} as certain degenerations of $\s$-trees, generalizing ordered set partitions.
We then define the \defn{$\s$-foam}~$\sFoam$ as the set of fibers of our insertion algorithm in $\s$-bushes, generalizing the braid fan.
See \cref{fig:mainSWeakOrder}\,(right).
The oriented dual graph of this polyhedral complex gives an alternative definition for the $\s$-weak order, equivalent to the original definition of~\cite{CeballosPons-sWeakOrderI}.
Exploiting the properties of this insertion algorithm enables us to prove our first battery of results.

\begin{theorem}
\label{thm:main1}
\textsf{The insertion algorithm in $\s$-bushes knows the $\s$-weak order and the facial $\s$-weak order:}
\begin{enumerate}[(i)]
\item The $\s$-foam~$\sFoam$ is a complete polyhedral complex (\cref{prop:sFoam}), whose oriented dual graph is isomorphic to the Hasse diagram of the $\s$-weak order (\cref{coro:sFoamsWeakOrder}).
\item Decomposing each insertion step yields a natural construction of the $\s$-weak order by interval doublings (\cref{prop:sWeakOrderIntervalDoblings}), thus recovering the result of~\cite[Thms.~1.21 \& 1.40]{CeballosPons-sWeakOrderI} that the $\s$-weak order is a congruence uniform lattice.
\item Considering all faces of the $\s$-foam, there is a natural \defn{facial $\s$-weak order} on $\s$-bushes (\cref{def:sFacialWeakOrder}), which is also constructible by interval doublings (\cref{prop:sFacialWeakOrderIntervalDoblings}).
\end{enumerate}
\end{theorem}

\subsection*{Lattice theory}

Our second battery of results generalizes the combinatorial description of~\cite{Reading-arcDiagrams} for the lattice theory of the weak order in terms of arcs, non-crossing arc diagrams, and subarcs.
We define an \defn{$\s$-arc} as an arc of~\cite{Reading-arcDiagrams} together with an integer bounded by~$\s$.
We then extend to all $\s$-arcs the combinatorial notions of \defn{crossings} and of \defn{subarcs} to obtain the following statement.

\begin{theorem}
\label{thm:main2}
\textsf{The $\s$-arcs combinatorially encode the lattice theory of the $\s$-weak order:}
\begin{enumerate}[(i)]
\item The join irreducible $\s$-trees are in bijection with the $\s$-arcs (\cref{prop:sArc2sTree}).
\item The $\s$-trees are in bijection with the non-crossing $\s$-arc diagrams (\cref{prop:bij_sTree2sArcDiagram}), and this bijection encodes the canonical join representations in the $\s$-weak order (\cref{prop:canonicalJoinRepresentation}).
\item The forcing order in the $\s$-weak order is isomorphic to the subarc order on $\s$-arcs (\cref{thm:forcingOrder}), so that the congruence lattice of the $\s$-weak order is isomorphic to the distributive lattice of subsets of $\s$-arcs closed under subarcs (\cref{coro:sWeakOrderCongruenceLattice}).
\end{enumerate}
\end{theorem}

We exploit\,\cref{thm:main2}\,(iii) to define relevant congruences of the $\s$-weak order and\,their\,quotients, generalizing the Tamari lattice~\cite{Tamari}, the Cambrian lattices~\cite{Reading-CambrianLattices, LangePilaud, ChatelPilaud}, the permutree lattices~\cite{PilaudPons-permutrees}, and the simple congruences of the weak order~\cite{HoangMutze, DemonetIyamaReadingReitenThomas, BarnardNovelliPilaud}.
This opens appealing conjectures about the combinatorics of these congruences (\cref{conj:Cambrian,,conj:permutrees,,conj:simpleCongruences}).

\begin{figure}[t]
	\capstart
	\centerline{\includegraphics[scale=.8]{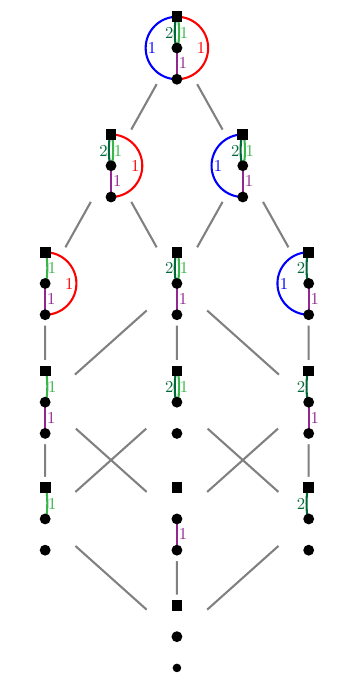} \quad \includegraphics[scale=.8]{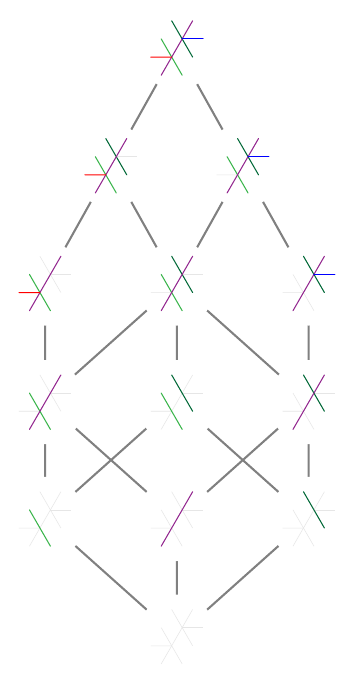} \quad \includegraphics[scale=.8]{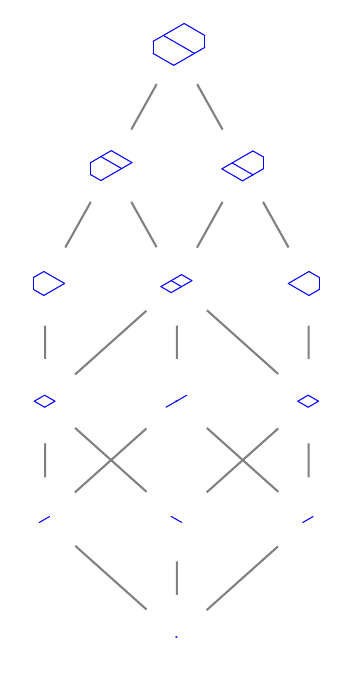}}
	\vspace{-.2cm}
	\caption{The congruence lattice of the $(1,2,0)$-weak order, where each congruence~$\equiv$ is replaced by its $\s$-arc down set (left), its quotient foam~$\quotientFoam$ (middle) and its quotientoplex~$\quotientoplex$~(right).}
	\vspace{-.2cm}
	\label{fig:mainCongruenceLattice}
\end{figure}

\subsection*{Polyhedral geometry}

We then construct geometric realizations of the $\s$-weak order and all its lattice quotients.
Namely, we first use the $\s$-foam~$\sFoam$ to construct polyhedral realizations of all lattice quotients of the $\s$-weak order, generalizing the quotient fans for the quotients of the weak order.
See \cref{fig:mainCongruenceLattice}\,(middle).
This gives our third battery of results.

\begin{theorem}
\label{thm:main3}
\textsf{All lattice quotients of the $\s$-weak order are realized by coarsenings of the $\s$-foam:}
\begin{enumerate}[(i)]
\item For any congruence~$\equiv$ of the $\s$-weak order, gluing together the maximal cells of the $\s$-foam~$\sFoam$ corresponding to $\s$-trees in the same congruence class of~$\equiv$ defines a complete polyhedral complex~$\quotientFoam$ (\cref{prop:quotientFoam2}).
\item The union of the walls of the quotient foam~$\quotientFoam$ coincides with the union of the $\s$-shards of the $\s$-arcs corresponding to~$\equiv$ (\cref{prop:quotientFoam1}), where the \defn{$\s$-shards} are affine cones associated to the $\s$-arcs, generalizing the shards of~\cite{Reading-posetRegions}.
\item The oriented dual graph of the quotient foam~$\quotientFoam$ is isomorphic to the Hasse diagram of the quotient of the $\s$-weak order by~$\equiv$ (\cref{prop:quotientFoam3}).
\end{enumerate}
\end{theorem}

We then extend the theory of shard polytopes and quotientopes~\cite{PilaudSantos-quotientopes, PadrolPilaudRitter}.
We define the \defn{shardoplex}~$\shardoplex$ associated to any \mbox{$\s$-arc~$\alpha$}, a polyhedral complex whose geometry is tuned to take care of the $\s$-shard of~$\alpha$.
We then use Minkowski sums of shardoplexes to construct the \defn{quotientoplex}~$\quotientoplex$ of any congruence~$\equiv$ of the $\s$-weak order, a polyhedral complex realizing the quotient of the $\s$-weak order by~$\equiv$.
See \cref{fig:mainCongruenceLattice}\,(right).
This yields our fourth battery of results.

\begin{theorem}
\label{thm:main4}
\textsf{All lattice quotients of the $\s$-weak order are realized as Minkowski sums of shardoplexes:}
\begin{enumerate}[(i)]
\item For each $\s$-arc~$\alpha$, the union of the walls of the dual polyhedral complex of the shardoplex~$\shardoplex$ contains the $\s$-shard~$\shard$ and is contained in the $\s$-shards of the subarcs of~$\alpha$ (\cref{prop:shard_trophyp}).
\item For any congruence~$\equiv$ of the $\s$-weak order, the Hasse diagram of the quotient of the $\s$-weak order by~$\equiv$ is isomorphic to the oriented skeleton of the quotientoplex~$\quotientoplex$ (\cref{prop:quotientoplex2}).
\item The quotientoplex~$\quotientoplex$ is a polytopal subdivision of a quotientope~$\quotientope[\tilde\equiv]$ for the congruence~$\tilde\equiv$ of the weak order obtained by projecting~$\equiv$ (\cref{prop:quotientoplex3}).
\end{enumerate}
\end{theorem}

We note that the proofs of \cref{thm:main4} are mainly based on tropical geometry~\cite{Joswig21} (in fact, tropical geometry is convenient even to define the notion of dual polyhedral complex).

Applying \cref{thm:main4} to the trivial congruence provides a definitive answer to the geometric conjecture of C.~Ceballos and V.~Pons~\cite[Conj.~3.1.2]{CeballosPons-sWeakOrderII} (which was already partially answered using flow polytopes and tropical geometry in~\cite{DLeonMoralesPhilippeTamayoYip}, in the situation when~$\s$ contains no~$0$).

\begin{corollary*}
For any tuple~$\s$ of non-negative integers, the Hasse diagram of the $\s$-weak order is isomorphic to the oriented skeleton of a polytopal subdivision of a graphical zonotope (\cref{coro:quotientoplexTrivial}).
\end{corollary*}

\subsection*{Plan}

\enlargethispage{.3cm}
The paper is organized as follows.
\cref{sec:insertion} describes $\s$-bushes, the insertion algorithm, and the $\s$-foam.
\cref{sec:sWeakOrder} defines the $\s$-weak order and $\s$-facial weak order and shows that they are constructible by interval doublings.
\cref{sec:canonicalComplex} introduces $\s$-arcs and non-crossing \mbox{$\s$-arc} diagrams, and describes canonical representations in the $\s$-weak order.
\cref{sec:quotients} describes the forcing order of the $\s$-weak order in terms of subarcs, and introduces a few relevant congruences of the $\s$-weak order.
\cref{sec:sQuotientopes} constructs the quotient foams~$\quotientFoam$, shardoplexes~$\shardoplex$, and quotientoplexes~$\quotientoplex$.
\cref{sec:tropicalGeometry} uses tropical geometry to prove the results of \cref{sec:sQuotientopes}.
Each section starts with a recollection of definitions and results on the weak order to be extended to the $\s$-weak~order.

\renewcommand{\thetheorem}{\arabic{theorem}}
\setcounter{theorem}{0}


\clearpage

\section{Insertion algorithm in $\s$-bushes}
\label{sec:insertion}

In this section, we describe a family of insertion algorithms whose fibers define polyhedral partitions of~$\R^n$.
We use the notations~$[n] \eqdef \{1, \dots, n\}$, $\llbracket n] \eqdef \{0, \dots, n\}$ and~${{]i,j[} \eqdef \{i+1, \dots, j-1\}}$ throughout the paper.
For a property~$P$, we denote~$\one_P \eqdef 1$ if~$P$ and~$\one_P \eqdef 0$ otherwise.


\subsection{Recollections~\ref{sec:insertion}: Insertion algorithm in increasing binary trees}
\label{subsec:recollectionsInsertion}

We first recall the classical insertion algorithm in increasing binary trees.
An \defn{increasing binary tree} is a rooted plane binary tree whose nodes are labeled by integers such that each node is strictly smaller than all its children.
It is called \defn{standard} if the nodes are bijectively labeled by~$[n]$.
The increasing binary trees are in bijection with words on integers with no repeated letter:
\begin{itemize}
\item we associate a word to any increasing binary tree by reading its node labels in inorder,
\item conversely, we associate an increasing binary tree~$\ins(w)$ to any word~$w$ inductively:
	\begin{itemize}
	\item if~$w = \varnothing$ is the empty word, then~$\ins(\varnothing)$ is the empty binary tree, with just a root, no internal node, and a single leaf,
	\item otherwise~$w = urv$ where~$r$ is the minimal letter of~$w$, and~$\ins(w)$ is the binary tree with root~$r$, left subtree~$\ins(u)$ and right subtree~$\ins(v)$.
	\end{itemize}
\end{itemize}
In particular, the standard increasing binary trees are in bijection with permutations.

\begin{figure}[b]
	\capstart
	\centerline{\includegraphics[scale=.6]{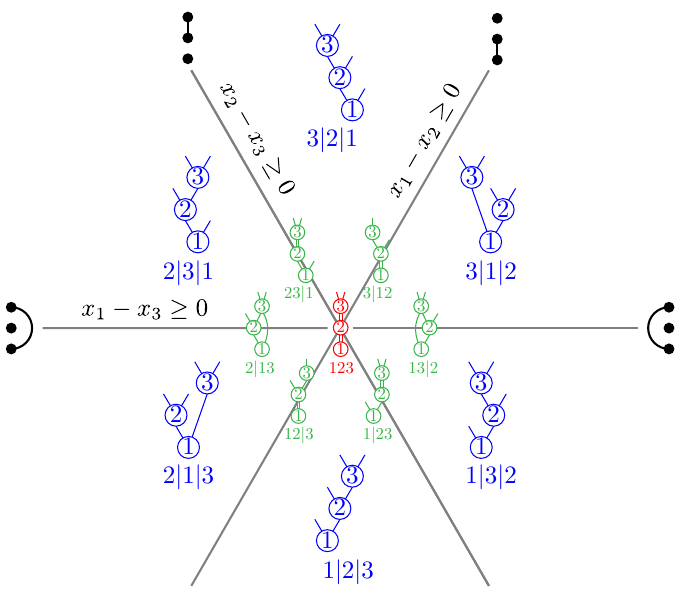}\qquad\includegraphics[scale=.6]{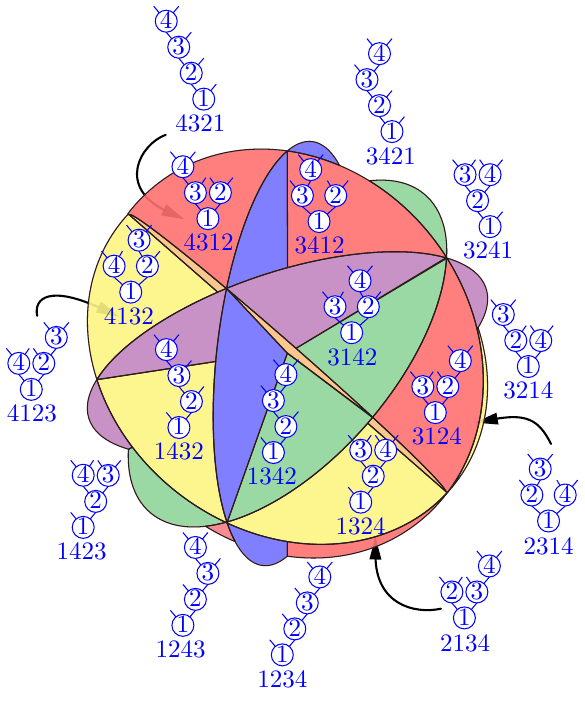}}
	\caption{The braid fan for~$n = 3$ with cones labeled by ordered set partitions and bushes (left) and for~$n = 4$ with maximal cones labeled by permutations and increasing binary trees (right).}
	\label{fig:braidFans}
\end{figure}

Consider now a point~$\b{x} \eqdef (x_1, \dots, x_n) \in \R^n$.
If~$\b{x}$ is generic (no repeated entry), we can consider the permutation~$\sigma$ of~$[n]$ that sorts~$\b{x}$, that is, such that~$x_{\sigma(1)} < \dots < x_{\sigma(n)}$.
Then we denote by~$\ins(\b{x})$ the increasing binary tree obtained by the insertion of~$\sigma$.
In other words, the root of~$\ins(\b{x})$ is~$1$, and its left subtree contains the positions~$j \in [2,n]$ such that~$x_1 > x_j$ while its right subtree contains the positions~$j \in [2,n]$ such that~$x_1 < x_j$.

The fibers of this insertion are the (open) regions of the \defn{braid arrangement}, defined by the hyperplanes~$\set{\b{x} \in \R^n}{x_i = x_j}$ for all~$1 \le i < j \le n$.
See \cref{fig:braidFans} for illustrations when~$n = 3$ and~$n = 4$.
The \defn{braid fan} is the complete simplicial fan defined by the braid arrangement.
It has a cone~$\polytope{C}_\mu$ for each ordered set partition~$\mu$ of~$[n]$, given by the points~$\b{x} \in \R^n$ such that~$x_i \le x_j$ if and only if the part of~$\mu$ containing~$i$ appears weakly before the part of~$\mu$ containing~$j$ (hence~$x_i = x_j$ if~$i$ and~$j$ belong to the same part of~$\mu$).
In particular, it has a region for each permutation of~$[n]$, and two regions are adjacent if the corresponding permutations differ by the transposition of adjacent entries (meaning at two consecutive positions).

In this paper, it will be impossible to work with permutations and ordered set partitions (except if we accept to restrict the generality to strict compositions~$\s$).
We thus describe all the combinatorics of the braid fan with treelike structures.
For this, we associate to each~$\b{x} \in \R^n$ (generic or not) a rooted plane graph~$\ins(\b{x})$ that we call a \defn{bush}.
It is constructed inductively~as~follows:
\begin{itemize}
\item start with the tree with just a root, no internal node, and a single leaf,
\item at step~$j$, attach a new node~$j$ with $2$ leaves
	\begin{itemize}
	\item either to the leaf between the two nodes~$u$ and~$v$ such that~$x_u < x_j < x_v$,
	\item or to the two leaves surrounding the node~$w$ with~$x_j = x_w$.
	\end{itemize}
\end{itemize}
This is illustrated in \cref{fig:braidFans}\,(left).
We naturally orient a bush increasingly, from its root to its leaves.
The fiber~$\fiber{\bush}$ of a bush~$\bush$ is the cone defined by the inequalities~$x_u < x_j < x_v$ and the equalities~$x_j = x_w$ discovered along the insertion (an alternative description will be discussed in \cref{subsec:sInsertionFibers}).
The cones of the braid fan are precisely the closures of the fibers of this insertion.


\subsection{$\s$-bushes, $\s$-trees, and $\s$-trunks}
\label{subsec:bushes}

Fix an $n$-tuple~$\s \eqdef (s_1, \dots, s_n)$ with~$s_i \in \N$ for any~$i \in [n]$ (note that we allow~${s_i = 0}$).
For any~${j \in \llbracket n]}$, we define the $j$-tuple~$\s_{\le j} \eqdef (s_1, \dots, s_j)$ and the numbers~$S_j \eqdef 1 + \sum_{i \le j} s_i$ and $T_j \eqdef 2 - \one_{s_1 = \dots = s_j = 0} +\sum_{i \le j}\max(0, s_i - 1)$.

We now define a family of plane rooted labeled graphs that naturally appear in the $\s$-insertion algorithm of \cref{subsec:insertion}.
See \cref{fig:algo,fig:foams} for illustrations.

\begin{definition}
\label{def:bush}
An \defn{$\s$-bush} is a plane graph with a root, $n$ internal nodes bijectively labeled by~$[n]$, and some leaves, defined inductively as follows:
\begin{itemize}
\item start with the rooted graph with just a root, no internal node, and a single leaf,
\item at step~$j$, attach either to a leaf or to two consecutive leaves, a new node~$j$ with~$s_j + 1$ leaves (except that if~$s_j = 0$ and $j$ is attached to two leaves, then $j$ gets itself two leaves).
\end{itemize}
\end{definition}

\begin{remark}
\label{rem:bush}
A few observations on \cref{def:bush}:
\begin{enumerate}
\item For~$j \in \llbracket n]$, deleting all nodes~$>j$ in an $\s$-bush~$\bush$ gives an $\s_{\le j}$-bush~$\bush_{\le j}$.

\item An $\s$-bush is naturally oriented from its root to its leaves, so that all its edges are increasing.
We draw bushes growing up, with their roots on the bottom and their leaves on top, and such that the node~$j$ is at level~$j$.
See \cref{fig:algo,fig:foams}.

\item For an $\s$-bush~$\bush$ and~$1 \le i < j \le n$, we say that~$i$ is an \defn{ancestor} of~$j$ in~$\bush$ and that $j$ is a \defn{descendant} of~$i$ in~$\bush$ if there is an increasing path from~$i$ to~$j$ in~$\bush$ (we consider~$i$ to be an ancestor and a descendant of itself).

\item The \defn{rank}~$r(\bush)$ of an $\s$-bush is its number of indegree~$1$ nodes. An \defn{$\s$-tree} (resp.~\mbox{\defn{$\s$-trunk}}) is an $\s$-bush of maximal rank~$n$ (resp.~minimal rank~$\min \set{i \in [n]}{s_i \ne 0} \cup \{n\}$).
Note that the $\s$-trees are precisely the rooted plane trees with internal nodes bijectively labeled by~$[n]$ such that the node~$j$ has~$s_j+1$ children that are either leaves or nodes larger than~$j$.

\item An $\s$-bush where~$X \subseteq [n]$ is the set of indegree~$2$ nodes has~$S_n - \#\set{x \in X}{s_x \ne 0}$~leaves.
There are~$\prod_{j \in [n]} \big( S_{j-1} - \#\set{x \in X}{x < j \text{ and } s_x \ne 0} - \one_{j \in X} \big)$ such~$\s$-bushes.
Hence,
\begin{itemize}
\item an $\s$-tree has~$S_n$ leaves, and there are~$\prod_{j \in [n]} S_{j-1}$ $\s$-trees,
\item an $\s$-trunk has~$T_n$ leaves, and there are~$\prod_{j \in [n]} \max(1, T_{j-1}-1)$ $\s$-trunks.
\end{itemize}

\item 
\label{item:parametrization}
In fact, denoting by
\[
\sTrees \eqdef \prod_{j \in [n]} [S_{j-1}]
\qquad\text{and}\qquad
\sTrunks \eqdef \prod_{j \in [n]} [\max(1, T_{j-1}-1)],
\]
we can associate to each~$\b{p} \in \sTrees$ the $\s$-tree~$\tree[S]_\b{p}$ obtained by attaching node~$j$ to leaf~$p_j$, and to each~$\b{q} \in \sTrunks$ the $\s$-trunk~$\tree[T]_\b{q}$ obtained by attaching node~$j$ to leaves~$q_j$ and~$q_j+1$ (or to the only leaf if there is only one).
\end{enumerate}
\end{remark}

\begin{remark}
\label{rem:sTreeVSsDecreasingTrees}
Our $\s$-trees are essentially the $\s$-decreasing trees of~\cite{CeballosPons-sWeakOrderI}.
We have chosen to slightly change their conventions to simplify our presentation, in particular to allow for more natural inductive arguments.
To change conventions, one just needs to reverse~$\s$ and relabel each node~$j$ by~$n-j$.
We will see that our $\s$-bushes provide alternative combinatorial models to the pure intervals of~\cite{CeballosPons-sWeakOrderII}, much more adapted to geometry.
Finally, we note that our insertion algorithm, while never made explicit in~\cite{CeballosPons-sWeakOrderI,CeballosPons-sWeakOrderII}, is somewhat underlying in their work~\cite{Pons-personnalCommunication}.
\end{remark}


\subsection{$\s$-insertion algorithm}
\label{subsec:insertion}

We now describe an iterative insertion algorithm that sends a point~$\b{x}$ in~$\R^n$ to an $\s$-bush~$\ins(\s, \b{x})$.
See \cref{fig:algo}.
To run the algorithm, we actually maintain, in each gap between two consecutive leaves of our bush, a label of the form~$(u,\rho)$, with~$u \in [n]$ and~$\rho \in \N$.
We also include the label~$(0,0)$ before the first leaf and the label~$(n+1,0)$ after the last leaf.

\begin{definition}
\label{def:insertion}
For each~$\b{x} \in \R^n$, we construct an $\s$-bush~$\ins(\s, \b{x})$ inductively as follows:
\begin{itemize}
\item start with the rooted graph with just a root, no internal node, and a single leaf,
\item at step~$j$,
	\begin{itemize}
	\item attach a new node~$j$ either to the leaf between two gap labels~$(u,\rho)$ and~$(v,\sigma)$ such that~$x_u-\rho < x_j < x_v-\sigma$, or to the two leaves around a gap label~$(w,\tau)$ such that~$x_j = x_w-\tau$, (with the convention that~$x_0 = -\infty$ and~$x_{n+1} = +\infty$),
	\item attach $s_j + 1$ leaves to the node~$j$, with gaps labeled by~$(j, s_j-1), \dots, (j,1), (j,0)$ (except, if~$s_j = 0$ and~$j$ has indegree~$2$, then we attach~$2$ leaves with gap label~$(j,0)$),
	\item add~$\max(0, s_j-1)$ to the second entry of all gap labels on the left~of~$j$.
	\end{itemize}
\end{itemize}
\end{definition}

\begin{figure}
	\capstart
	\centerline{\includegraphics[scale=.9]{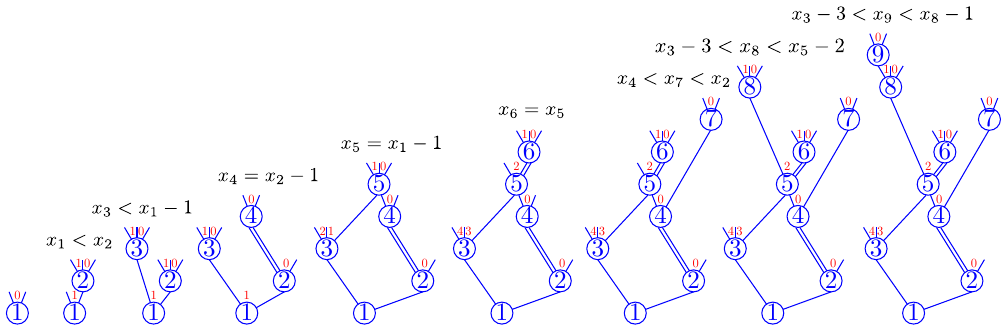}}
	\caption{Flow of the $\s$-insertion algorithm to compute~$\ins(\s, \b{x})$ for $\s = (1, 2, 2, 0, 2, 2, 1, 2, 1)$ and~$\b{x} = (5, 6, 3, 5, 4, 4, 5.5, 1.5, .25)$. We indicate each gap label~$(u,\rho)$ by writing~$\rho$ in red above the node~$u$. The equalities or inequalities that led to this $\s$-tree are indicated at each step of the algorithm.}
	\label{fig:algo}
\end{figure}

This algorithm is illustrated in \cref{fig:algo}.

\begin{remark}
A few observations on \cref{def:insertion}:
\begin{enumerate}
\item The crucial invariant of the algorithm is that the values~$x_u-\rho$ for the gap labels~$(u, \rho)$ are always strictly increasing from left to right. This invariant is maintained by the shifts performed on gap labels at each step. It implies that the attaching step is well-defined.

\item The gap labels only depend on the $\s$-bush~$\ins(\s, \b{x})$, not on the exact values of~$\b{x}$.

\item For~$j \in \llbracket n]$, deleting all nodes~$>j$ in~$\ins(\s, \b{x})$ gives~$\ins(\s_{\le j}, \b{x}_{\le j})$.

\item If~$\b{x}$ is generic enough (\eg if~$x_i-x_j \notin \N$ for~$1 \le i < j \le n$), then~$\ins(\s, \b{x})$ is an $\s$-tree.

\item When~$\s = 1^n$ and~$\b{x}$ is generic, this is the classical insertion in an increasing binary tree.
\end{enumerate}
\end{remark}


\subsection{$\s$-insertion fibers}
\label{subsec:sInsertionFibers}

We now consider the fibers of the $\s$-insertion algorithm of \cref{subsec:insertion}.

\begin{definition}
For an $\s$-bush~$\bush$, the \defn{fiber} of~$\bush$ is~$\fiber{\bush} \eqdef \set{\b{x} \in \R^n}{\ins(\s, \b{x}) = \bush}$, and the \defn{closed fiber} of~$\bush$ is the closure~$\closedFiber{\bush}$ of~$\fiber{\bush}$.
\end{definition}

For instance, the fiber~$\fiber{\bush}$ of the $\s$-bush on the right of \cref{fig:algo} is the set of points~$\b{x} \in \R^9$ that satisfy the indicated equalities and inequalities, and the closed fiber~$\closedFiber{\bush}$ is obtained by replacing the strict inequalities by large inequalities.
We thus get the following observation.

\begin{lemma}
\label{lem:fibersPolyhedra}
The closed fiber~$\closedFiber{\bush}$ of any $\s$-bush~$\bush$ is a polyhedron.
\end{lemma}

\begin{proof}
The fiber~$\fiber{\bush}$ is a (relatively open) polyhedron, defined by the inequalities~$x_u-\rho \! < \! x_j \! < \! x_v-\sigma$ and the equalities~$x_j \! = \! x_w-\tau$ from \cref{def:insertion}.
Hence, the closed fiber~$\closedFiber{\bush}$ is a polyhedron.
\end{proof}

We now aim at an irredundant equality and inequality description of the closed fiber~$\closedFiber{\bush}$ of an $\s$-bush~$\bush$.
It is convenient to introduce the following definitions, illustrated in \cref{exm:holesAscentsDescents}.

\begin{definition}
A \defn{hole} of an $\s$-bush~$\bush$ is a pair~$(i,j)$ of nodes such that $j$ has two incoming edges whose greatest common ancestor is~$i$.
\end{definition}

\begin{definition}
The \defn{left} (resp.~\defn{right}) \defn{zigzag} of a node~$j$ of an $\s$-bush~$\bush$ is the increasing path~$Z$ that follows the leftmost (resp.~rightmost) increasing path of~$\bush$ starting at~$j$ until it reaches either a leaf, or a node~$k$ with two distinct parents, in which case it continues with the right (resp.~left) zigzag of~$k$.
The nodes where~$Z$ exits through the leftmost (resp.~rightmost)  edge are the \defn{zig} (resp.~\defn{zag}) nodes~of~$Z$.
\end{definition}

\begin{definition}
\label{def:ascentDescentBush}
An \defn{ascent} (resp.~\defn{descent}) of an $\s$-bush~$\bush$ is a pair~$(i,j)$ of nodes of~$\bush$ such that 
\begin{itemize}
\item $j$ has indegree~$1$,
\item $i$ is the greatest ancestor of~$j$ such that the leftmost (resp.~rightmost) increasing path from~$i$ to~$j$ in~$\bush$ takes the leftmost (resp.~rightmost) outgoing edge at each node, except~at~node~$i$,
\item either~$s_j = 0$ or all nodes along the left (resp.~right) zigzag of~$j$ have indegree~$2$.
\end{itemize}
\end{definition}

\begin{remark}
\label{rem:ascentDescentTree}
For $\s$-trees, our definition of ascents and descents specializes to~\cite[Def.~1.24]{CeballosPons-sWeakOrderI}.
Namely, $(i,j)$ is an ascent (resp.~descent) of an $\s$-tree~$\tree$ if
\begin{itemize}
\item $i$ is an ancestor of~$j$ and the increasing path from~$i$ to~$j$ in~$\tree$ takes the leftmost (resp.~rightmost) outgoing edge at each node, except~at~node~$i$,
\item either~$s_j = 0$ or the leftmost (resp.~rightmost) edge of~$j$ is a leaf.
\end{itemize}
\end{remark}

\begin{definition}
\label{def:rhs}
Consider an $\s$-bush~$\bush$ and~$1 \le i < j \le n$ such that~$i$ is an ancestor of~$j$.
Let~$\pi$ be the leftmost (resp.~rightmost) increasing path in~$\bush$ from~$i$ to~$j$, and arriving through the right (resp.~left) incoming edge of~$j$ if~$(i,j)$ is a hole of~$\bush$.
We define~\defn{$\mu(\bush, i, j)$} (resp.~\defn{$\nu(\bush, i, j)$}) as $r - 1 + \sum_{k \in K} \max(0, s_k-1)$ where~$r$ is the number of children of~$i$ weakly (resp.~strictly) to the right of~$\pi$, and $K$ is the set of nodes~$i < k < j$ and weakly (resp.~strictly) to the right of~$\pi$.
Note that~$\mu(\bush, i, j) \le \nu(\bush, i, j)$, and~$\mu(\bush, i, j) = \nu(\bush, i, j)$ if~$(i,j)$ is a hole of~$\bush$.
\end{definition}

\begin{example}
\label{exm:holesAscentsDescents}
On the rightmost $\s$-bush~$\bush$ of \cref{fig:algo},
\begin{itemize}
\item the holes are $(2,4)$, $(1,5)$, and $(5,6)$,
\item the left zigzag of~$2$ is~$2 \to 4 \to 5 \to 6$ and the right zigzag of~$3$ is~$3 \to 5 \to 8 \to 9$,
\item the ascents are~$(1,2)$, $(4,7)$, $(3,9)$ and the descents are~$(2,7)$, $(5,8)$ and~$(8,9)$,
\item $\mu(\bush, 1, 5) = \nu(\bush, 1, 5) = 1$, 
$\mu(\bush, 3, 9) = 3$, $\nu(\bush, 2, 7) = 0$,
\item the closed fiber~$\closedFiber{\bush}$ is the polyhedron defined by the equalities~$x_2 - x_4 = 1$, $x_1 - x_5 = 1 $, $x_5 - x_6 = 0$ and the inequalities~$x_1 - x_2 \le 0$, $x_4 - x_7 \le 0$, $x_3 - x_9 \le 3$, $x_2 - x_7 \ge 0$, $x_5 - x_8 \ge 2$, and $x_8 - x_9 \ge 1$.
\end{itemize}
\end{example}

\begin{proposition}
\label{prop:fibers}
The closed fiber~$\closedFiber{\bush}$ of an $\s$-bush~$\bush$ is irredundantly described by
\begin{itemize}
\item an equality~$x_i - x_j = \mu(\bush, i, j)$ for each hole~$(i,j)$ of~$\bush$,
\item an inequality $x_i - x_j \le \mu(\bush, i, j)$ for each ascent~$(i,j)$ of~$\bush$,
\item an inequality $x_i - x_j \ge \nu(\bush, i, j)$ for each descent~$(i,j)$ of~$\bush$.
\end{itemize}
\end{proposition}

\begin{proof}
Fix~$j \in [n]$ and consider the $\s_{\le j-1}$-bush~$\bush_{\le j-1}$ with gaps labeled according to the insertion~$\ins(\s_{\le j-1}, \b{x}_{\le j-1})$ of any~$\b{x} \in \fiber{\bush}$.
Observe (by induction on~$j$) that, if~$(i, \sigma)$ is the gap label immediately to the left (resp.~right) of a leaf~$\ell$ of~$\bush_{\le j-1}$, then
\begin{enumerate}[(a)]
\item The leftmost (resp.~rightmost) increasing path from~$i$ to~$\ell$ takes the leftmost (resp.~rightmost) outgoing edge at each node, except~at~$i$. We say that $i$ is the \defn{left} (resp.~\defn{right}) \defn{ancestor} of~$\ell$.
\item $\sigma = r - 1 + \sum_{k \in K} \max(0, s_k-1)$ where $r$ is the number of children of~$i$ and~$K$ the set of nodes~${i < k < j}$, all located weakly (resp.~strictly) to the right of the leftmost (resp.~rightmost) increasing path in~$\bush$ from~$i$ to~$\ell$.
\end{enumerate}

Assume now that $j$ is attached to the two leaves~$\ell, \ell'$ around a gap label~$(w, \tau)$ of~$\bush_{\le j-1}$.
Then $w$ is the greatest common ancestor of~$\ell$ and~$\ell'$ in~$\bush$ (by~(a)), and~$\tau = \mu(\bush, w, j) = \nu(\bush, w, j)$ (by~(b)).
Hence, the corresponding equation~$x_j = x_w - \tau$ of~$\closedFiber{\bush}$ is indeed of the form~$x_i - x_j = \mu(\bush, i, j)$ for a hole~$(i,j)$ of~$\bush$.

Assume now that~$j$ is attached to a leaf~$\ell$ between two gap labels~$(u, \rho)$ and~$(v, \sigma)$~of~$\bush_{\le j-1}$.
Then~$u$ is the left ancestor of~$j$ (by~(a)), and~$\rho = \mu(\bush, u, j)$ (by~(b)).
Moreover, if~$s_j \ne 0$ and~$h$ lies on the left zigzag~$Z$ of~$j$, then~$x_u - \mu(\bush, u, j) - \zeta \le x_h \le x_j - \zeta$, with~$\zeta = \sum_{k \in K} \max(0, s_k-1)$ where~$K$ is the set of nodes~$j \le k < h$ on the right of~$Z$, including the zig nodes of~$Z$, but not the zag nodes of~$Z$.
If the node~$h$ has indegree~$1$, then both inequalities are strict on~$\fiber{\bush}$, so that the inequality~$x_u - \mu(\bush, u, j) \le x_j$ is actually redundant.
The same arguments hold for the inequality~$x_j \le x_v - \nu(\bush, v, j)$.
We conclude that all facet defining inequalities of~$\closedFiber{\bush}$ are of the form~$x_i - x_j \le \mu(\bush, i, j)$ for an ascent~$(i,j)$ of~$\bush$, or $x_i - x_j \ge \nu(\bush, i, j)$ for a descent~$(i,j)$ of~$\bush$.

Finally, we need to prove that our system of equalities and inequalities is irredundant.
The equalities are irredundant because they are echeloned.
For the inequalities, we choose for instance an ascent~$(i,j)$ of~$\bush$, and we prove that there is~$\b{x} \in \closedFiber{\bush}$ for which the inequality~$x_i - x_j \le \mu(\bush, i, j)$ is the only tight inequality in our system.
For this, start with~$\b{y}$ in the interior of~$\fiber{\bush}$ (so that all our inequalities are strict), and let~$\delta = y_i - y_j - \mu(\bush, i, j)$.
Let~$Z$ be the left zigzag of~$j$, and $J$ be the set of zig nodes of~$Z$.
Consider the point~$\b{x} \eqdef \b{y} + \delta \sum_{j \in J} \b{e}_j$.
As~$i \notin J$ while~$j \in J$, we have~$x_i - x_j = y_i - y_j - \delta = \mu(\bush, i, j)$.
Finally, we let the reader check that~$\b{x}$ satisfies the equalities given by all holes of~$\bush$ and strictly satisfies all inequalities given by the other ascents and descents~of~$\bush$.
\end{proof}

\begin{corollary}
The dimension of~$\closedFiber{\bush}$ is the rank of~$\bush$.
\end{corollary}

\begin{proof}
By \cref{prop:fibers}, the codimension of~$\closedFiber{\bush}$ is the number of indegree~$2$ nodes of~$\bush$.
Hence the dimension is the number of indegree~$1$ nodes of~$\bush$, thus the rank of~$\bush$.
\end{proof}

\begin{example}
\label{exm:grid}
Recall from \cref{rem:bush}\,\eqref{item:parametrization} that for~$\b{q} \in \sTrunks$, we denote by~$\trunk_\b{q}$ the $\s$-trunk obtained by attaching node~$i$ to leaves~$q_i$ and~$q_i+1$ (or to the only leaf if there is only one).
Its insertion fiber~$\fiber{\trunk_\b{q}}$ is the affine subspace of dimension~$n+T_n-S_n$, defined by the equations~$x_1 - x_j = \mu(\trunk_\b{q}, 1, j) = q_j-1$ for all~$j \in [n]$ such that~$T_{j-1} \ge 2$.
Hence, the insertion fibers of the $\s$-trunks form a grid in the basis~$x_1 - x_j$.
This is visible in \cref{fig:foams}.
\end{example}


\subsection{$\s$-foam}
\label{subsec:sFoam}

We now consider the collection of all $\s$-insertion fibers.

\begin{definition}
The \defn{$\s$-foam} is the collection~$\sFoam$ of closed fibers~$\closedFiber{\bush}$ of all~$\s$-bushes~$\bush$.
\end{definition}

Some $\s$-foams are illustrated in \cref{fig:foams}.
As~$\ins(\s, \b{x} + \lambda \sum_{i \in [n]} \b{e}_i) = \ins(\s, \b{x})$ for any~$\lambda \in \R$, all fibers contain the line~$\R \sum_{i \in [n]} \b{e}_i$ in their lineality space, hence we intersect the $\s$-foams with the hyperplane~$\bigset{\b{x} \in \R^n}{\sum_{i \in [n]} x_i = 0}$ in all illustrations.
We need the following property of the $\s$-foam.

\begin{figure}[t]
	\capstart
	\centerline{\includegraphics[scale=.6]{120-foam} \hspace{.5cm} \includegraphics[scale=.6]{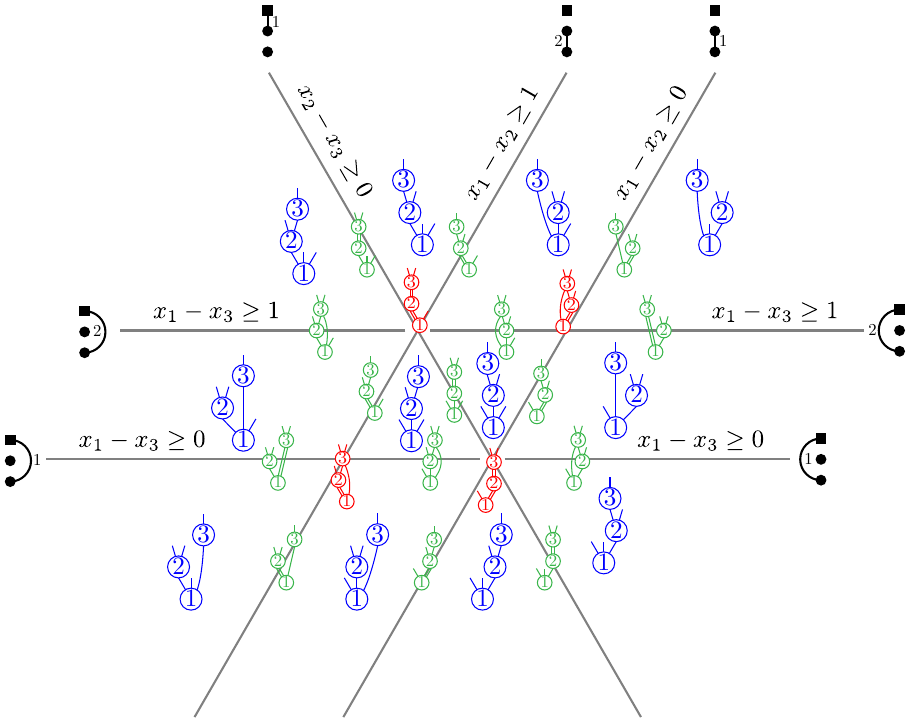}}
	\caption{The $\s$-foam for $\s = (1,2,0)$ (left) and~$\s = (2,1,0)$ (right).}
	\label{fig:foams}
\end{figure}

\begin{proposition}
\label{prop:fiberFaces}
Any face of a closed fiber~$\closedFiber{\bush}$ is a closed fiber~$\closedFiber{\bush'}$.
\end{proposition}

\begin{proof}
Consider a facet~$\polytope{G}$ of a closed fiber~$\closedFiber{\bush}$.
We will prove that~$\polytope{G} = \closedFiber{\bush'}$ where~$\bush' \eqdef \ins(\s, \b{g})$ for an arbitrary~$\b{g}$ in the relative interior of~$\polytope{G}$.
This shows the statement by induction on the dimension (since any face of~$\closedFiber{\bush}$ is a face of a facet of~$\closedFiber{\bush}$).

\medskip
Let~$\polytope{H}$ be the affine span of~$\closedFiber{\bush}$, let~$\polytope{H}'$ be the affine span of~$\polytope{G}$, and let~$\polytope{H}^+$ be the half-space of~$\polytope{H}$ defined by~$\polytope{H}'$ and containing~$\closedFiber{\bush}$ (this is well-defined since~$\polytope{G}$ is a facet of~$\closedFiber{\bush}$).
Pick a generic~$\b{x}' \in \polytope{H}'$ and then some~$\b{x} \in \polytope{H}^+$ close enough to~$\b{x}'$ so that $x'_i < x'_j - \sigma$ implies~$x_i < x_j - \sigma$ for any~$i,j \in [n]$ and~$0 \le \sigma \le S_n$ (this is possible as we have finitely many open conditions).

For~$j \in [n]$, consider the $\s_{\le j}$-bushes~$\bush_j \eqdef \ins(\s_{\le j}, \b{x}_{\le j})$ and~$\bush'_j \eqdef \ins(\s_{\le j}, \b{x}'_{\le j})$.
Let~$i$ be the minimal index such that~$\bush_i \ne \bush'_i$.
Hence, there is a gap label~$(u,\rho)$ in~$\bush_{i-1} = \bush'_{i-1}$ such that~${x_i \ne x_u - \rho}$ while~$x'_i = x'_u - \rho$.
Assume for instance~$x_i < x_u - \rho$ (the other case is similar).

We claim that we can reconstruct the sequence~$\bush_1, \dots, \bush_n$ from the sequence~$\bush'_1, \dots, \bush'_n$ and \viceversa.
To see it, we prove by induction that, for any~$j > i$,
\begin{enumerate}[(a)]
\item the gap labels in~$\bush_j$ are precisely the gap labels in~$\bush'_j$, except that one gap label of the form~$(u,\rho')$ in~$\bush_j$ may disappear in~$\bush'_j$,
\item if a gap label~$(u, \rho')$ of~$\bush_j$ does not appear in~$\bush'_j$, then the gap label~$(v, \sigma)$ immediately to the left of~$(u, \rho')$ in~$\bush_j$ is such that~$x'_v - \sigma = x'_u - \rho'$,
\item the position of~$j+1$ in~$\bush_{j+1}$ determines the position of~$j+1$ in~$\bush'_{j+1}$ and \viceversa.
\end{enumerate}
Note that~(a) and~(b) hold for~$j = i$.
Indeed, the gap~$(u,\rho)$ of~$\bush_i$ has disappeared in~$\bush'_i$, the gap immediately to its left is~$(i, 0)$, and we have~$x'_i - 0 = x'_u - \rho$.
Assume now that~(a) and~(b) hold for some~$j > i$.
For any gap label~$(w,\tau)$ common to~$\bush_j$ and~$\bush'_j$, we have
\begin{itemize}
\item $x_{j+1} = x_w-\tau \iff x'_{j+1} = x'_w-\tau$. Indeed, let~$\polytope{K}$ be the hyperplane of equation~${y_{j+1} = y_w-\tau}$. If~$x_{j+1} = x_w-\tau$, then~$\polytope{H} \subseteq \polytope{K}$, so that~$x'_{j+1} = x'_w-\tau$. Conversely, if~$x'_{j+1} = x'_w-\tau$, then~$\polytope{H}' \subseteq \polytope{K}$, so that~$\polytope{H} \subseteq \polytope{K}$ (otherwise, $\polytope{H}'$ would have codimension at least~$2$ in~$\polytope{H}$, since~$\b{x}'$ is generic in~$\polytope{H}'$),  hence~$x_{j+1} = x_w-\tau$.
\item $x_{j+1} < x_w-\tau \iff x'_{j+1} < x'_w-\tau$ by our assumption that~$\b{x}$ is close enough~to~$\b{x}'$.
\end{itemize}
Assume now that some gap label~$(u,\rho')$ in~$\bush_j$ disappeared in~$\bush'_j$, and let~$(v, \sigma)$ be the gap label immediately to its left in~$\bush_j$, so that~$x'_v - \sigma = x'_u - \rho'$ by~(b).
Assume that~$x_v - \sigma \le x_{j+1} \le x_u - \rho'$.
As~$\b{x}$ is close enough to~$\b{x}'$, we also have~$x'_v - \sigma \le x'_{j+1} \le x'_u - \rho'$.
As~$x'_v - \sigma = x'_u - \rho'$, we obtain that~$x'_{j+1} = x'_v - \sigma$.
Since~$(v, \sigma)$ is a common gap label in~$\bush_j$ and~$\bush'_j$, this implies that~$x_v - \sigma = x_u - \rho'$.
We thus obtain that~$j+1$ cannot be attached only to the leaf between the gap labels~$(v, \sigma)$ and~$(u, \rho')$, nor to the two leaves around the gap label~$(u,\rho')$.
We conclude that the position of~$j+1$ in~$\bush_{j+1}$ determines the position of~$j+1$ in~$\bush'_{j+1}$ and \viceversa, so that~(c) holds for~$j$.
This in turn implies that~(a) and~(b) hold for~$j+1$ since we shift a gap label common to~$\bush_j$ and~$\bush'_j$ by the same quantity to obtain the corresponding gap label common to~$\bush_{j+1}$ and~$\bush'_{j+1}$.

\medskip
We now consider the facet~$\polytope{G}$ of~$\closedFiber{\bush}$.
We pick an arbitrary~$\b{g}$ in the relative interior of~$\polytope{G}$, and choose~$\b{f}$ close enough to~$\b{g}$ such that~$\b{f} \in \fiber{\bush}$.
We prove that~$\polytope{G} = \closedFiber{\bush'}$ where~$\bush' \eqdef \ins(\s, \b{g})$.

Observe first that, for any~$\b{x}'$ in the interior of~$\polytope{G}$, we can choose~$\b{x}$ close enough to~$\b{x}'$ such that~$\b{x} \in \fiber{\bush}$.
As we proved that~$\ins(\s, \b{x}')$ is determined by~$\ins(\s, \b{x}) = \bush = \ins(\s, \b{f})$, we obtain that~$\ins(\s, \b{x}') = \ins(\s, \b{g}) = \bush'$.
Hence, the interior of~$\polytope{G}$ is contained in~$\fiber{\bush'}$.

Conversely, consider any~$\b{x}'$ in~$\fiber{\bush'}$, and let~$\b{x}$ in~$\polytope{H}^+$ be close enough to~$\b{x}'$.
Since~$\ins(\s, \b{x})$ is determined by~$\ins(\s, \b{x}') = \bush' = \ins(\s, \b{g})$, we obtain that~$\ins(\s, \b{x}) = \ins(\s, \b{f}) = \bush$.
Hence, $\b{x}'$ lies on the boundary of~$\closedFiber{\bush}$.
We conclude that the interior of~$\fiber{\bush'}$ is contained in~$\polytope{G}$.
\end{proof}

\begin{proposition}
\label{prop:sFoam}
The $\s$-foam~$\sFoam$ is a \defn{complete polyhedral complex} (\ie a set of polyhedra closed under faces, pairwise intersecting along faces, and whose union completely covers~$\R^n$).
\end{proposition}

\begin{proof}
We already observed that the closed fibers are polyhedra (\cref{lem:fibersPolyhedra}), and that their faces are closed fibers (\cref{prop:fiberFaces}).
The fibers (and thus their closures) cover~$\R^n$ as the insertion algorithm is defined on~$\R^n$.
We thus just need to prove that the intersection~$\polytope{X}$ of two closed fibers~$\closedFiber{\bush}$ and~$\closedFiber{\bush'}$ is a face of both.
For this, consider~$\b{x} \in \polytope{X}$, and let~$\bush'' \eqdef \ins(\s, \b{x})$.
Since~$\b{x} \in \closedFiber{\bush}$, \cref{prop:fiberFaces} implies that $\closedFiber{\bush''}$ is a face of~$\closedFiber{\bush}$ (possibly~$\closedFiber{\bush}$ itself).
We thus obtain that~$\closedFiber{\bush''}$ is contained in~$\polytope{X}$.
Hence, $\polytope{X}$ is a union of faces of~$\closedFiber{\bush}$.
As it is convex, it is actually a face of~$\closedFiber{\bush}$ (possibly~$\closedFiber{\bush}$ itself).
By symmetry, it is also a face of~$\closedFiber{\bush'}$ (possibly~$\closedFiber{\bush'}$ itself).
\end{proof}


\subsection{$\s$-rotations}
\label{subsec:sRotations}

We now describe the dual graph of the $\s$-foam~$\sFoam$.
For this, we describe two operations, which are obviously inverse to each other, as illustrated in \cref{fig:stitchingIncisionRotations}\,(left).

\begin{figure}[t]
	\capstart
	\centerline{\includegraphics[scale=1]{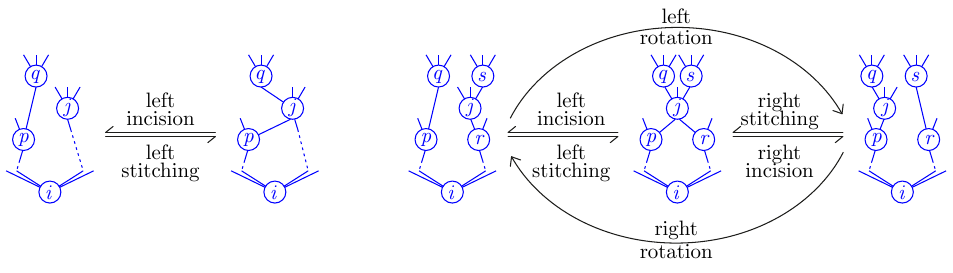}}
	\caption{The left stitching and incision of \cref{def:stitching,def:incision} (left) and the left and right rotations of \cref{def:rotation}.}
	\label{fig:stitchingIncisionRotations}
\end{figure}

\begin{definition}
\label{def:stitching}
Let~$(i,j)$ be an ascent (resp.~descent) in an $\s$-tree~$\tree$.
Let~$p \to q$ be the last edge whose source is smaller than~$j$ along the rightmost (resp.~leftmost) increasing path leaving~$i$ through the edge immediately to the left (resp.~right) of the path from~$i$ to~$j$ (note that~$q$ might be a leaf).
The \defn{stitching} of~$(i,j)$ in~$\tree$ is the $\s$-bush~$\bush$ obtained from~$\tree$ by replacing the edge~$p \to q$ by~$p \to j$ and the leftmost (resp.~rightmost) outgoing edge of~$j$ (which is a leaf) by~$j \to q$.
\end{definition}

\begin{definition}
\label{def:incision}
Consider an $\s$-bush~$\bush$ with a single hole~$(i,j)$.
Let~$p \to j$ and~$j \to q$ denote the leftmost (resp.~rightmost) incoming and outgoing edges at~$j$ (note that~$q$ can be a leaf).
The \defn{left} (resp.~\defn{right}) \defn{incision} of~$\bush$ is the $\s$-tree~$\tree$ obtained from~$\bush$ by replacing the edge~$p \to j$ by~$p \to q$ and the edge~$j \to q$ by a leaf at~$j$.
\end{definition}

\begin{proposition}
\label{prop:stitchingIncision}
If~$(i,j)$ is an ascent (resp.~descent) in an $\s$-tree~$\tree$, and~$\bush$ is the stitching of~$(i,j)$ in~$\tree$, then~$\closedFiber{\bush}$ is the facet of~$\closedFiber{\tree}$ corresponding to~$(i,j)$.
If an $\s$-bush~$\bush$ has a single hole, and~$\tree$ and~$\tree'$ are the left and right incisions of~$\bush$, then~$\closedFiber{\tree}$ and~$\closedFiber{\tree'}$ are the two facets of~$\sFoam$ containing~$\closedFiber{\bush}$.
\end{proposition}

\begin{proof}
If~$\bush$ is the stitching of an ascent (resp.~descent) $(i,j)$ in an $\s$-tree~$\tree$, then~$(i,j)$ is the only hole of~$\bush$, so that~$\closedFiber{\bush}$ satisfies the single equation~$x_i - x_j = \mu(\bush, i, j) = \nu(\bush, i, j)$, and we have~$\mu(\bush, i, j) = \mu(\tree, i, j)$ (resp.~$\nu(\bush, i, j) = \nu(\tree, i, j)$).
Moreover, as in the proof of \cref{prop:fiberFaces}, an immediate induction shows that, for any~$k \ge j$, the gap labels in~$\tree_{\le k}$ and~$\bush_{\le k}$ are identical, except that a gap label~$(i, \rho)$ of~$\tree$ disappears in~$\bush$.
Hence, we obtain that~$\closedFiber{\bush}$ satisfies the same inequalities as~$\closedFiber{\tree}$.
We conclude that~$\closedFiber{\bush}$ is the facet of~$\closedFiber{\tree}$ defined by~$x_i - x_j = \mu(\bush, i, j)$.

Conversely, if an $\s$-bush~$\bush$ has a single hole, then~$\closedFiber{\bush}$ is a codimension~$1$ polyhedron of the \linebreak $\s$-foam~$\sFoam$.
As~$\sFoam$ is a complete polyhedral complex, $\closedFiber{\bush}$ is contained in precisely two full-dimensional polyhedra of~$\sFoam$.
Since the stitching and incision operations are clearly inverse to each other, the first part of the statement implies that~$\closedFiber{\tree}$ and~$\closedFiber{\tree'}$ are the two facets of~$\sFoam$ containing~$\closedFiber{\bush}$.
\end{proof}

We now connect incisions and stitchings with the following operations, adapted from the description of~\cite{CeballosPons-sWeakOrderI}, and illustrated in \cref{fig:stitchingIncisionRotations}\,(right).

\begin{definition}[{\cite[Def.~1.30]{CeballosPons-sWeakOrderI}}]
\label{def:rotation}
Consider an ascent (resp.~descent)~$(i,j)$ in an $\s$-tree~$\tree$.
Let~$r$ be the parent of~$j$ (note that~$r$ might be~$i$).
Let~$j \to s$ be the rightmost (resp.~leftmost) outgoing edge of~$j$ (note that~$s$ might be a leaf).
Let~$p \to q$ be last edge whose source is smaller than~$j$ along the rightmost (resp.~leftmost) increasing path leaving~$i$ through the edge immediately to the left (resp.~right) of the path from~$i$ to~$j$ (note that~$q$ might be a leaf).
The \defn{left} (resp.~\defn{right}) \defn{rotation} of~$(i,j)$ transforms~$\tree$ by replacing the edges~$r \to j$, $j \to s$ and $p \to q$ by new edges~$r \to s$, $j \to q$ and~$p \to j$ respectively.
\end{definition}

\begin{lemma}
A left (resp.~right) rotation is the composition of an ascent (resp.~descent) stitching of \cref{def:stitching} with a right (resp.~left) incision of \cref{def:incision}.
\end{lemma}

\begin{corollary}
\label{coro:dualGraph}
The dual graph of the $\s$-foam is the rotation graph on $\s$-trees.
Its incidence graph is the incision (or equivalently stitching) graph.
\end{corollary}

\begin{remark}
\label{rem:coverRelationsRefinementBushes}
More generally, given an $\s$-bush~$\bush$, the combinatorial description of the $\s$-bushes~$\bush'$ such that~$\fiber{\bush'}$ is a face of~$\fiber{\bush}$ is quite technical.
Namely, the $\s$-bushes~$\bush'$ such that~$\fiber{\bush'}$ is a facet of~$\fiber{\bush}$ are obtained by selecting an ascent or a descent~$(i,j)$ of~$\bush$ and performing a technical sewing process along~$(i,j)$.

In contrast, given an $\s$-bush~$\bush'$, the combinatorial description of which are the $\s$-bushes~$\bush$ such that~$\fiber{\bush'}$ is a face of~$\fiber{\bush}$ is easier.
Let~$j$ be an indegree~$2$ node of~$\bush'$, and let~$i \to j$ and~$j \to k$ denote the leftmost (resp.~rightmost) incoming and outgoing edges at~$j$ (note that~$k$ can be leaf).
A \defn{left} (resp.~\defn{right}) \defn{incision} of~$\bush'$ at~$j$ is an $\s$-bush~$\bush$ obtained as follows:
\begin{itemize}
\item replace the edge~$i \to j$ by~$i \to k$,
\item if the two leftmost (resp.~rightmost) outgoing edges of~$j$ have a lowest common descendant~$\ell$, then perform a (left or right) incision at~$\ell$ and replace the edge~$j \to k$ by~$j \to \ell$,
\item otherwise, replace the edge~$j \to k$ by a leaf.
\end{itemize}
Then $\fiber{\bush'}$ is a facet of~$\fiber{\bush}$ if and only if $\bush$ is obtained by some incision at an indegree~$2$ node of~$\bush'$.
Hence, $\fiber{\bush'}$ is a face of~$\fiber{\bush}$ if and only if~$\bush$ is obtained by a sequence of incisions in~$\bush'$.
We skip the proof of this description as we will not need it in the remaining of the paper.
\end{remark}

\begin{remark}
Given an~$\s$-bush~$\bush'$, we denote by~$\trees(\bush')$ the set of $\s$-trees~$\tree$ such that~$\fiber{\bush'}$ is a face of~$\fiber{\tree}$.
In other words, all $\s$-trees obtained by performing left or right incisions at all indegree~$2$ nodes of~$\bush'$.
We distinguish two particular $\s$-trees of~$\trees(\bush')$: the \defn{left tree}~$\leftTree(\bush')$ obtained by performing only left incisions, and the \defn{right tree}~$\rightTree(\bush')$ obtained by performing only right incisions.
In other words, if~$\b{x'} \in \fiber{\bush'}$, then~$\leftTree(\bush') = \ins(\s, \b{x'} + \varepsilon \b{\omega})$ and~$\leftTree(\bush') = \ins(\s, \b{x'} - \varepsilon \b{\omega})$ for a sufficiently small~$\epsilon > 0$ and~$\b{\omega} \eqdef (1, 2, \dots n) - (n, \dots, 2, 1) = \sum_{1 \le i < j \le n} \b{e}_j - \b{e}_i = \sum_{i \in [n]} (2i-n-1) \, \b{e}_i$.
Note that the holes, ascents and descents of~$\bush'$ can be derived from the ascents and descents of~$\leftTree(\bush')$ and~$\rightTree(\bush')$.
Namely, the holes of~$\bush'$ are the ascents of~$\leftTree(\bush')$ that are also descents of~$\rightTree(\bush')$, the ascents of~$\bush'$ are the ascents of~$\rightTree(\bush')$, and the descents of~$\bush'$ are the descents of~$\leftTree(\bush')$.
In fact, $\trees(\bush')$ is actually the interval~$[\leftTree(\bush'), \rightTree(\bush')]$ in the $\s$-weak order (defined in the next section).
These intervals are called pure intervals in~\cite{CeballosPons-sWeakOrderII}.
Such an interval is also determined by the $\s$-tree~$\leftTree(\bush')$ (resp.~$\rightTree(\bush')$) together with a subset of its ascents (resp.~descents).
To sum up, there are bijections between the $\s$-bushes, the faces of the $\s$-foam, the pure intervals of~\cite{CeballosPons-sWeakOrderII}, and the pairs~$(\tree, A)$ where~$A$ is a subset of ascents (resp.~descents) of an $\s$-tree~$\tree$.
\end{remark}


\section{The $\s$-weak order and facial $\s$-weak order}
\label{sec:sWeakOrder}

In this section, we consider the $\s$-weak order~$W_\s$ on $\s$-trees (directly adapted from the original definition of~\cite{CeballosPons-sWeakOrderI}), and we extend it to the facial $\s$-weak order~$FW_\s$ on all $\s$-bushes.
We prove that these two posets are actually congruence uniform lattices by exhibiting a construction by interval doublings (recovering a result of~\cite{CeballosPons-sWeakOrderI} for the $\s$-weak order).


\subsection{Recollections~\ref{sec:sWeakOrder}: The weak order and facial weak order}
\label{subsec:recollectionsWeakOrder}

We first remind basic properties of the weak order on permutations of~$[n]$ and the facial weak order on ordered set partitions of~$[n]$.

\subsubsection{Lattices and interval doublings}

A \defn{lattice}~$(L, \le, \meet, \join)$ is a poset~$(L, \le)$ where any subset~$X$ admits a \defn{meet}~$\bigMeet X$ (greatest lower bound) and a \defn{join}~$\bigJoin X$ (least upper bound).
Particularly interesting lattices are semidistributive and congruence uniform  lattices, whose definitions are delayed to~\cref{subsec:recollectionsQuotients,subsec:recollectionsCanonicalComplex}, where we will have introduced more lattice theoretic background.
We just note here that congruence uniformity implies semidistributivity.

The \defn{doubling} of a subset~$X$ of a poset~$P$ is the poset~$P[X]$ on~$(P \ssm X) \sqcup (X \times \{0,1\})$ defined~by
\begin{itemize}
\item $a \le b$ in~$P[X]$ if~$a, b \notin X$ and~$a \le b$ in~$P$,
\item $(a,i) \le b$ in~$P[X]$ if~$a \in X$, $b \notin X$, $i \in \{0,1\}$, and $a \le b$ in~$P$,
\item $a \le (b,j)$ in~$P[X]$ if~$a \notin X$, $b \in X$, $j \in \{0,1\}$, and $a \le b$ in~$P$,
\item $(a,i) \le (b,j)$ in~$P[X]$ if~$a,b \in X$, $i,j \in \{0,1\}$, and $a \le b$ in~$P$ and~$i \le j$.
\end{itemize}
Recall that~$C \subseteq P$ is \defn{order convex} if $x \le y \le z$ and~$x,z \in C$ implies~$y \in C$, and that an \defn{interval} of~$P$ is a subset of the form~$[x,z] \eqdef \set{y \in P}{x \le y \le z}$.
A.~Day~\cite{Day} observed that if~$L$ is a lattice and $C \subseteq L$ is order convex, then~$L[C]$ is again a lattice.
In fact, a lattice is congruence normal (resp.~uniform) if and only if it can be obtained from a distributive lattice by a sequence of doublings of order convex sets (resp.~of intervals).

\subsubsection{The weak order}

We now remind basic properties of the weak order on permutations and of the facial weak order on ordered partitions, and show that they are constructible by interval doublings.
The \defn{inversion set} of a permutation~$\sigma$ of~$[n]$ is~$\inv(\sigma) \! \eqdef \! \set{(\sigma_i, \sigma_j)}{1 \le i < j \le n \text{ and } \sigma_i > \sigma_j}$.
The \defn{weak order}~$W_n$ is the partially ordered set of permutations of~$[n]$ ordered by inclusion of their inversion sets.
Its cover relations are given by transpositions of adjacent entries (meaning at two consecutive positions).
See \cref{fig:weakOrder}.
The weak order is known to be a congruence uniform lattice~\cite{GuilbaudRosenstiehl,DuquenneCherfouh,ContePolyBarbut}.

\begin{figure}[b]
	\capstart
	\centerline{\raisebox{1.5cm}{\includegraphics[scale=.8]{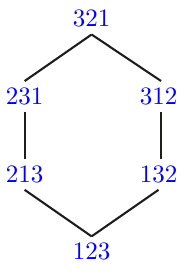}} \hspace{1.5cm} \includegraphics[scale=.8]{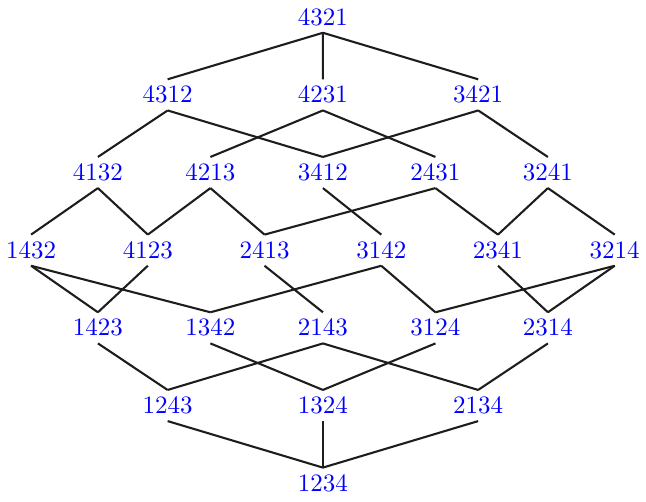}}
	\caption{The weak order~$W_n$ for~$n = 3$ (left) and~$n = 4$ (right).}
	\label{fig:weakOrder}
\end{figure}

As we will mimic it for the $\s$-weak order and facial $\s$-weak order, we now present a construction of the weak order by interval doublings, closely related to the insertion algorithm of \cref{sec:insertion}.
By induction, it suffices to exhibit a sequence of interval doublings from the weak order~$W_{n-1}$ to the weak order~$W_n$.
We denote by~$\bar\sigma$ the permutation of~$[n-1]$ obtained by deleting the entry~$n$ in a permutation~$\sigma$ of~$[n]$.
For any~$i \in [n]$, we denote by~$W_n^i$ the poset obtained from the weak order~$W_n$ by identifying two permutations~$\sigma$ and~$\sigma'$ if and only if~$\bar\sigma = \bar\sigma'$ and the relative position of~$k$ and~$n$ are the same in~$\sigma$ and~$\sigma$, for all~$n - i < k < n$.
Note that~$W_n^1$ is isomorphic to the weak order~$W_{n-1}$ and~$W_n^n$ is just the weak order~$W_n$.
For instance, \cref{fig:weakOrderIntervalDoublings} represents the posets~$W_4^i$ for~$i = 1, \dots, 4$.
In this picture, we represent an element of~$W_n^i$, on three lines (red, green, blue).
The red line contains~$[n-i-1]$, the green line contains~$[n-i,n[$, and the blue line just contains~$n$.
The red and green lines form the permutation~$\bar\sigma$, the green and blue lines record the relative position of~$k$ and~$n$ for all~$n - i < k \le n$, while the red and blue lines are incomparable.

We claim that~$W_n^{i+1}$ is obtained from~$W_n^{i}$ by interval doublings.
Consider~$U \eqdef \{u_1 < \dots < u_p\}$ and~$V \eqdef \{v_1 < \dots < v_q\}$ such that~$U \sqcup V = {]n-i, n[}$.
Let~$\tilde\sigma^{i}_{U,V}$ (resp.~$\tilde\tau^{i}_{U,V}$) denote the class in~$W_n^{i}$ of the permutations~$\sigma$ (resp.~$\tau$) of~$[n]$ such that~$\bar\sigma = [1, \dots, n-i-1, u_1, \dots, u_p, n-i, v_1, \dots, v_q]$ (resp.~$\bar\tau = [u_p, \dots, u_1, n-i, v_q, \dots, v_1, n-i-1, \dots, 1]$) and where~$n$ is to the right of any~$u \in U$ and to the left of any~$v \in V$.
We invite the reader to check that~$W_n^{i+1}$ is obtained from~$W_n^{i}$ by doubling, for all~$U \sqcup V = {]n-i, n[}$, the interval~$[\tilde\sigma^{i}_{U,V}, \tilde\tau^{i}_{U,V}]$ in~$W_n^{i}$ to the intervals~$[\tilde\sigma^{i+1}_{U \cup \{n-i\},V}, \tilde\tau^{i+1}_{U \cup \{n-i\},V}]$ and~$[\tilde\sigma^{i+1}_{U,V \cup \{n-i\}}, \tilde\tau^{i+1}_{U,V \cup \{n-i\}}]$ in~$W_n^{i+1}$.
These interval doublings are illustrated in \cref{fig:weakOrderIntervalDoublings}.

\begin{figure}[t]
	\capstart
	\centerline{\includegraphics[scale=.6]{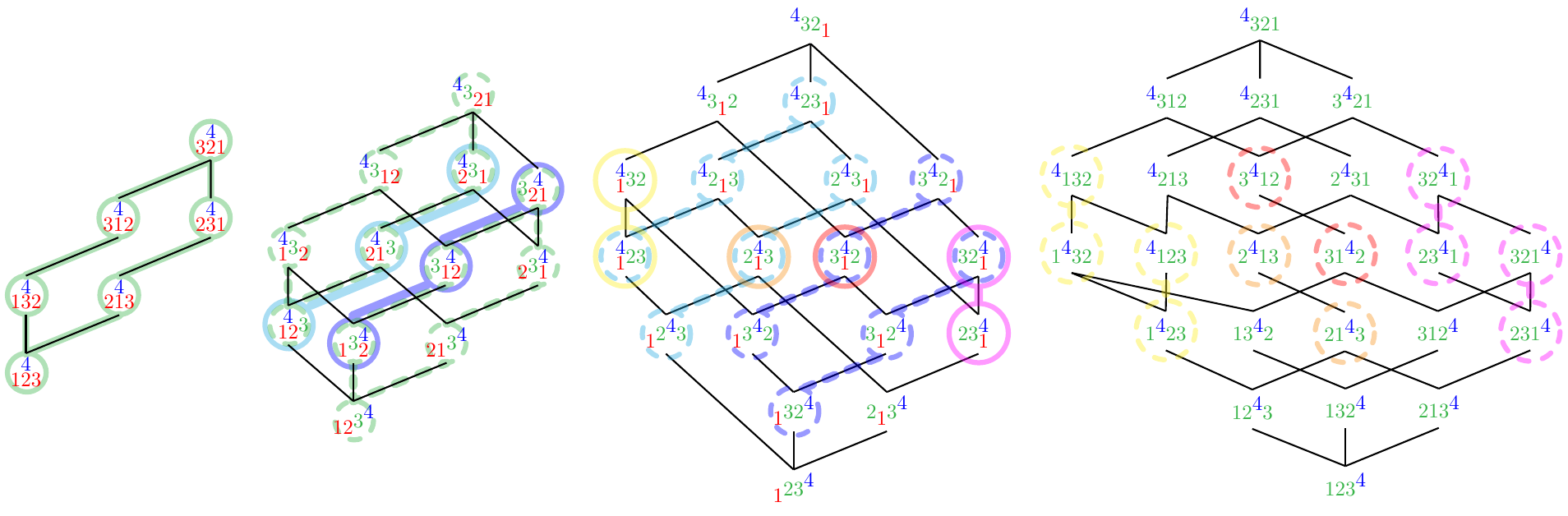}}
	\caption{Interval doublings from the weak order~$W_3$ (left) to the weak order~$W_4$ (right). The four lattices are~$W_3 \simeq W_4^1$, $W_4^2$, $W_4^3$, and~$W_4^4 = W_4$. We bold some intervals of~$W_4^{i-1}$ and dash the corresponding doubled intervals in~$W_4^i$.}
	\label{fig:weakOrderIntervalDoublings}
\end{figure}

\vspace{-.1cm}
\parpic(4.5cm,3cm)(10pt, 120pt)[r][b]{\includegraphics[scale=.8]{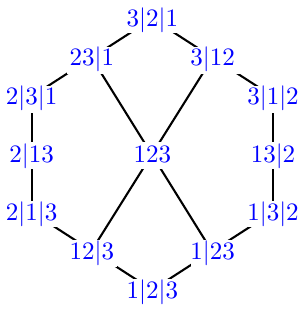}}{
\subsubsection{The facial weak order}
Finally, we briefly describe the facial weak order.
It is a lattice structure on all cones of the braid arrangement, or equivalently, on all faces of the permutahedron, which extends the weak order.
It is illustrated on the right for~${n = 3}$.
It was introduced for the braid arrangement in~\cite{KrobLatapyNovelliPhanSchwer,BoulierHivertKrobNovelli}, studied for arbitrary Coxeter arrangements in~\cite{PalaciosRonco, DermenjianHohlwegPilaud}, and extended even further in~\cite{DermenjianHohlwegMcConvillePilaud, Hanson}.
The \defn{facial weak order}~$FW_n$ is the partial order on ordered set partitions of~$[n]$ where~$\mu \le \nu$ if the following equivalent assertions holds: 
\begin{itemize}
\item $\min(\mu) \le \min(\nu)$ and~$\max(\mu) \le \max(\nu)$, where~$\min(\mu)$ and~$\max(\mu)$ respectively denote the weak order minimal and maximal permutations refining~$\mu$,
\item $\inv(\mu) \subseteq \inv(\nu)$ and~$\ninv(\mu) \supseteq \ninv(\nu)$, where~$\inv(\mu)$ (resp.~$\ninv(\mu)$) is the set of pairs~$i < j$ such that the part of~$\mu$ containing~$i$ is weakly after (resp.~before) the part of~$\mu$ containing~$j$,
\item the corresponding cones $\polytope{C}_{\mu}$ and~$\polytope{C}_{\nu}$ of the braid arrangement are connected by a path of pairs~$(\polytope{F}, \polytope{G})$, where either~$\polytope{F}$ is a facet of~$\polytope{G}$ with the same weak order maximum, or~$\polytope{G}$ is a facet of~$\polytope{F}$ with the same weak order minimum.
\end{itemize}
}


\subsection{The $\s$-weak order}
\label{subsec:sWeakOrder}

We now define the $\s$-weak order as introduced in~\cite{CeballosPons-sWeakOrderI} (modulo our minor convention changes, see \cref{rem:sTreeVSsDecreasingTrees}).

\begin{definition}[{\cite[Def.~1.3]{CeballosPons-sWeakOrderI}}]
\label{def:positions}
For an $\s$-tree~$\tree$ and~$1 \le i < j \le n$, the \defn{position}~$\pos(\tree, i, j) \in \llbracket s_i]$ is the minimum of~$s_i$ and the number of outgoing edges of~$i$ strictly to the right of the increasing path from the root of~$\tree$ to~$j$.
\end{definition}

\begin{lemma}
\label{lem:propertiesPositions}
For any~$\s$-tree~$\tree$ and~$1 \le i < j < k \le n$, 
\[
\pos(\tree, j, k) > 0 \implies \pos(\tree, i, j) \le \pos(\tree, i, k)
\quad\text{and}\quad
\pos(\tree, j, k) < s_j \implies \pos(\tree, i, j) \ge \pos(\tree, i, k).
\]
\end{lemma}

\begin{proof}
If~$\pos(\tree, j, k) > 0$, then $j$ is weakly right of the path from the root of~$\tree$ to~$k$.
Hence, the path from the root of~$\tree$ to~$j$ is weakly on the right of the path from the root of~$\tree$ to~$k$.
By definition, this yields~$\pos(\tree, i, j) \le \pos(\tree, i, k)$.
The other inequality is similar.
\end{proof}

Although not strictly necessary in this paper, we recall from~\cite{CeballosPons-sWeakOrderI} that the inequalities of~\cref{lem:propertiesPositions} actually characterize the position vectors of $\s$-trees.

\begin{proposition}[{\cite[Prop.~1.6]{CeballosPons-sWeakOrderI}}]
\label{prop:characterizationPositions}
The following are equivalent for a $\binom{n}{2}$-tuple~$(P_{i,j})_{1 \le i < j \le n}$:
\begin{enumerate}
\item there exists an $\s$-tree~$\tree$ such that~$\pos(\tree, i, j) = P_{i,j}$ for all~$1 \le i < j \le n$,
\item $0 \le P_{i,j} \le s_i$ for all~$1 \le i < j \le n$, and $P_{j,k} > 0 \implies P_{i,j} \le P_{i,k}$ and $P_{j,k} < s_j \implies P_{i,j} \ge P_{i,k}$ for all~$1 \le i < j < k \le n$.
\end{enumerate}
\end{proposition}

The main object of~\cite{CeballosPons-sWeakOrderI,CeballosPons-sWeakOrderII} is the following order on $\s$-trees, illustrated in~\cref{fig:weakOrders}.

\begin{figure}[t]
	\capstart
	\centerline{\raisebox{.5cm}{\includegraphics[scale=.6]{120-weakOrder}} \hspace{1.5cm} \includegraphics[scale=.6]{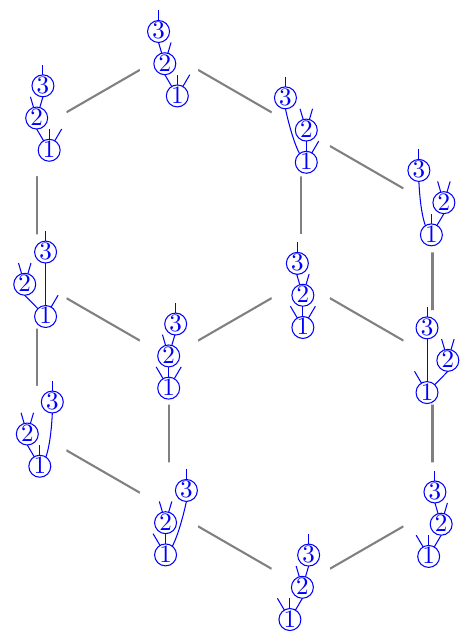}}
	\caption{The $\s$-weak order~$W_\s$ for $\s = (1,2,0)$ (left) and~$\s = (2,1,0)$ (right).}
	\label{fig:weakOrders}
\end{figure}

\begin{definition}[{\cite[Def.~1.9]{CeballosPons-sWeakOrderI}}]
\label{def:sWeakOrder}
The \defn{$\s$-weak order}~$W_\s$ is the partially ordered set of $\s$-trees given by~$\tree \le \tree'$ if and only if~$\pos(\tree, i, j) \le \pos(\tree', i, j)$ for all~$1 \le i < j \le n$.
\end{definition}

We now describe the cover relations in the $\s$-weak order in terms of ascents, descents, and rotations (see \cref{def:ascentDescentBush,rem:ascentDescentTree,def:rotation}).

\begin{proposition}[{\cite[Thm.~1.32]{CeballosPons-sWeakOrderI}}]
\label{prop:positionsRotation}
If two $\s$-trees~$\tree$ and~$\tree'$ are related by the rotation of an ascent~$(i,j)$ of~$\tree$ to a descent~$(i,j)$ of~$\tree'$, then~$\tree$ is the maximal (resp.~$\tree'$ is the minimal) $\s$-tree such that~$\tree \le \tree'$ and~$\pos(\tree, i, j) < \pos(\tree', i, j)$.
\end{proposition}

\begin{proposition}[{\cite[Thm.~1.32]{CeballosPons-sWeakOrderI}}]
\label{prop:sWeakOrderCoverRelations}
The $\s$-trees which cover (resp.~are covered by) an $\s$-tree~$\tree$ in the $\s$-weak order~$W_\s$ are precisely those obtained by rotating an ascent (resp.~descent) of~$\tree$.
\end{proposition}

We obtain the following geometric consequence of \cref{coro:dualGraph,prop:sWeakOrderCoverRelations}.

\begin{corollary}
\label{coro:sFoamsWeakOrder}
The Hasse diagram of the $\s$-weak order~$W_\s$ is isomorphic to the dual graph of the $\s$-foam oriented in the direction~$\b{\omega} \eqdef (1, 2, \dots n) - (n, \dots, 2, 1) = \!\!\sum\limits_{1 \le i < j \le n} \!\! \b{e}_j - \b{e}_i = \sum\limits_{i \in [n]} (2i-n-1) \, \b{e}_i$.
\end{corollary}

\begin{remark}
We note that C.~Ceballos and V.~Pons also briefly mention the construction of the ``$\s$-braid arrangement'' in \cite[end of Sect.~3.2~\&~Fig.~27]{CeballosPons-sWeakOrderII} that seems to coincide with our $\s$-foam~\cite{Pons-personnalCommunication}.
Note that we prefer the term ``foam'' to ``arrangement'' as it is not anymore a hyperplane arrangement (not even affine).
\end{remark}

Finally, we state to the key result of~\cite{CeballosPons-sWeakOrderI}.

\begin{theorem}[{\cite[Thms.~1.21 \& 1.40]{CeballosPons-sWeakOrderI}}]
\label{thm:sWeakOrderLattice}
The $\s$-weak order~$W_\s$ is a congruence uniform lattice.
\end{theorem}

\begin{remark}
The proof of the lattice property in~\cite{CeballosPons-sWeakOrderI} actually exploits the characterization of~\cref{prop:characterizationPositions} to explicitly describe the join in the $\s$-weak order~$W_\s$.
Namely, the join~$\tree \join \tree'$ of two $\s$-trees~$\tree$ and~$\tree'$ satisfies~$\pos(\tree \join \tree', i, j) = \smash{M_{i,j}^{tc}}$ where~$M_{i,j} \eqdef \max \big( \pos(\tree, i, j), \pos(\tree', i, j) \big)$ and $\smash{M_{i,k}^{tc}} \eqdef \max \set{M_{j_0,j_1}}{i = j_0 < \dots < j_q = k \text{ and } M_{j_p, j_{p+1}} > 0 \text{ for all } p \in [q]}$ for~${1 \le i < j \le n}$.
The congruence uniformity is proved in~\cite{CeballosPons-sWeakOrderI} by showing semidistributivity, and exhibiting a certain edge labeling of the $\s$-weak order that ensures congruence uniformity.
We now give an alternative argument for the congruence uniformity (and hence the semidistributivity) based on interval doublings, which naturally arise from our insertion algorithm, and will be extended to the facial $\s$-weak order in \cref{subsec:sFacialWeakOrder}.
We note that C.~Ceballos and V.~Pons also mention a proof by interval doublings in~\cite[Rem.~1.41]{CeballosPons-sWeakOrderI}.
Their sequence of doublings can be found in~\cite[function \texttt{lattice\_doublings} at line 1489]{Pons-sagedemo} and seems quite different from ours~\cite{Pons-personnalCommunication}.
\end{remark}

\begin{proposition}
\label{prop:sWeakOrderIntervalDoblings}
The $\s$-weak order~$W_\s$ is constructible by a sequence of interval doublings.
\end{proposition}

\begin{proof}[Proof sketch]
We just give a brief description of the sequence of doublings, generalizing the description of \cref{subsec:recollectionsWeakOrder} for the interval doublings in the weak order.
Let~$\bar\s \eqdef \s_{\le n-1} = (s_1, \dots, s_{n-1})$ and denote by~$\bar\tree$ the $\bar\s$-tree obtained by deleting the node~$n$ in an $\s$-tree~$\tree$.
For any~$i \in [n-1]$ and~${j \in \llbracket s_{n-i}]}$, we denote by~$W_\s^{i,j}$ the poset obtained from the $\s$-weak order~$W_\s$ by identifying two $\s$-trees~$\tree$ and~$\tree'$ if and only if~$\bar\tree = \bar\tree'$, $\pos(\tree, k, n) = \pos(\tree', k, n)$ for all~${n - i < k < n}$, and~$\pos(\tree, n-i, n)$ and~$\pos(\tree', n-i, n)$ either coincide or are both at least~$j$.
Note that~$W_\s^{1,0}$ \linebreak is isomorphic to the $\bar\s$-weak order~$W_{\bar\s}$, that~$W_\s^{n-1,s_1}$ is just the $\s$-weak order~$W_\s$, and \linebreak that~${W_\s^{i,s_i} = W_\s^{i+1,0}}$.
See \cref{fig:sWeakOrderIntervalDoublings} for illustrations.

We now observe that we can construct~$W_\s^{i,j}$ from~$W_\s^{i,j-1}$ by interval doublings.
For~$J \subseteq [n]$, the \defn{left} (resp.~\defn{right}) \defn{$\s$-comb} with nodes~$J$ is the increasing tree on~$J$, where each node~$j \in J$ has~$s_j+1$ outgoing edges, which are all leaves except the leftmost (resp.~rightmost).
Consider~${U \eqdef \{u_1 < \dots < u_p\}}$ and~$V \eqdef \{v_1 < \dots < v_q\}$ such that~$U \sqcup V = {]n-i, n[}$. 
We denote by~$\tree[S]_{U,V}^{i,j}$ the class in~$W_\s^{i,j}$ of $\s$-trees~$\tree[S]$ such that~$\bar{\tree[S]}$ is obtained by attaching the right comb with nodes~$U$ to the $(j+1)$-st rightmost outgoing edge of~$n-i$ in the right comb with nodes~$[n-i-1]\cup V$ and where the node $n$ is attached either to the rightmost leaf of the right comb with nodes~$U$ or to leftmost leaf of the $j$-th rightmost subtree of~$n-i$.
We denote by~$\tree[T]_{U,V}^{i,j}$ the class in $W_\s^{i,j}$ of $\s$-trees~$\tree$ such that $\bar{\tree[T]}$ is obtained by attaching the left comb with nodes~$V$ to the $j$-th rightmost outgoing edgege of $n-i$ in the left comb with nodes $[n-i-1]\cup U$ and where the node $n$ is attached either to the leftmost leaf of the left comb with nodes~$V$ or to the rightmost leaf of the $(j+1)$-st rightmost subtree of~$n-i$. 
We let the reader check that~$W_\s^{i,j}$ is obtain from~$W_\s^{i,j-1}$ by doubling the intervals~$[\tree[S]_{U,V}^{i,j}, \tree[T]_{U,V}^{i,j}]$ for all~$U \sqcup V = {]n-i,n[}$.
\begin{figure}[t]
	\capstart
	\centerline{\includegraphics[scale=.6]{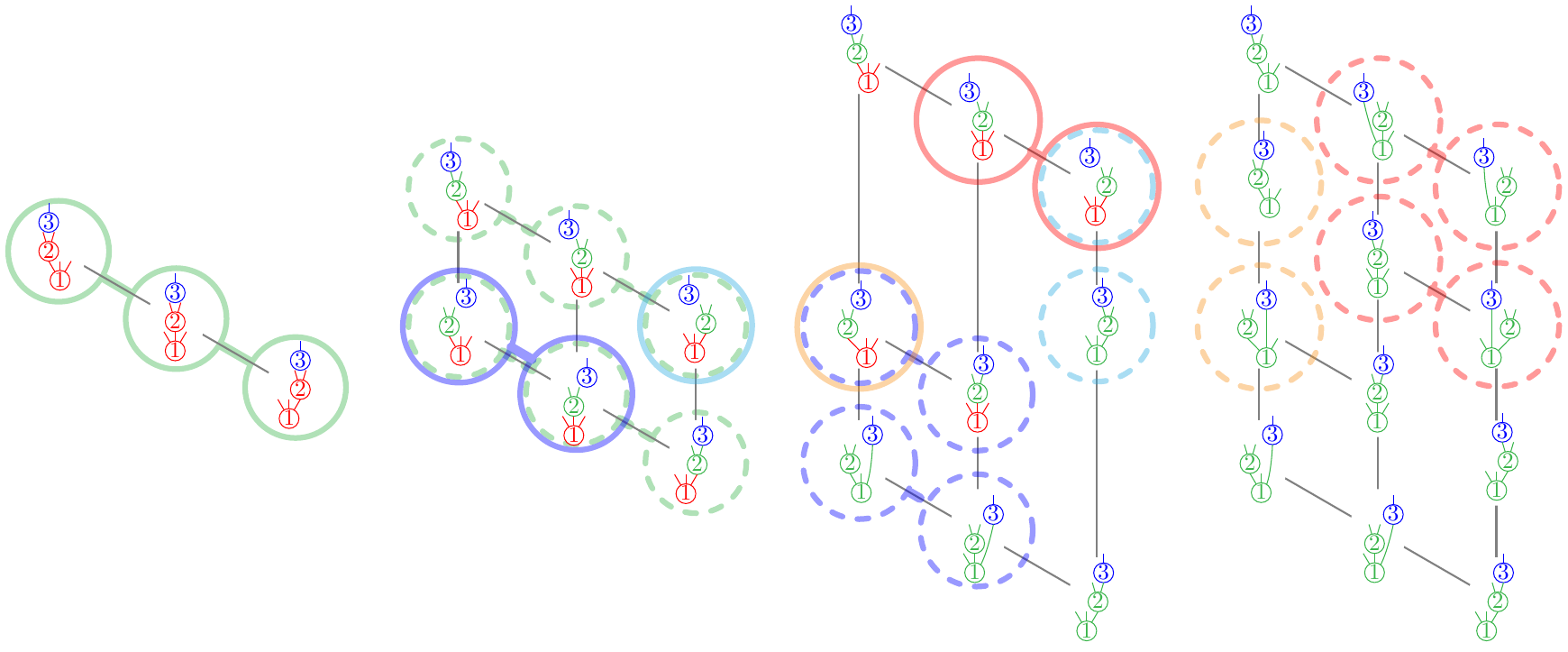}}
	\caption{Interval doublings from the $(2,1)$-weak order (left) to the $(2,1,0)$-weak order (right). The four lattices are~$W_{(2,1)} \simeq W_{(2,1,0)}^{1,0}$, $W_{(2,1,0)}^{1,1} = W_{(2,1,0)}^{2,0}$, $W_{(2,1,0)}^{2,1}$, and~$W_{(2,1,0)}^{2,2} = W_{(2,1,0)}$.}
	\label{fig:sWeakOrderIntervalDoublings}
\end{figure}
\end{proof}


\subsection{The facial $\s$-weak order}
\label{subsec:sFacialWeakOrder}

We now extend the $\s$-weak order to all $\s$-bushes (hence, to all polyhedra of the $\s$-foam~$\sFoam$).
Although it is not required to read the remaining of the paper, we have included this section as the ideas are very similar to that of \cref{subsec:sWeakOrder}.

\begin{definition}
\label{def:leftRightPositions}
For an $\s$-bush~$\bush$ and~$1 \le i < j \le n$, we define the \defn{left} (resp.~\defn{right}) \defn{position}~$\lpos(\bush, i, j)$ (resp.~$\rpos(\bush, i, j)$) as follows:
\begin{itemize}
\item If~$i$ is an ancestor of~$j$, consider the rightmost (resp.~leftmost) increasing path~$\pi$ from~$i$ to~$j$ such that any inner node of~$\pi$ with an incoming edge strictly left (resp.~right) of~$\pi$ has an outgoing edge strictly left (resp.~right) of~$\pi$, and define~$\lpos(\bush, i, j)$ (resp.~$\rpos(\bush, i, j)$) as the number of outgoing edges of~$i$ strictly left (resp.~right) of~$\pi$.
\item Otherwise, define~$\lpos(\bush, i, j) = s_i$ and~$\rpos(\bush, i, j) = 0$ if~$i$ is on the left of any increasing path from the root of~$\bush$ to~$j$, and $\lpos(\bush, i, j) = 0$ and~$\rpos(\bush, i, j) = s_i$ otherwise.
\end{itemize}
\end{definition}

\begin{remark}
A few observations on \cref{def:leftRightPositions}:
\begin{itemize}
\item \cref{def:leftRightPositions} contains \cref{def:positions}: for all~$1 \le i < j \le n$, there is at most one increasing path from~$i$ to~$j$ in an $\s$-tree~$\tree$, so that $\pos(\tree, i, j) = \rpos(\tree, i, j) = s_i - \lpos(\tree, i, j)$.
\item For any $\s$-bush~$\bush$ and any~$1 \le i < j \le n$, we have~$0 \le \lpos(\bush, i, j) , \rpos(\bush, i, j) \le s_i$, and~$\lpos(\bush, i, j) + \rpos(\bush, i, j) \le s_i+1$.
Moreover, both~$\lpos(\bush, i, j)$ and~$\rpos(\bush, i, j)$ satisfy the conditions of \cref{lem:propertiesPositions}.
We conjecture that these conditions characterize the left and right positions of~$\s$-bushes.
\item For any two $\s$-bushes~$\bush$ and~$\bush'$, one can check that the closed fiber~$\closedFiber{\bush}$ is a face of the closed fiber~$\closedFiber{\bush'}$ if and only if~$\lpos(\bush, i, j) \ge \lpos(\bush', i, j)$ and $\rpos(\bush, i, j) \ge \rpos(\bush', i, j)$ for all~$1 \le i < j \le n$.
\end{itemize}
\end{remark}

We now use~$\lpos(\bush, i, j)$ and~$\rpos(\bush, i, j)$ to define a different order on $\s$-bushes, which extends the $\s$-weak order of \cref{def:sWeakOrder} and generalizes the facial weak order of~\cite{KrobLatapyNovelliPhanSchwer}.
It is illustrated in \cref{fig:facialWeakOrders}.

\begin{figure}[t]
	\capstart
	\centerline{\raisebox{2cm}{\includegraphics[scale=.6]{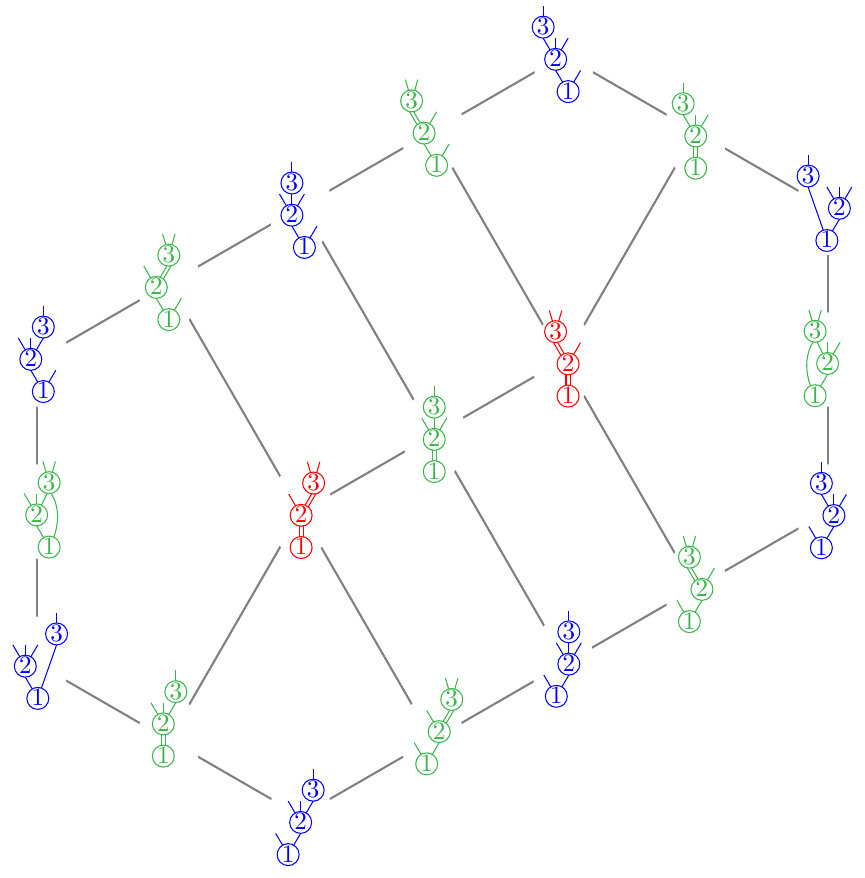}} \hspace{.5cm} \includegraphics[scale=.6]{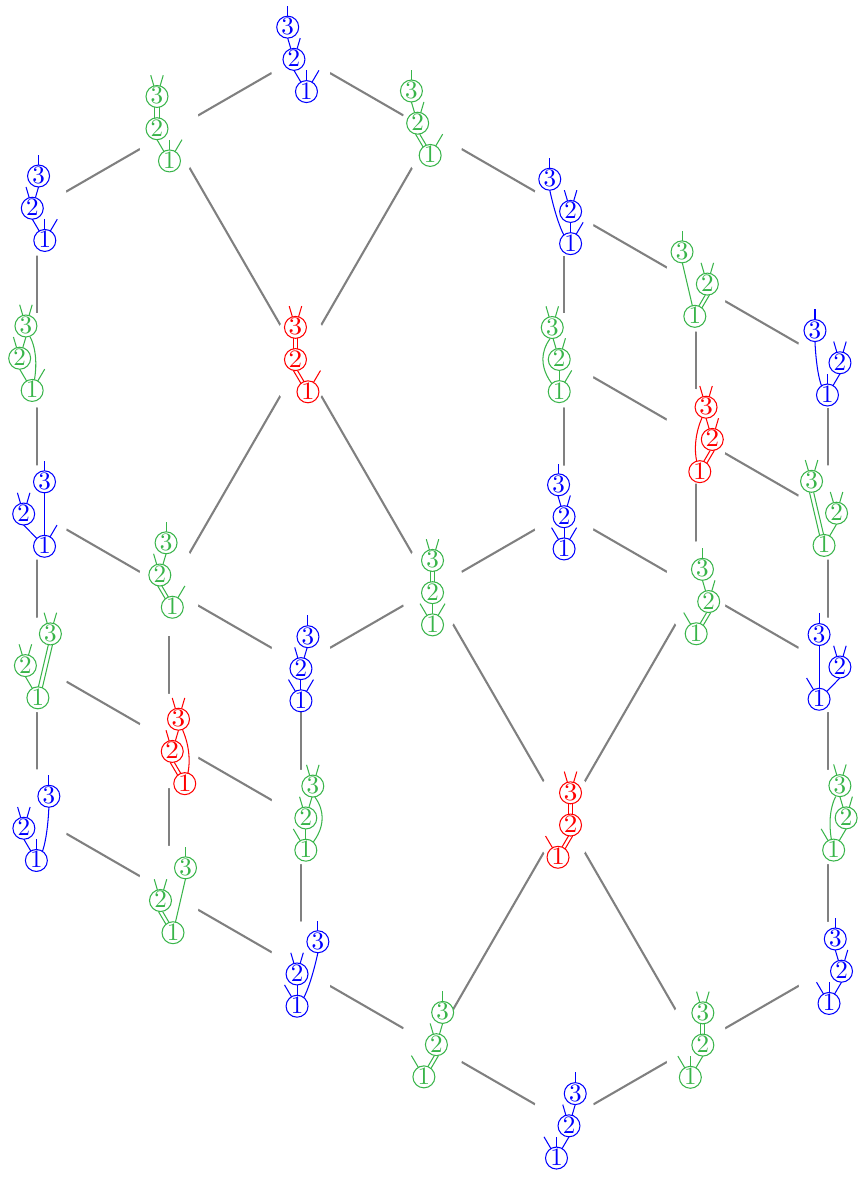}}
	\caption{The facial $\s$-weak order~$FW_\s$ for $\s = (1,2,0)$ (left) and~$\s = (2,1,0)$ (right).}
	\label{fig:facialWeakOrders}
\end{figure}

\begin{definition}
\label{def:sFacialWeakOrder}
The \defn{facial $\s$-weak order}~$FW_\s$ is the partially ordered set of $\s$-bushes given by~$\bush \le \bush'$ if and only if~$\lpos(\bush, i, j) \ge \lpos(\bush', i, j)$ and~$\rpos(\bush, i, j) \le \rpos(\bush', i, j)$ for all~$1 \le i < j \le n$.
\end{definition}

\begin{remark}
As for the facial weak order on ordered set partitions of~$[n]$, there are several equivalent perspectives on the facial $\s$-weak order.
Namely, 
\begin{itemize}
\item $\bush \le \bush'$ if and only if $\min(\bush) \le \min(\bush')$ and~$\max(\bush) \le \max(\bush')$, where~$\min(\bush)$ and~$\max(\bush)$ denote the $\s$-weak order minimal and maximal $\s$-trees refining the $\s$-bush~$\bush$,
\item if an $\s$-bush~$\bush$ has an indegree~$2$ node~$j$, then it covers (resp.~is covered by) the $\s$-bush obtained from~$\bush$ by detaching the left (resp.~right) incoming edge of~$j$, and these are all cover relations of the facial $\s$-weak order (note that this simple description contrasts with the involved description of the cover relations of the refinement order, see \cref{rem:coverRelationsRefinementBushes}).
\end{itemize}
\end{remark}

\begin{proposition}
\label{prop:sFacialWeakOrderIntervalDoblings}
The facial $\s$-weak order~$FW_\s$ is constructible by interval doublings, so that it is a congruence uniform lattice.
\end{proposition}

\begin{proof}[Proof sketch]
The proof follows the same lines as that of \cref{prop:sWeakOrderIntervalDoblings}.
Let~$\bar\s \eqdef \s_{\le n-1} = (s_1, \dots, s_{n-1})$ and denote  by~$\bar\bush$ the $\bar\s$-bush obtained by deleting the node~$n$ in an $\s$-bush~$\bush$.
For any~$i \in [n-1]$ and~$j \in \llbracket s_{n-i}]$, we denote by~$FW_\s^{i,j}$ the poset obtained from the facial $\s$-weak order~$FW_\s$ by identifying two $\s$-bushes~$\bush$ and~$\bush'$ if and only if~$\bar\bush = \bar\bush'$, $\lpos(\bush, k, n) = \lpos(\bush', k, n)$ and $\rpos(\bush, k, n) = \rpos(\bush', k, n)$ for all~$n - i < k < n$, $\lpos(\bush, n-i, n)$ and~$\lpos(\bush', n-i, n)$ either coincide or are both at most~$s_i-j$, and $\rpos(\bush, n-i, n)$ and~$\rpos(\bush', n-i, n)$ either coincide or are both at least~$j$.
As before, observe that~$FW_\s^{1,0}$ is isomorphic to the facial $\bar\s$-weak order~$FW_{\bar\s}$, that~$FW_\s^{n-1,s_1}$ is just the facial $\s$-weak order~$FW_\s$, and that~$FW_\s^{i,s_i} = FW_\s^{i+1,1}$.
We then prove that we can construct~$FW_\s^{i,j}$ from~$FW_\s^{i,j-1}$ by a sequence of interval tripling (which are obtained by doubling twice each interval).
\end{proof}


%


\section{Canonical representations in the $\s$-weak order}
\label{sec:canonicalComplex}

This section is devoted to the combinatorial description of the join irreducibles and canonical join representations in the $\s$-weak order~$W_\s$.


\subsection{Recollections~\ref{sec:canonicalComplex}: Canonical representations in the weak order}
\label{subsec:recollectionsCanonicalComplex}

We first recall the combinatorics of canonical representations of permutations in terms of non-crossing arc diagrams~\cite{Reading-arcDiagrams}, and of weak order intervals in terms of semi-crossing arc bidiagrams~\cite{AlbertinPilaud}.

\subsubsection{Canonical representations in semidistributive lattices}

We first recall the definition of canonical representations in a finite semi-distributive lattice~$(L, \le, \meet, \join)$.
A \defn{join representation} of~${x \in L}$ is a subset~$J \subseteq L$ such that~$x = \bigJoin J$.
Such a representation is \defn{irredundant} if~$x \ne \bigJoin J'$ for any strict subset~$J' \subsetneq J$.
The irredundant join representations of an element~$x \in L$ are ordered by containment of the down sets of their elements, \ie~$J \le J'$ if and only if for any~$y \in J$ there exists~$y' \in J'$ such that~$y \le y'$ in~$L$.
The \defn{canonical join representation} of~$x$ is the minimal irredundant join representation of~$x$ for this order when it exists.
Its elements are the \defn{canonical joinands} of~$x$.
Note that by definition, the canonical joinands form an antichain of join irreducible elements of~$L$ (the elements that are their only irredundant join representation, or equivalently that cover a single element).

The lattice~$L$ is \defn{join semidistributive} if the following equivalent assertions hold
\begin{itemize}
\item $x \join y = x \join z$ implies $x \join (y \meet z) = x \join y$ for any~$x, y, z \in L$,
\item all elements of~$L$ admit a canonical join representation.
\end{itemize}
The \defn{canonical join complex} of~$L$ is the simplicial complex whose vertices are the join irreducible elements of~$L$ and whose simplices are the canonical join representations of elements of~$L$, see~\cite{Reading-arcDiagrams, Barnard}.
The \defn{canonical meet representations}, the \defn{canonical meetands}, and the \defn{canonical meet complex} of a meet semidistributive lattice are defined dually.

The lattice~$L$ is \defn{semidistributive} if it is both join and meet semidistributive.
In this situation, its canonical join and meet complexes are isomorphic flag simplicial complexes~\cite{Barnard}.
The \defn{canonical representation} of an interval~$[x,y]$ of~$L$ is the disjoint union~$J \sqcup M$, where~$J$ is the canonical join representation of~$x$ and~$M$ the canonical meet representation of~$y$.
The \defn{canonical complex} of~$L$ is the flag simplicial complex whose vertex set is the disjoint union of the join irreducible elements of~$L$ and the meet irreducible elements of~$L$, and whose simplices are the canonical representations of the intervals of~$L$, see~\cite{AlbertinPilaud}.

\subsubsection{Canonical join representations of permutations and non-crossing arc diagrams}

We now recall the combinatorial model of~\cite{Reading-arcDiagrams} for the canonical join representations in the weak order.
An \defn{arc} on~$[n]$ is a quadruple~$(i, j, A, B)$ where~$1 \le i < j \le n$ and~$A \sqcup B$ forms a partition of the interval~${]i,j[} \eqdef \{i+1, \dots, j-1\}$.
We represent an arc by a curve wiggling around the vertical axis, starting at height~$i$ and ending at height~$j$, and passing to the right of the points of~$A$ and to the left of the points of~$B$.
Two arcs~$\alpha \eqdef (i, j, A, B)$ and~$\alpha' \eqdef (i', j', A', B')$ are \defn{crossing} if they cross in their interior or have the same bottom endpoints or the same top endpoints.
In other words, if there exist~$u \ne v$ such that ${u \in (A' \cup \{i', j'\}) \cap (B \cup \{i, j\})}$ and ${v \in (A \cup \{i, j\}) \cap (B' \cup \{i', j'\})}$.
A \defn{non-crossing arc diagram} on~$[n]$ is a collection of pairwise non-crossing arcs on~$[n]$.

\vspace{-.1cm}
\parpic(5cm,1.3cm)(20pt, 100pt)[r][b]{\includegraphics[scale=.75]{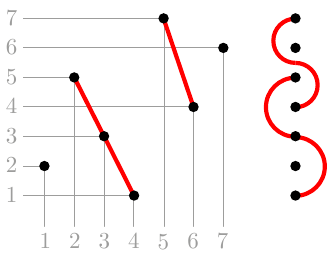}}{
In~\cite{Reading-arcDiagrams}, N.~Reading defined an elegant bijection between permutations of~$[n]$ and non-crossing arc diagrams on~$[n]$.
It sends a permutation~$\sigma$ to the non-crossing arc diagram~$\delta_\join(\sigma)$ with an arc~$(\sigma_{j+1}, \sigma_j, A_j, B_j)$ for each descent~$j \in [n-1]$ of~$\sigma$ (\ie with~$\sigma_j > \sigma_{j+1}$), where

\begin{minipage}{8cm}
\begin{align*}
A_j & \eqdef \set{\sigma_i}{1 \le i < j \text{ and } \sigma_{j+1} < \sigma_i < \sigma_j} \\
\text{and} \qquad
B_j & \eqdef \set{\sigma_k}{j+1 < k \le n \text{ and } \sigma_{j+1} < \sigma_k < \sigma_j}.
\end{align*}
\end{minipage}
}

\vspace{.4cm}
\noindent
As illustrated on the right with~$\sigma = 2531746$, the non-crossing arc diagram~$\delta_\join(\sigma)$ can also been visualized by representing the table~$(j, \sigma_j)$ of the permutation~$\sigma$, drawing the segments corresponding to the descents of~$\sigma$, and collapsing the picture vertically, allowing the segments to bend but not to cross each other nor to pass through a point.

\begin{figure}[t]
	\capstart
	\centerline{\raisebox{2.7cm}{\includegraphics[scale=.8]{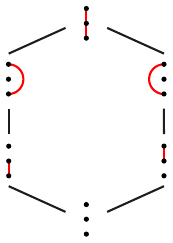}} \hspace{1.5cm} \includegraphics[scale=.8]{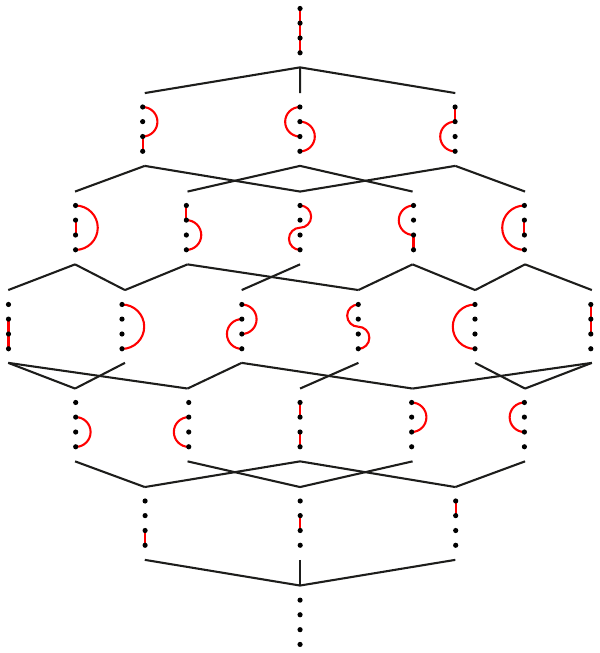}}
	\caption{The weak order for~$n = 3$ (left) and~$n = 4$ (right), where each permutation~$\sigma$ is replaced by its non-crossing arc diagram~$\delta_\join(\sigma)$.}
	\label{fig:weakOrderArcDiagrams}
\end{figure}

The single arcs correspond to permutations with a single descent, that is, to join irreducible permutations of the weak order.
More generally, the non-crossing arc diagram~$\delta_\join(\sigma)$ of a permutation~$\sigma$ encodes the canonical join representation of~$\sigma$ in the weak order.

\begin{theorem}[\cite{Reading-arcDiagrams}]
\label{thm:canonicalJoinComplexWeakOrder}
The map~$\delta_\join$ is a bijection from the permutations of~$[n]$ to the non-crossing arc diagrams on~$[n]$.
The canonical join representation~$\sigma$ is precisely the set of join irreducible permutations corresponding to the arcs of the non-crossing arc diagram~$\delta_\join(\sigma)$.
In other words, the canonical join complex of the weak order~$W_n$ is isomorphic to the non-crossing arc complex on~$[n]$.
\end{theorem}

See~\cref{fig:weakOrderArcDiagrams} for the weak order on non-crossing arc diagrams and~\cite{Reading-arcDiagrams} for details.
The canonical meet representation of a permutation is obtained similarly, using ascents instead of descents to define the non-crossing arc diagram~$\delta_\meet(\sigma)$.

\subsubsection{Canonical representations of weak order intervals and semi-crossing arc bidiagrams}

Two arcs~$\alpha_\join \eqdef (i_\join, j_\join, A_\join, B_\join)$ and~$\alpha_\meet \eqdef (i_\meet, j_\meet, A_\meet, B_\meet)$ are \defn{semi-crossing} if they do not cross in their interiors with~$\alpha_\join$ going in the diagonal direction and $\alpha_\meet$ in the anti-diagonal direction at the crossing, and do not start at the same point with~$\alpha_\join$ leaving right of~$\alpha_\meet$ at this point, and do not end at the same point with~$\alpha_\join$ arriving left of~$\alpha_\meet$ at this point.
In other words, if there exists no~$u < v$ with ${u \in (A_\meet \cup \{i_\meet\}) \cap (B_\join \cup \{i_\join\})}$ and ${v \in (A_\join \cup \{j_\join\}) \cap (B_\meet \cup \{j_\meet\})}$.
A \defn{semi-crossing arc bidiagram} is a disjoint union~$\delta_\join \sqcup \delta_\meet$ of two non-crossing arc diagrams~$\delta_\join$ and~$\delta_\meet$ such that any~$\alpha_\join \in \delta_\join$ and~$\alpha_\meet \in \delta_\meet$ are semi-crossing.
The following statement is illustrated in \cref{fig:weakOrderCanonicalComplex} and details can be found in~\cite{AlbertinPilaud}.

\begin{figure}[t]
	\capstart
	\centerline{\raisebox{1.5cm}{\includegraphics[scale=1]{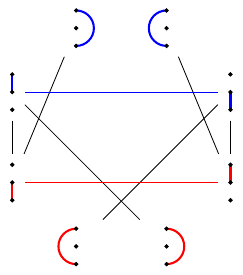}} \hspace{1cm} \includegraphics[scale=.45]{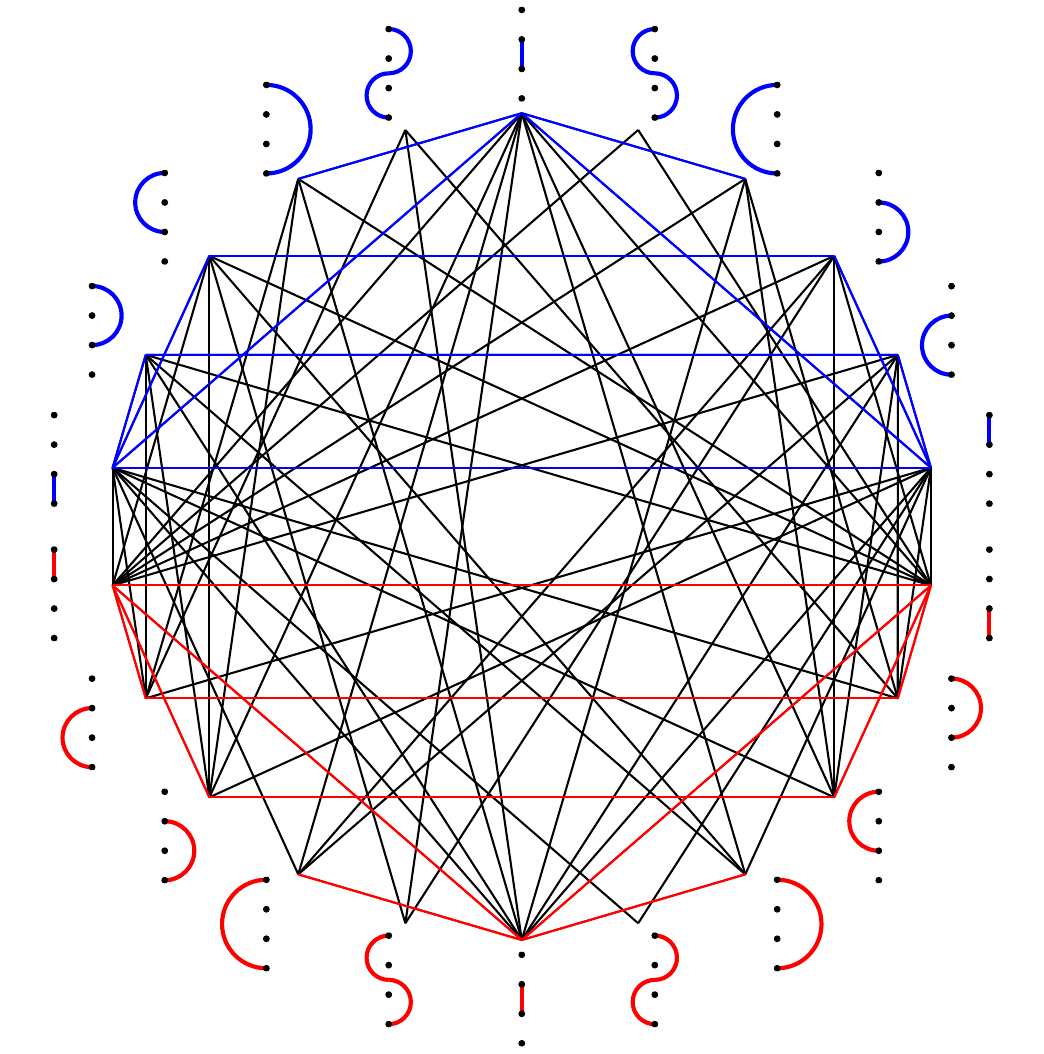}}
	\caption{The canonical complex of the weak order for~$n = 3$ (left) and~$n = 4$ (right), interpreted as the semi-crossing arc complex. The canonical join and meet complexes of the weak order, interpreted as non-crossing arc complexes, appear in red and blue respectively. Adapted from~\cite[Fig.~8 \& 9]{AlbertinPilaud}.}
	\label{fig:weakOrderCanonicalComplex}
\end{figure}

\begin{theorem}[\cite{AlbertinPilaud}]
\label{thm:canonicalComplexWeakOrder}
The map~$[x,y] \mapsto \delta_\join(\sigma) \sqcup \delta_\meet(\sigma)$ is a bijection from the weak order intervals to the semi-crossing arc bidiagrams, which encodes the canonical representations of the weak order.
In other words, the canonical complex of the weak order~$W_n$ is isomorphic to the semi-crossing arc complex on~$[n]$.
\end{theorem}


\subsection{Join irreducibles of the $\s$-weak order, $\s$-arcs, $\s$-shards}
\label{subsec:sArcs}

We first define the main characters of this paper, generalizing the arcs of~\cite{Reading-arcDiagrams}.

\begin{definition}
An \defn{$\s$-arc} is a quintuple~$(i, j, A, B, r)$ where~$1 \le i < j \le n$, the sets~$A$ and~$B$ form a partition of~$\set{k \in {]i,j[}}{s_k \ne 0}$, and~$r \in [s_i]$.
\end{definition}

Note that the first entry~$i$ of an $\s$-arc must satisfy~$s_i \ne 0$ (otherwise, we have no choice for~$r$).
We represent $\s$-arcs graphically as follows.
We place~$n$ points on the vertical axis, where the point at level~$i$ is round if~$s_i \ne 0$ and square if~$s_i = 0$.
The $\s$-arc~$(i, j, A, B, r)$ is represented by a curve wiggling around the vertical axis, starting from the point at level~$i$ and ending at the point at level~$j$, passing on the right of the points in~$A$ and on the left of the points in~$B$ (we do not care if it passes on the left or right of the square points), and with an additional label~$r$ written close~to~it.

\begin{remark}
The number of~$\s$-arcs is $\sum_{1 \le i < j \le n} s_i \, 2^{\# \set{k \in {]i,j[}}{s_k \ne 0}}$.
\end{remark}

We now associate two specific $\s$-trees to each $\s$-arc.

\pagebreak
\enlargethispage{.8cm}

\parpic(4cm,3cm)(2pt, 115pt)[r][b]{\includegraphics[scale=.8]{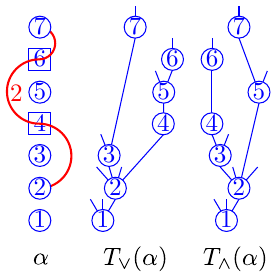}}{
\begin{definition}
\label{def:sArc2sTree}
For~$J \subseteq [n]$, the \defn{left} (resp.~\defn{right}) \defn{$\s$-comb} with nodes~$J$ is the increasing tree on~$J$, where each node~$j \in J$ has~$s_j+1$ outgoing edges, which are all leaves except the leftmost (resp.~rightmost).
For an $\s$-arc~$\alpha \eqdef (i, j, A, B, r)$, we define~$\tree_{\!\join}(\alpha)$ (resp.~$\tree_{\!\meet}(\alpha)$) as the $\s$-tree obtained by attaching the right (resp.~left) $\s$-comb with nodes~$A \cup \{j\}$ (resp.~$B \cup \{j\}$) to the $(r+1)$-st (resp.~$r$-th) rightmost leaf of~$i$ in the right (resp.~left) $\s$-comb with nodes~$[n] \ssm (A \cup \{j\})$ (resp.~$[n] \ssm (B \cup \{j\}$).
See examples of~$\tree_{\!\join}(\alpha)$ and~$\tree_{\!\meet}(\alpha)$ on the right for~$\s = (2, 3, 1, 0, 1, 0, 0)$ and~$\alpha = (2, 7, \{3\}, \{5\}, 2)$.
\end{definition}
}

\begin{lemma}
\label{lem:positionsArc}
For any~$\s$-arc~$\alpha \eqdef (i, j, A, B, r)$, and any~$1 \le k < \ell \le n$, we have
\begin{align*}
\pos(\tree_{\!\join}(\alpha), k, \ell) & = 
\begin{cases}
r & \text{if $k = i$ and $\ell \in A \cup \{j\}$} \\
s_k & \text{if $k \in B$ and $\ell \in A \cup \{j\}$} \\
0 & \text{otherwise}
\end{cases},
\\
\text{and}\qquad
\pos(\tree_{\!\meet}(\alpha), k, \ell) & = 
\begin{cases}
r-1 & \text{if $k = i$ and $\ell \in B \cup \{j\}$} \\
0 & \text{if $k \in A$ and $\ell \in B \cup \{j\}$} \\
s_k & \text{otherwise}
\end{cases}.
\end{align*}
\end{lemma}

\begin{proof}
Immediate from \cref{def:positions,def:sArc2sTree}.
\end{proof}

\begin{corollary}
\label{coro:comparisonTjoinTmeet}
For any two $\s$-arcs~$\alpha \eqdef (i, j, A, B, r)$ and~$\alpha' \eqdef (i', j', A', B', r')$, we have
\begin{enumerate}[(i)]
\item $\tree_{\!\join}(\alpha) \le \tree_{\!\join}(\alpha')$ if and only if~$A \cup \{j\} \subseteq A' \cup \{j'\}$ and~$B \cup \{i\} \subseteq B' \cup \{i'\}$ and~$i = i' \implies r \le r'$,
\item $\tree_{\!\meet}(\alpha) \le \tree_{\!\meet}(\alpha')$ if and only if~$A \cup \{i\} \supseteq A' \cup \{i'\}$ and~$B \cup \{j\} \supseteq B' \cup \{j'\}$ and~$i = i' \implies r < r'$,
\item $\tree_{\!\join}(\alpha) \le \tree_{\!\meet}(\alpha')$ if and only if there is no~$1 \le k < \ell \le n$ such that~$k \in (A' \cup \{i'\}) \cap (B \cup \{i\})$ and~$\ell \in (A \cup \{j\}) \cap (B' \cup \{j'\})$ except if~$k = i'$ and~$r < r'$.
\end{enumerate}
\end{corollary}

\begin{proof}
Straightforward from \cref{def:sWeakOrder,lem:positionsArc}.
\end{proof}

Recall that an element of a lattice~$(L, \le, \meet, \join)$ is \defn{join} (resp.~\defn{meet}) \defn{irreducible} if it covers (resp.~is covered by) a single element~of~$L$.

\pagebreak
\begin{proposition}
\label{prop:sArc2sTree}
The map~$\tree_{\!\join}$ (resp.~$\tree_{\!\meet}$) is a bijection from the $\s$-arcs to the join (resp.~meet) irreducible $\s$-trees in the $\s$-weak order.
\end{proposition}

\begin{proof}
Consider an $\s$-tree~$\tree$ that is a join irreducible in the $\s$-weak order.
By~\cref{prop:sWeakOrderCoverRelations}, this means that~$\tree$ has a unique descent~$(i,j)$.
This implies that~$\tree$ is a right $\s$-comb, to which we have attach on some leaf of node~$i$ a right $\s$-comb with top vertex~$j$.
Let~$r \eqdef \pos(\tree, i, j) \in [s_i]$, let~$A$ denote the set of nodes along the path from~$i$ to~$j$ (but distinct from~$i$ and~$j$), and let~$B$ be the complement of~$A$ in~$\set{k \in {]i,j[}}{s_k \ne 0}$.
Define the $\s$-arc~$\alpha_\join(\tree) \eqdef (i, j, A, B, r)$.
Then~$\tree_{\!\join}$ and~$\alpha_\join$ are obviously inverse to each other.
\end{proof}

We now provide the geometric counterpart of the $\s$-arcs, generalizing the shards of~\cite{Reading-posetRegions} (see also~\cite[Sect.~9.7]{Reading-posetRegionsChapter}).

\begin{definition}
\label{def:sShard}
The \defn{$\s$-shard} of an $\s$-arc~$\alpha \eqdef (i, j, A, B, r)$ is the polyhedron~$\shard$ of~$\R^n$ defined by
\begin{itemize}
\item the equality~$x_i - x_j = r - 1 + \sum_{k \in B} \max(0, s_k-1)$,
\item the inequalities~$x_i - x_a \ge r - 1 + \sum_{k \in B \cap {]i,a[}} \max(0, s_k-1)$ for all~$a \in A$, and
\item the inequalities~$x_i - x_b \le r - 1 + \sum_{k \in B \cap {]i,b[}} \max(0, s_k-1)$ for all~$b \in B$.
\end{itemize}
\end{definition}

\begin{lemma}
\label{lem:shard}
For any $\s$-arc~$\alpha \eqdef (i, j, A, B, r)$, the $\s$-shard~$\shard$ is the union of the closed fibers~$\closedFiber{\bush}$ over all $\s$-bushes~$\bush$ in which~$(i,j)$ is a hole and the rightmost path from~$i$ to the left incoming edge at~$j$ has the nodes of~$A$ weakly on its left, and the nodes of~$\bush$ and $r$ children of~$i$ strictly~on~its~right.
\end{lemma}

\begin{proof}
We prove both inclusions.

Consider an $\s$-bush~$\bush$ where~$(i,j)$ is a hole and the rightmost path from~$i$ to the left incoming edge at~$j$ has the nodes of~$A$ weakly on its left, and the nodes of~$\bush$ and $r$ children of~$i$ strictly on its right.
Consider~$\b{x}$ in the fiber~$\fiber{\bush}$.
As in the proof of \cref{prop:fibers}, we check that~$\b{x}$ must satisfy the equality $x_i - x_j = r - 1 + \sum_{k \in B} \max(0, s_k-1)$ since~$(i,j)$ is a hole of~$\bush$ and the inequalities ${x_i - x_a \ge r - 1 + \sum_{k \in B \cap {]i,a[}} \max(0, s_k-1)}$ for~$a \in A$ and~${x_i - x_b \le r - 1 + \sum_{k \in B \cap {]i,b[}} \max(0, s_k-1)}$ for~$b \in B$.

Conversely, consider~$\b{x} \in \shard$ and let~$\bush \eqdef \ins(\s, \b{x})$.
Follow the steps of the insertion algorithm described in \cref{def:insertion}, and construct inductively the rightmost path~$\pi$ leaving through the $(r+1)$-st leaf of node~$i$.
Since~$\b{x} \in \shard$, we obtain by induction that all nodes of~$A$ are weakly on the left of~$\pi$ while all nodes of~$B$ are strictly on its right, and finally that~$(i,j)$ is a hole of~$\bush$.
It follows from the description of the insertion algorithm in \cref{def:insertion} that~$\bush$ is an $\s$-bush where~$(i,j)$ is a hole and the rightmost path from~$i$ to the left incoming edge at~$j$ has the nodes of~$A$ weakly on its left, and the nodes of~$B$ and $r$ children of~$i$ strictly on its right.
\end{proof}

\begin{corollary}
The union of all codimension~$1$ closed fibers in~$\sFoam$ is precisely the union of the shards~$\shard$ for all $\s$-arcs~$\alpha$.
\end{corollary}


\subsection{Canonical join representations in the $\s$-weak order and non-crossing $\s$-arc diagrams}
\label{subsec:sArcDiagrams}

We now generalize \cref{thm:canonicalJoinComplexWeakOrder} to the $\s$-weak order, exploiting its semidistributivity established in \cref{thm:sWeakOrderLattice} and \cite[Thms.~1.21 \& 1.40]{CeballosPons-sWeakOrderI}.

\begin{definition}
\label{def:non-crossing}
Consider two $\s$-arcs~$\alpha \eqdef (i, j, A, B, r)$ and~$\alpha' \eqdef (i', j', A', B', r')$.
Assume without loss of generality that~$j \le j'$ (otherwise, exchange~$\alpha$ and~$\alpha'$).
Then~$\alpha$ and~$\alpha'$ are \defn{non-crossing} if~$j < j'$, and one of the following conditions hold
\begin{enumerate}
\item\label{non-crossing_item1} $j \le i'$,
\item\label{non-crossing_item2} $i < i' < j$ and~$i' \in A$ and~$j \notin A'$ and $A' \cap {]i,j[} \subseteq A \cap {]i',j'[}$,
\item\label{non-crossing_item3} $i < i' < j$ and~$i' \in B$ and~$j \notin B'$ and $A' \cap {]i,j[} \supseteq A \cap {]i',j'[}$,
\item\label{non-crossing_item4} $i = i'$ and~$r < r'$ and~$j \notin A'$ and $A' \cap {]i,j[} \subseteq A \cap {]i',j'[}$,
\item\label{non-crossing_item5} $i = i'$ and~$r = r'$ and~$s_j = 0$ and $A' \cap {]i,j[} = A \cap {]i',j'[}$,
\item\label{non-crossing_item6} $i = i'$ and~$r > r'$ and~$j \notin B'$ and $A' \cap {]i,j[} \supseteq A \cap {]i',j'[}$,
\item\label{non-crossing_item7} $i' < i$ and~$i \in A'$ and~$j \notin B'$ and $A' \cap {]i,j[} \supseteq A \cap {]i',j'[}$,
\item\label{non-crossing_item8} $i' < i$ and~$i \in B'$ and~$j \notin A'$ and $A' \cap {]i,j[} \subseteq A \cap {]i',j'[}$.
\end{enumerate}
A \defn{non-crossing $\s$-arc diagram} is a set~$\delta$ of pairwise non-crossing $\s$-arcs.
The \defn{non-crossing $\s$-arc complex} is the (flag) simplicial complex of non-crossing $\s$-arc diagrams.
\end{definition}

\pagebreak
\begin{lemma}
\label{lem:NonCrossingAreNotComparable}
If $\alpha$ and~$\alpha'$ are non-crossing, then $\tree_{\!\join}(\alpha)$ and $\tree_{\!\join}(\alpha')$ are incomparable in $\s$-weak~order.
\end{lemma}

\begin{proof}
Write $\alpha \eqdef (i, j, A, B, r)$ and~$\alpha' \eqdef (i', j', A', B', r')$, and assume that~$\tree_{\!\join}(\alpha) \leq \tree_{\!\join}(\alpha')$.
By \cref{coro:comparisonTjoinTmeet}\,(i), we have~$A \cup \{j\} \subseteq A' \cup \{j'\}$ and~$B \cup \{i\} \subseteq B' \cup \{i'\}$ and~$i = i' \implies r \le r'$.
The two inclusions imply in particular that~$i' \le i < j \le j'$.
If~$j = j'$, then~$\alpha$ and~$\alpha'$ are crossing.
Otherwise, $j < j'$ and the first inclusion implies~$j \in A'$.
If~$i' = i$, then~$r \le r'$, and since~$j \in A'$, we obtain by Condition~\eqref{non-crossing_item5} and~\eqref{non-crossing_item6} of \cref{def:non-crossing} that~$\alpha$ and~$\alpha'$ are crossing.
Otherwise, $i' < i$, and the second inclusion implies~$i \in B'$, which together with~$j \in A'$ implies by Condition~\eqref{non-crossing_item8} of \cref{def:non-crossing} that~$\alpha$ and~$\alpha'$ are crossing.
\end{proof}

\begin{definition}
\label{def:sTree2sArcDiagram}
For a descent~$(i,j)$ of an $\s$-tree~$\tree$, let~$r \eqdef \pos(\tree, i, j)$ and~$A$ (resp.~$B$) be the set of nodes~$i < k < j$ with~${s_k \ne 0}$ and weakly on the left (resp.~strictly on the right) of the path from~$i$ to~$j$.
Define~\defn{$\alpha_\join(\tree, i, j)$}$\eqdef (i, j, A, B, r)$ and~\defn{$\delta_\join(\tree)$}$\eqdef \set{\alpha_\join(\tree, i, j)}{(i,j) \text{ descent of } \tree}$.
Define similarly~\defn{$\alpha_\meet(\tree, i, j)$} for an ascent~$(i,j)$ of~$\tree$, and~\defn{$\delta_\meet(\tree)$}$\eqdef \set{\alpha_\meet(\tree, i, j)}{(i,j) \text{ ascent of } \tree}$.
\begin{figure}[t]
	\capstart
	\centerline{\raisebox{.5cm}{\includegraphics[scale=.6]{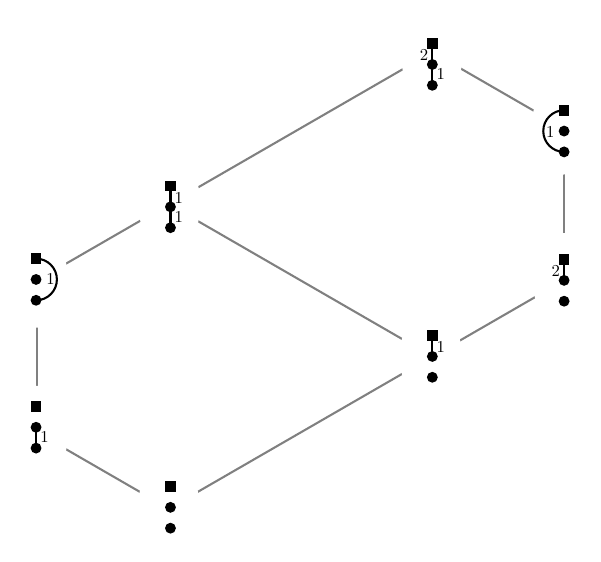}} \hspace{1.5cm} \includegraphics[scale=.6]{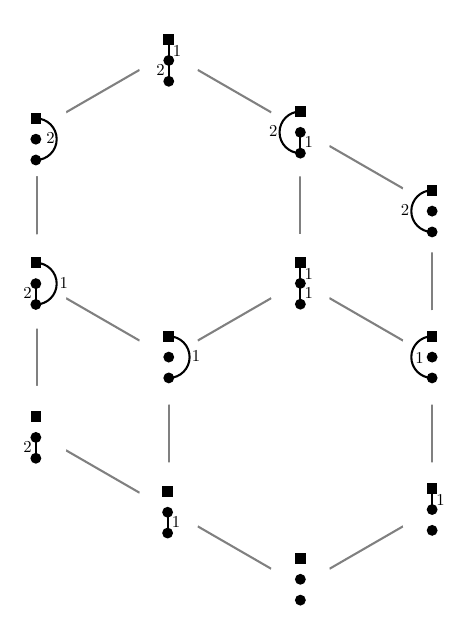}}
	\caption{The $\s$-weak order~$W_\s$ for $\s = (1,2,0)$ (left) and~$\s = (2,1,0)$ (right), where each $\s$-tree~$\tree$ is replaced by its non-crossing $\s$-arc diagram~$\delta_\join(\tree)$.}
	\label{fig:sWeakOrderArcDiagrams}
\end{figure}
\end{definition}

\begin{remark}
\label{rem:noncrossingsArcDiagram}
The maps~$\delta_\join$ and~$\delta_\meet$ have both a graphical and a geometric interpretation.
Namely, the non-crossing $\s$-arc diagram~$\delta_\join(\tree)$ (resp.~$\delta_\meet(\tree)$) of an $\s$-tree~$\tree$ is obtained \\[.2cm]
\noindent
\begin{minipage}[b]{11cm}
\begin{itemize}
\item  by drawing the path joining~$i$ to~$j$ in~$\tree$ for each descent (resp.~ascent)~$(i,j)$ of~$\tree$, perturbing all these paths so that they pass slightly to the right (resp.~left) of their intermediate nodes, and flattening the picture horizontally, allowing the arcs to bend but not to cross nor to pass a node, see on the right,
\item as the set of all $\s$-arcs~$\alpha$ whose corresponding $\s$-shard~$\shard$ supports a lower facet of the maximal cell~$\closedFiber{\tree}$ of the $\s$-foam corresponding to~$\tree$.
\end{itemize}
\end{minipage}
\quad
\includegraphics[scale=.75]{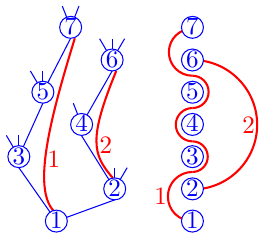}
\end{remark}

The following statement is illustrated by \cref{fig:sWeakOrderArcDiagrams}.

\begin{proposition}
\label{prop:bij_sTree2sArcDiagram}
The maps~$\delta_\join$ and~$\delta_\meet$ are bijections from $\s$-trees to non-crossing $\s$-arc diagrams.
\end{proposition}

\begin{proof}
By symmetry, we only prove it for~$\delta_\join$.
First, for a descent~$(i,j)$ of an $\s$-tree~$\tree$, $\alpha_\join(\tree, i, j)$ is indeed an $\s$-arc, and it is not difficult to see that the $\s$-arcs of~$\delta_\join(\tree)$ are non-crossing by \cref{rem:noncrossingsArcDiagram}.
We now define the inverse map of~$\delta_\join$.
Consider a non-crossing $\s$-arc diagram~$\delta$. 
We construct an $\s$-tree~$\tree_{\!\join}(\delta)$ inductively as follows:
\begin{itemize}
\item start with the rooted tree with just a root~$1$ and~$s_1+1$ leaves,
\item at step $k\geq 2$, attach a node $k$ with $s_k+1$ leaves to the rightmost leaf of the $(r+1)$-st subtree of node $i$ where $i$ and $r$ are defined by:
\begin{itemize}
\item if there exists an $\s$-arc $(i',k,A,B,r')$ in $\delta$, then $i \eqdef i'$ and $r \eqdef r'$ (such an arc is necessarily unique),
\item otherwise, if there exists an $\s$-arc $(i',j,A,B,r')$ in $\delta$ with $k\in A$, then
\begin{align*}
i & \eqdef \max\set{i'}{\exists \; (i',j,A,B,r')\in \delta\text{ with }k\in A} \\
\text{and} \qquad r & \eqdef \max\set{r'}{\exists \; (i,j,A,B,r')\in \delta\text{ with }k\in A},
\end{align*}
\item otherwise $i \eqdef 1$ and $r \eqdef 0$.
\end{itemize}
\end{itemize}
Let us check that $\delta_\join(\tree_{\!\join}(\delta))=\delta$. 
It is clear that any arc $(i,j,A,B,r)$ of $\delta$ will correspond to a descent $(i,j)$ in $\tree_{\!\join}$, created at step $j$.
Reciprocally, let $(i,j)$ be a descent of $\tree_{\!\join}(\delta)$. 
The construction of $\tree_{\!\join}(\delta)$ implies that there is an arc $(i,j,A,B,r)$ in $\delta$ with $\pos(\tree_{\!\join}(\delta),i,j)=r$. 
We show by induction on $k\in {]i,j[}$ and $i$ that $A$ (resp. $B$) corresponds to the nodes of $\tree$ weakly on the left (resp. strictly on the right) of the path from $i$ to $j$ in $\tree_{\!\join}(\delta)$.

Suppose that $k\in A$. 
During the construction of $\tree_{\!\join}(\delta)$, the node $k$ was attached to the rightmost leaf  of the $(r'+1)$-st subtree of a node $i'$ with $i'>i$ or $i'=i$ and $r'\geq r$. 
This implies that $k$ is weakly on the left of the path from $i$ to $j$. 
Indeed, this is clear if $i'=i$ and $r'\geq r$. 
Suppose that~$i'>i$. 
Then we are in the case~(\ref{non-crossing_item7}) of \cref{def:non-crossing} since $i<i'<k<j$ and $k\in A$. 
Thus~$i'\in A$ and we conclude by induction.

Now, suppose that $k\in B$. 
During the construction of $\tree_{\!\join}(\delta)$, the node $k$ was attached either to the rightmost available leaf at step $k$, in which case it is clear that $k$ is strictly on the right of the path from $i$ to $j$ in $\tree_{\!\join}(\delta)$, or to the rightmost leaf  of the $(r'+1)$-st subtree of a node $i'$ such that there is an $\s$-arc $(i',j',A',B',r')$ in $\delta$, with $k\in A'\cup\{j'\}$. 
In this case, assume that $j<j'$. 
Then the only possible cases of \cref{def:non-crossing} are (\ref{non-crossing_item3}), which implies $i'\in B$, (\ref{non-crossing_item6}), which implies $i=i'$ and $r'<r$, or (\ref{non-crossing_item7}), which implies $i\in A'$. 
In all these cases we can conclude either directly or by induction that $k$ is strictly on the right of the path from $i$ to $j$ in $\tree_{\!\join}(\delta)$.
Similarly, if $j'<j$ the only possible cases of \cref{def:non-crossing} are (\ref{non-crossing_item2}), (\ref{non-crossing_item4}), or (\ref{non-crossing_item8}) and we arrive to the same conclusion.

We have proven that $\delta_\join\circ \tree_{\!\join}$ is the identity function on the set of non-crossing $\s$-arc diagrams.

To finish the proof of the bijection, we show that $\delta_\join$ is injective.
Let $\tree$ and $\tree'$ be two $\s$-trees and denote by $j$ the first step where their inductive construction differs.
Up to exchanging $\tree$ and~$\tree'$ we can assume that the leaf of $\tree_{\leq j-1}$ where $j$ is attached to build $\tree$ is on the left of the leaf where $j$ is attached to build $\tree'$. 
In particular $j$ is not on the rightmost path from the root in~$\tree$.
If $s_j=0$, then $j$ forms an ascent $(i,j)$ in $\tree$, which gives an $\s$-arc $\alpha_\join(\tree,i,j)$ that is in $\delta_\join(\tree)$ but not in $\delta_\join(\tree')$. 
Assume that $s_j\neq 0$. 
Then there is an $\s$-arc $(i,k,A,B,r)$ with $j\in A$ that is in $\delta_\join(\tree)$ but not in $\delta_\join(\tree')$ (we take $i$ and $k$ to be the smallest and greatest nodes such that $j$ is on the rightmost path from $i$ to $k$ in $\tree$).
This proves that $\delta_\join$ is injective and we have the desired bijection.
\end{proof}

\begin{lemma}
\label{lem:prepCanonicalJoinRepresentation}
For a descent~$(i,j)$ of an $\s$-tree~$\tree$, the tree~$\tree_{\!\join}(\alpha_\join(\tree, i, j))$ is the unique minimal element of the set~$\set{\tree' \le \tree}{\pos(\tree', i, j) = \pos(\tree, i, j)}$.
\end{lemma}

\begin{proof}
Let~$(i, j, A, B, r) \eqdef \alpha_\join(\tree, i, j)$.
Note that~$\pos(\tree, i, a) \ge r$ and~$\pos(\tree, a, j) = 0$ for all~$a \in A$, while~$\pos(\tree, i, b) < r$ and $\pos(\tree, b, j) = s_b$ for all~$b \in B$.
For any~$\s$-tree~$\tree' \le \tree$ with~${\pos(\tree', i, j) = r}$,
\begin{itemize}
\item for~$a \in A$, we have~$\pos(\tree', a, j) \le \pos(\tree, a, j) = 0 < s_a$ (because~$s_a \ne 0$), so that we get~$\pos(\tree', i, a) \ge \pos(\tree', i, j) = r$ (by \cref{lem:propertiesPositions}),
\item for~$a \in A \cup \{j\}$ and~$b \in B$ with~$b < a$, we have~$\pos(\tree', i, b) \le \pos(\tree, i, b) < r \le \pos(\tree', i, a)$, so that~$\pos(\tree', b, a) = s_b$ (by \cref{lem:propertiesPositions}).
\end{itemize}
We conclude that~$\pos(\tree', k, \ell) \ge \pos(\tree_{\!\join}(\alpha_\join(\tree, i, j)), k, \ell)$ for all~$1 \le k < \ell \le n$ by~\cref{lem:positionsArc}, so that~$\tree' \ge \tree_{\!\join}(\alpha_\join(\tree, i, j))$.
This concludes the proof since~${\pos(\tree_{\!\join}(\alpha_\join(\tree, i, j)), i, j) = r = \pos(\tree, i, j)}$ by \cref{def:sArc2sTree,def:sTree2sArcDiagram}, and~$\tree \ge \tree_{\!\join}(\alpha_\join(\tree, i, j))$ by setting~$\tree' = \tree$.
\end{proof}

\begin{proposition}
\label{prop:canonicalJoinRepresentation}
The canonical join (resp.~meet) representation of an~$\s$-tree~$\tree$ is~$\tree = \bigJoin_{\alpha \in \delta_\join(\tree)} \tree_{\!\join}(\alpha)$ (resp.~$\tree = \bigMeet_{\alpha \in \delta_\meet(\tree)} \tree_{\!\meet}(\alpha)$).
In other words, the map~$\alpha_\join$ (resp.~$\alpha_\meet$) induces an isomorphism from the canonical join (resp.~meet) complex of the $\s$-weak order~$W_\s$ to the non-crossing $\s$-arc complex.
\end{proposition}

\begin{proof}
We only prove the result for the canonical join representation.

We first show that~$\tree = \bigJoin_{\alpha \in \delta_\join(\tree)} \tree_{\!\join}(\alpha)$.
Indeed, \cref{lem:prepCanonicalJoinRepresentation} implies that~$\tree \ge \tree_{\!\join}(\alpha_\join(\tree, i, j))$ for any descent~$(i,j)$ of~$\tree$, so that~$\tree \ge \bigJoin_{\alpha \in \delta_\join(\tree)} \tree_{\!\join}(\alpha)$.
Moreover, as any~$\tree'$ covered by~$\tree$ is obtained by rotating a descent~$(i,j)$ of~$\tree$, we have~${\pos(\tree', i, j) = \pos(\tree, i, j) - 1 < \pos(\tree_{\!\join}(\alpha_\join(\tree, i, j)),i,j)}$, so that~$\tree' \not\ge \bigJoin_{\alpha \in \delta_\join(\tree)} \tree_{\!\join}(\alpha)$. 

It follows from~\cref{lem:NonCrossingAreNotComparable} and~\cref{prop:bij_sTree2sArcDiagram} that the $\s$-trees~$\tree_{\!\join}(\alpha)$ for $\alpha\in\delta_\join(\tree)$ form an antichain in the $\s$-weak order, thus the join representation $\tree = \bigJoin_{\alpha \in \delta_\join(\tree)} \tree_{\!\join}(\alpha)$ is irredundant.

Consider now any join representation~$\tree = \bigJoin \c{J}$, and assume by contradiction that there is a descent~$(i,j)$ of~$\tree$ such that there is no~$\tree[J] \in \c{J}$ with~$\tree_{\!\join}(\alpha_\join(\tree, i, j)) \le \tree[J]$.
Let~$\tree'$ be the $\s$-tree obtained by rotating the descent~$(i,j)$ of~$\tree$.
For all~$\tree[J] \in \c{J}$, we thus obtain that~$\pos(\tree[J], i, j) < \pos(\tree, i, j)$ by \cref{lem:prepCanonicalJoinRepresentation}, so that~$\tree[J] \le \tree'$ by \cref{prop:positionsRotation}.
Hence, we conclude that~${\bigJoin \c{J'} \le \tree' < \tree}$, a contradiction.
\end{proof}

For instance, \cref{fig:sWeakOrderCanonicalComplex} contains the canonical join (red) and meet (blue) complexes of the $\s$-weak order for~$\s = (1,2,0)$ and~$\s = (2,1,0)$.


\subsection{Canonical complex of the $\s$-weak order and semi-crossing $\s$-arc bidiagrams}
\label{subsec:canonicalComplex}

We finally generalize \cref{thm:canonicalComplexWeakOrder} to the $\s$-weak order.
This section is not required in the sequel.

\begin{definition}
A \defn{semi-crossing $\s$-arc bidiagram} is a disjoint union~$\delta_\join \sqcup \delta_\meet$ of non-crossing $\s$-arc diagrams such that for any~$\alpha_\join \eqdef (i_\join, j_\join, A_\join, B_\join, r_\join) \in \delta_\join$ and~$\alpha_\meet \eqdef (i_\meet, j_\meet, A_\meet, B_\meet, r_\meet)$,  there is no~$1 \le k < \ell \le n$ such that~$k \in (A_\meet \cup \{i_\meet\}) \cap (B_\join \cup \{i_\join\})$ and~$\ell \in (A_\join \cup \{j_\join\}) \cap (B_\meet \cup \{j_\meet\})$ except if~$k = i_\meet$ and~$r_\join < r_\meet$.
The \defn{semi-crossing $\s$-arc complex} is the (flag) simplicial complex whose ground set contains two copies~$\alpha_\join$ and~$\alpha_\meet$ of each $\s$-arc~$\alpha$ and whose simplices are all semi-crossing $\s$-arc bidiagrams.
\end{definition}

\begin{proposition}
\label{prop:canonicalRepresentation}
The map~$[\tree,\tree'] \to \delta_\join(\tree) \sqcup \delta_\meet(\tree')$ is a bijection from intervals of the~$\s$-weak order to semi-crossing $\s$-arc diagrams.
The canonical representation of an interval~$[\tree,\tree']$ in the~$\s$-weak order is given by $\bigJoin_{\alpha \in \delta_\join(\tree)} \tree_{\!\join}(\alpha) \sqcup \bigMeet_{\alpha \in \delta_\meet(\tree')} \tree_{\!\meet}(\alpha)$.
In other words, the canonical complex of the $\s$-weak order~$W_\s$ is isomorphic to the semi-crossing $\s$-arc complex.
\end{proposition}

\begin{proof}
According to \cref{prop:canonicalJoinRepresentation,coro:comparisonTjoinTmeet}\,(iii), this is a direct application of the ideas of~\cite{AlbertinPilaud}.
\end{proof}

See \cref{fig:sWeakOrderCanonicalComplex} for illustrations.

\begin{figure}[t]
	\capstart
	\centerline{\raisebox{.2cm}{\includegraphics[scale=.8]{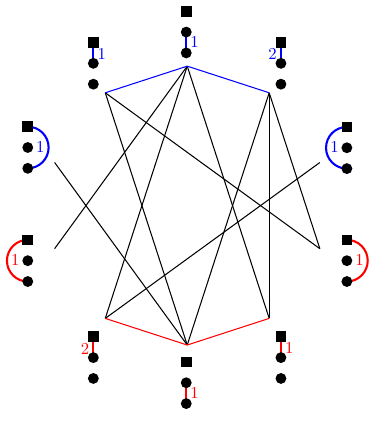}} \hspace{1cm} \includegraphics[scale=.8]{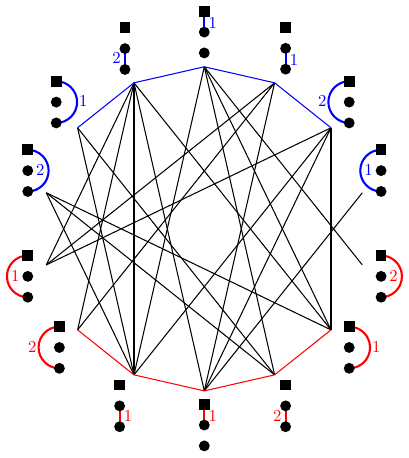}}
	\caption{The canonical complex of the $\s$-weak order~$W_\s$ for $\s = (1,2,0)$ (left) and~$\s = (2,1,0)$ (right), interpreted as the semi-crossing arc complex. The canonical join and meet complexes of the $\s$-weak order, interpreted as non-crossing arc complexes, appear in red and blue respectively.}
	\label{fig:sWeakOrderCanonicalComplex}
\end{figure}


\section{Lattice congruences of the $\s$-weak order}
\label{sec:quotients}

In this section, we describe the combinatorics of the congruence lattice of the $\s$-weak order, and we explore a few relevant families of these congruences.


\subsection{Recollections~\ref{sec:quotients}: Lattice congruences of the weak order}
\label{subsec:recollectionsQuotients}

We first recall the combinatorics of lattice congruences of the weak order.
We refer to~\cite{Reading-posetRegionsChapter} for an enlightening survey on the topic.

\subsubsection{Lattice congruences and quotients}

Consider a finite lattice~$(L, \le, \meet, \join)$.
A \defn{congruence}~$\equiv$ on~$L$ is an equivalence relation on~$L$ that respects the lattice operations, \ie such that $x \equiv x'$ and $y \equiv y'$ implies $x \join y \equiv x' \join y'$ and $x \meet y \equiv x' \meet y'$.
Equivalently, the equivalence classes are intervals, and the maps~$\projDown[\equiv]$ and~$\projUp[\equiv]$ sending an element to the minimum and maximum elements in its congruence class are order preserving.
The \defn{lattice quotient}~$L/{\equiv}$ is the lattice structure on the congruence classes of~$\equiv$, where for any two congruence classes~$X$ and~$Y$,  the order is given by~$X \le Y$ if and only if $x \le y$ for some representatives~$x \in X$ and~$y \in Y$, and the join~$X \join Y$ (resp.~meet~$X \meet Y$) is the congruence class of~$x \join y$ (resp.~$x \meet y$) for any representatives~$x \in X$ and~$y \in Y$.
Intuitively, the quotient~$L/{\equiv}$ is obtained by contracting the equivalence classes of~$\equiv$ in the lattice~$L$.
More precisely, we say that an element~$x$ is \defn{contracted} by~$\equiv$ if it is not minimal in its equivalence class of~$\equiv$, \ie if~$x \ne \projDown(x)$.
As each class of~$\equiv$ is an interval of~$L$, it contains a unique uncontracted element, and the quotient~$L/{\equiv}$ is isomorphic to the subposet of~$L$ induced by its uncontracted elements~$\projDown(L)$.

\smallskip
The prototypical example of congruence of the weak order~$W_n$ is the \defn{sylvester congruence}~$\equiv_\textrm{sylv}$ \cite{Tonks, LodayRonco, HivertNovelliThibon-algebraBinarySearchTrees}.
Its congruence classes are the fibers of the binary tree insertion algorithm, or equivalently the sets of linear extensions of binary trees (labeled in inorder and considered as posets oriented from bottom to top).
It can also be seen as the transitive closure of the rewriting rule~$U u w V v W \equiv_\textrm{sylv} U w u V v W$ where~$u < v < w$ are letters and~$U,V,W$ are words on~$[n]$.
The uncontracted permutations for~$\equiv_\textrm{sylv}$ are those avoiding the pattern~$312$.
The quotient~$W_n / {\equiv_\textrm{sylv}}$ is (isomorphic to) the classical \defn{Tamari lattice}~\cite{Tamari}, whose elements~are the binary trees on~$n$ nodes and whose cover relations are right rotations in binary trees.
See~\cref{fig:sylvesterCongruence}.
\begin{figure}[b]
	\capstart
	\centerline{\includegraphics[scale=.6]{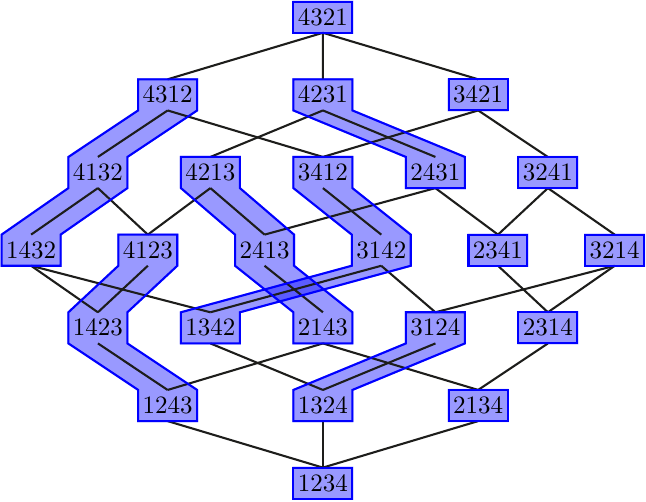} \; \includegraphics[scale=.48]{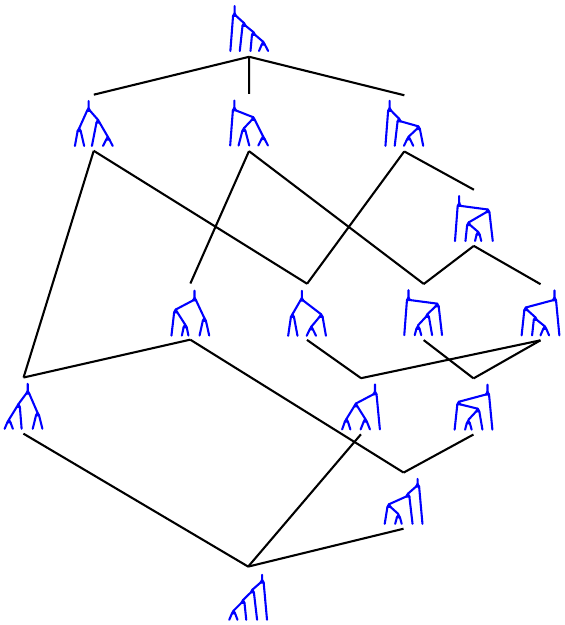}}
	\caption{The sylvester congruence~$\equiv_\textrm{sylv}$~(left), and the Tamari lattice (right). \cite[Fig.~2]{PilaudSantos-quotientopes}}
	\label{fig:sylvesterCongruence}
\end{figure}

\subsubsection{Canonical representations in lattice quotients}

\enlargethispage{.4cm}
If a lattice~$L$ is semidistributive, then any lattice quotient~$L/{\equiv}$ is also semidistributive.
Moreover, via the identification between congruence classes of~$\equiv$ and their minimal elements, the canonical join representations in the quotient~$L/{\equiv} \simeq \projDown(L)$ are precisely the canonical join representations of~$L$ that only involve join-irreducibles of~$L$ uncontracted by~$\equiv$.
In other words, the canonical join complex of the quotient~$L/{\equiv}$ is isomorphic to the subcomplex of the canonical complex of~$L$ induced by the join-irreducibles of~$L$ uncontracted~by~$\equiv$.

\smallskip
Recall from~\cref{subsec:recollectionsCanonicalComplex} that the weak order~$W_n$ is semidistributive, the join-irreducible permutations correspond to arcs on~$[n]$, and the canonical join representations of permutations correspond to non-crossing arc diagrams.
This yields the following statement.

\begin{theorem}[{\cite[Thm.~4.1]{Reading-arcDiagrams}}]
\label{thm:joinMeetRepresentationsQuotient}
For any lattice congruence~$\equiv$ of the weak order~$W_n$, the set of join-irreducibles of~$W_n$ uncontracted by~$\equiv$ corresponds to a set of arcs~$\c{A}_\equiv$, and the canonical join representations in the lattice quotient~$W_n/{\equiv}$ correspond to non-crossing arc diagrams using only arcs of~$\c{A}_\equiv$.
\end{theorem}

For instance, the uncontracted join-irreducibles of the sylvester congruence~$\equiv_\textrm{sylv}$ are given by the set~$\c{A}_\textrm{sylv} = \set{(i, j, {]i, j[}, \varnothing)}{1 \le i < j \le n}$ of right arcs, \ie those which pass on the right of all dots in between their endpoints.
Therefore, the sylvester congruence classes are in bijection with non-crossing arc diagrams with arcs in~$\c{A}_\textrm{sylv}$, also known as non-crossing partitions.

\subsubsection{The congruence lattice and the forcing relation}

The \defn{congruence lattice}~$\con(L)$ is the set of all congruences of~$L$ ordered by refinement.
It is a lattice where the meet is the intersection of relations and the join is the transitive closure of union of relations.
Consider a join irreducible element~$j$ of~$L$, and let~$j_\star$ be the unique element covered by~$j$.
We say that~$\equiv$ \defn{contracts}~$j$ if~$j \equiv j_\star$ and we denote by~$\con(j)$ the unique minimal congruence of~$L$ that contracts~$j$.
It turns out that~$\con(j)$ is join irreducible in~$\con(L)$ and that all join irreducible congruences in~$\con(L)$ are of this form.
Hence, any congruence of~$L$ is completely determined by the subset~$J_\equiv$ of join irreducible elements of~$L$ that it contracts.
For~$j$ and~$j'$ join irreducible in~$L$, we say that~$j$ \defn{forces}~$j'$, and write~$j \succcurlyeq j'$, if $\con(j) \ge \con(j')$, that is, if any congruence contracting~$j$ also contracts~$j'$.
The relation~$\succcurlyeq$ is a preorder (\ie a transitive and reflexive, but not necessarily antisymmetric, relation) on the join irreducible elements of~$L$.
Moreover, the down sets of~$\succcurlyeq$ (\ie the subsets~$J$ such that~$j \succcurlyeq j'$ and~$j \in J$ implies~$j' \in J$) are precisely the subsets~$J_\equiv$ of join irreducible elements of~$L$ that are contracted by some congruence~$\equiv$ on~$L$.
We sum up with the following statement.

\begin{theorem}[{\cite[Prop.~9.5.16]{Reading-posetRegionsChapter}}]
\label{thm:congruenceLattice}
The congruence lattice~$\con(L)$ is isomorphic to the lattice of down sets of the forcing relation~$\succcurlyeq$ (hence, it is a distributive lattice).
\end{theorem}

When~$\succcurlyeq$ is a poset (meaning antisymmetric), then there is a bijection between its join irreducible elements and its join irreducible congruences.
The lattice~$L$ is \defn{congruence uniform} if and only if it satisfies this property and the dual property (\ie the meet irreducible elements are in bijection with the meet irreducible congruences).
As already mentioned in \cref{subsec:recollectionsWeakOrder}, this is equivalent to constructibility by interval doublings, and it implies semidistributivity.

\smallskip
The weak order~$W_n$ is a congruence uniform lattice, and the forcing order on join-irreducible permutations can be described visually on arcs as follows.
We say that an arc~$\alpha \eqdef (i, j, A, B)$ is a subarc of an arc~$\alpha' \eqdef (i', j', A', B')$, if~$i' \le i < j \le j'$ and~${A \subseteq A'}$ and~${B \subseteq B'}$.
Visually, $\alpha$ is a subarc of~$\alpha'$ if the endpoints of~$\alpha$ are located in between those of~$\alpha'$ and~$\alpha$ agrees with~$\alpha'$ in between its endpoints (meaning, $\alpha$ and~$\alpha'$ pass on the left and on the right of the same points in between the endpoints of~$\alpha$).
See \cref{fig:subarcOrder}.
Then $\alpha$ is a subarc of~$\alpha'$ if and only if the join irreducible permutation corresponding to~$\alpha$ forces the join irreducible permutation corresponding to~$\alpha'$.
We thus obtain the following description of the lattice congruences of the weak order~$W_n$.

\begin{figure}
	\capstart
	\centerline{\includegraphics[scale=.5,valign=c]{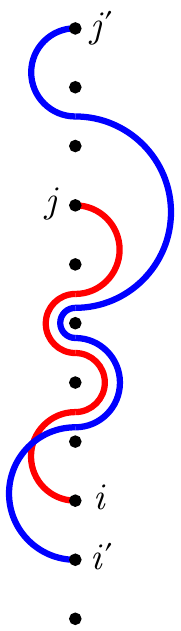} \hspace{2cm} \includegraphics[scale=.6,valign=c]{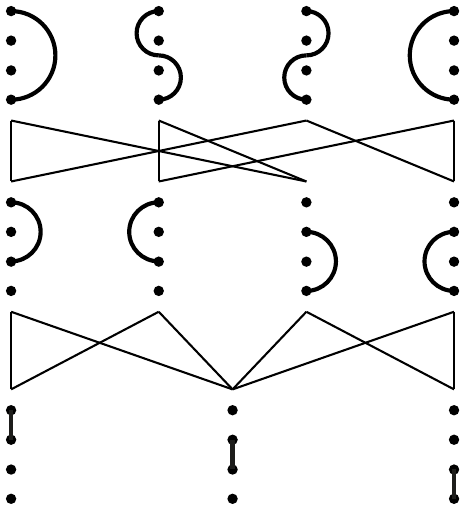}}
	\caption{The subarc relation (left) and the subarc poset for~$n = 4$ (right). The red arc~$(i,j,A,B)$ is a subarc of the blue arc~$(i',j',A',B')$. \mbox{\cite[Fig.~5]{PilaudSantos-quotientopes}}}
	\label{fig:subarcOrder}
\end{figure}

\begin{theorem}[{\cite[Thm.~4.4 \& Coro.~4.5]{Reading-arcDiagrams}}]
\label{thm:arcIdeals}
The map~${\equiv} \mapsto \c{A}_\equiv$ is a bijection between the lattice congruences of the weak order and the down sets of the subarc poset.
\end{theorem}

For instance, we have represented in \cref{fig:weakOrderCongruenceLattice} the upper set of the congruence lattice~$\con(W_4)$ generated by the recoil congruence (these are precisely the congruences whose quotient fan is essential, see \cite{PilaudSantos-quotientopes} and~\cref{fig:weakOrderQuotientopeLattice}).

\subsubsection{Forcing in polygonal lattices}

A \defn{polygon} in~$L$ is an interval~$[x,y]$ that is the union of two maximal chains joining~$x$ to~$y$, which are disjoint except at~$x$ and~$y$.
The \defn{edges} are the order relations appearing in the polygon.
The two edges incident to $x$ are called the \defn{bottom edges} of the polygon, the two edges incident to $y$ are called the \defn{top edges} of the polygon, and the other edges are called the \defn{side edges} of the polygon.
The lattice~$L$ is \defn{polygonal} if
\begin{itemize}
\item for any~$y, y'$ covering the same element~$x$, the interval~$[x, y \join y']$ is a polygon, and
\item for any~$x, x'$ covered by the same element~$y$, the interval~$[x \meet x', y]$ is a polygon.
\end{itemize}
Polygonal lattices behave particularly nicely with respect to congruence relations.
To be precise, let us define two notions of forcing on all cover relations of~$L$ (not only on the relations $j_\star \iscovered j$ as before).
We say that a congruence~$\equiv$ of~$L$ \defn{contracts} a cover relation~$x \iscovered y$ in~$L$ if~$x \equiv y$.
We say that~$x \iscovered y$ \defn{forces} $x' \iscovered y'$ if any congruence contracting~$x \iscovered y$ also contracts~$x' \iscovered y'$.
We say that $x \iscovered y$ \defn{forces} $x' \iscovered y'$ \defn{in a polygon} if there is some polygon in~$L$ containing $x \iscovered y$ and $x' \iscovered y'$ such that one of the following holds:
\begin{enumerate}
\item $x \iscovered y$ is a bottom edge of the polygon and $x' \iscovered y'$ is the opposite top edge,
\item $x \iscovered y$ is a bottom edge of the polygon and $x' \iscovered y'$ is a side edge,
\item $x \iscovered y$ is a top edge of the polygon and $x' \iscovered y'$ is the opposite bottom edge,
\item $x \iscovered y$ is a top edge of the polygon and $x' \iscovered y'$ is a side edge.
\end{enumerate}
See \cref{fig:forcingPolygon}.
Note that these are the only forcing relations in a polygon, and the following holds.

\begin{figure}
	\capstart
	\centerline{\includegraphics[scale=1]{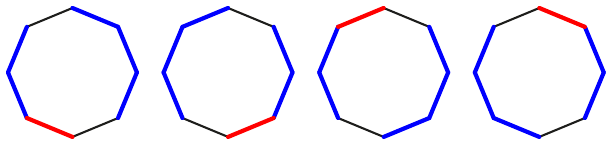}}
	\caption{Forcing in a polygon. In each picture, the red edge forces all the blue edges.}
	\label{fig:forcingPolygon}
\end{figure}

\begin{theorem}[{\cite[Thm.~9-6.5]{Reading-posetRegionsChapter}}]
\label{thm:edge_forcing_polygon}
If $L$ is a finite polygonal lattice, then the forcing relation (on cover relations of~$L$) is the transitive closure of the forcing in polygons relation. 
\end{theorem}

It is known that the weak order~$W_n$ is polygonal~\cite{CaspardContePolyBarbutMorvan}.
More generally, the polygonality of posets of regions of hyperplane arrangements was studied in details in~\cite[Sect.~9-6]{Reading-posetRegionsChapter}.

\begin{figure}[p]
	\capstart
	\centerline{\includegraphics[scale=.6]{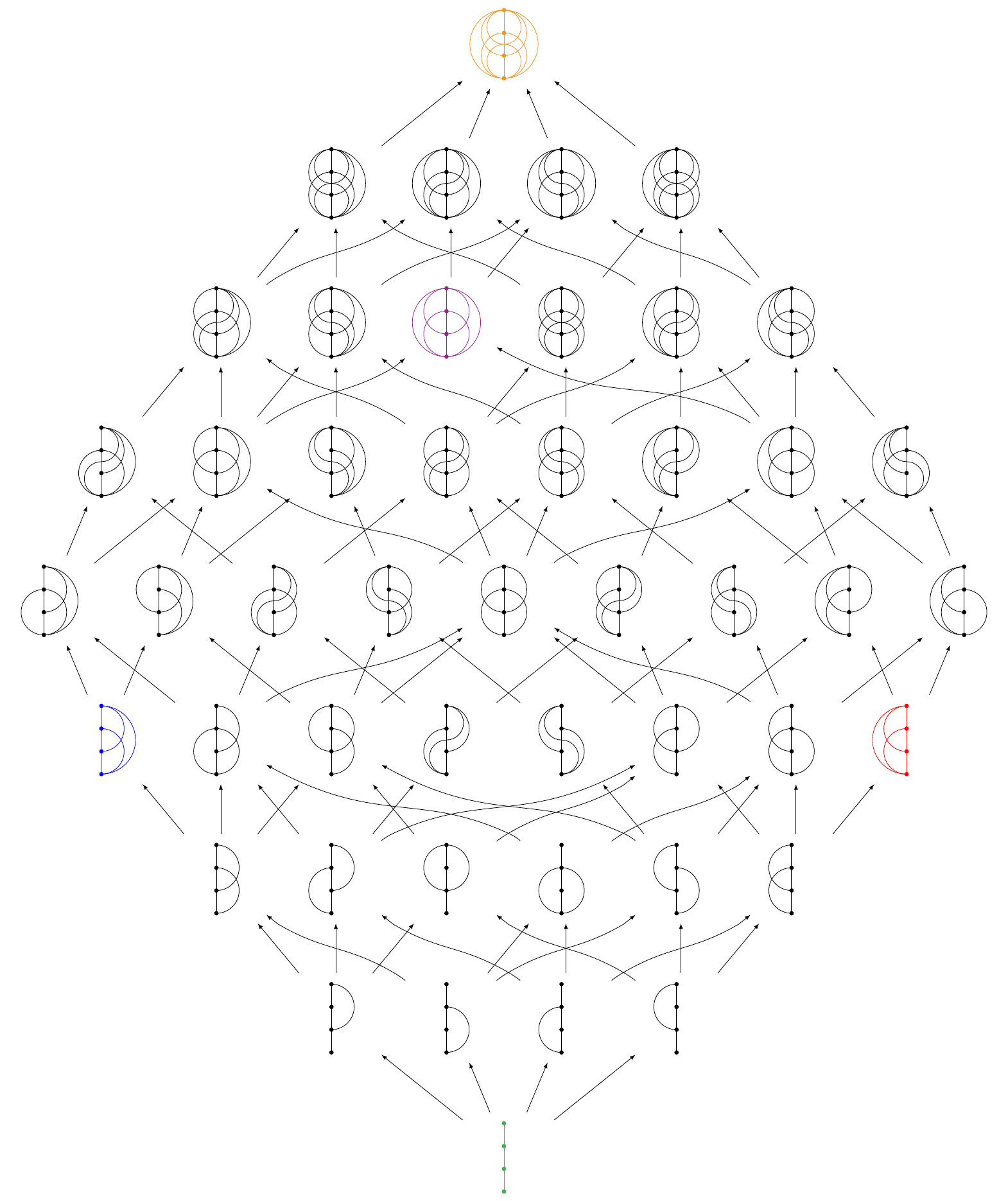}}
	\caption{The congruence lattice of the weak order~$W_4$, where each congruence~$\equiv$ is replaced by its down set~$\c{A}_\equiv$. We have colored in green / blue / red / purple / orange the recoil / sylvester / anti-sylvester / Baxter / trivial (also generic rectangulation) congruence. See also~\cref{fig:weakOrderQuotientopeLattice}. Adapted from~\cite[Fig.~6]{PilaudSantos-quotientopes}}
	\label{fig:weakOrderCongruenceLattice}
\end{figure}

\subsubsection{Special congruences of the weak order}
\label{subsubsec:specialCongruences}

We conclude this recollection with some particularly relevant congruences of the weak order~$W_n$:
\begin{enumerate}
\item the \defn{recoil congruence}~$\equiv_\textrm{rec}$ is defined by the down set~$\c{A}_\textrm{rec} = \set{(i, i+1, \varnothing, \varnothing)}{i \in [n-1]}$ of basic arcs. It has a congruence class for each subset~$I \subseteq [n-1]$ given by the permutations whose recoils (descents of the inverse) are at positions in~$I$. It can also be seen as the transitive closure of the rewriting rule~$U u v V \equiv_{\textrm{rec}} U v u V$ for~$|u - v| > 1$. The quotient~$W_n/{\equiv_\textrm{rec}}$ is the boolean lattice.
\label{item:recoilCongruence}

\item for an arc~$\alpha \eqdef (i, j, A, B)$, the \defn{$\alpha$-Cambrian congruence}~$\equiv_\alpha$ is defined by the down set of subarcs of~$\alpha$.
It was introduced by N.~Reading~\cite{Reading-CambrianLattices} as a generalization of the sylvester congruence, obtained for~$\alpha = (1, n, {]1,n[}, \varnothing)$.
The $\alpha$-Cambrian congruence classes are fibers of the $\alpha$-Cambrian tree insertion, or equivalently linear extensions of $\alpha$-Cambrian trees, see~\cite{LangePilaud, ChatelPilaud, PilaudPons-permutrees}.
An \defn{$\alpha$-Cambrian tree} is a tree on~$[i,j]$ such that the node~$b \in \{i\} \cup A$ (resp.~$b \in B \cup \{j\}$) has one ancestor (resp.~descendant) subtree and two descendant (resp.~ancestor) subtrees, and~$a < b < c$ for any nodes $a$ in the left descendant (resp.~ancestor) subtree of~$b$ and~$c$ in the right descendant (resp.~ancestor) subtree of~$b$.
The $\alpha$-Cambrian congruence can also be seen as the transitive closure of the three rewriting rules~${U u v V \equiv_\alpha U v u V}$ for~$u < i$ or~$v > j$, $U u w V v W \equiv_\alpha U w u V v W$ for~$u < v < w$ with~$v \in A$, and~$U v V u w W \equiv_\alpha U v V w u W$ for~$u < v < w$ with~$v \in B$.
\label{item:CambrianCongruence}

\item for~$\decoration \in \Decorations^n$, the \defn{$\decoration$-permutree congruence}~$\equiv_\decoration$ is defined by the down set~$\c{A}_\decoration$ of arcs that do not pass on the right the points~$j$ with~$\decoration_j \in \{\upCirc, \upDownCirc\}$ nor on the left of the points~$j$ with~$\decoration_j \in \{\downCirc, \upDownCirc\}$.
Its congruence classes correspond to $\decoration$-permutrees~\cite{PilaudPons-permutrees}.
It can also be seen as the transitive closure of the rewriting rules~$U u w V v W \equiv_\decoration U w u V v W$ for~$u < v < w$ with~$\decoration_v \in \{\downCirc, \upDownCirc\}$ and~$U v V u w W \equiv_\decoration U v V w u W$ for~$u < v < w$ with~${\decoration_v \in \{\upCirc, \upDownCirc\}}$.
\label{item:permutreeCongruence}

\item the \defn{Baxter congruence}~$\equiv_\textrm{Bax}$ is defined by the down set of arcs that do not cross the vertical axis, \ie~$\c{A}_\textrm{Bax} = \set{(i, j, A, B)}{A = \varnothing \text{ or } B = \varnothing}$.
Its congruence classes correspond to diagonal rectangulations~\cite{LawReading} or equivalently pairs of twin binary trees~\cite{Giraudo}, which are counted by the Baxter numbers.
It can also be seen as the transitive closure of the rewriting rule~$U v V u x W w X \equiv_{\textrm{Bax}} U v V x u W w X$ for~$u < v, w < x$.
See also~\cite{CardinalPilaud}.
\label{item:BaxterCongruence}

\item the \defn{generic rectangulation congruence}~$\equiv_\textrm{Rec}$ is defined by the down set of arcs that do not cross twice the vertical axis, \ie~$\c{A}_\textrm{Rec} \! = \! \set{(i, j, A, B)}{A \text{ and } B \text{ are empty or intervals}}$
Its congruence classes correspond to generic rectangulations~\cite{Reading-rectangulations} up to wall slides.
See also~\cite{AsinowskiCardinalFelsnerFusy,CardinalPilaud}.
\label{item:rectangulationCongruence}

\item for~$p \ge 1$, the \defn{$p$-twist congruence}~$\equiv_{p\textrm{-twist}}$ is defined by the down set of arcs passing on the left of at most~$p$ points, \ie~$\c{A}_{p\textrm{-twist}} = \set{(i, j, A, B)}{|B| \le p}$.
Its congruence classes correspond to certain acyclic pipe dreams~\cite{Pilaud-brickAlgebra}.
It can also be seen as the transitive closure of the rewriting rule~$U u w V_1 v_1 \dots V_p v_p W \equiv_{p\textrm{-twist}} U w u V_1 v_1 \dots V_p v_p W$ for~${u \! < \! v_1, \dots, v_p \! < \! w}$.
\label{item:pTwistCongruence}

\item a congruence~$\equiv$ of the weak order~$W_n$ is \defn{regular} if the cover graph of the quotient~$W_n / {\equiv}$ is regular (or equivalently if the quotient fan~$\equiv_\equiv$ is simplicial and the quotientope~$\quotientope$ is simple, see \cref{subsec:recollectionsGeometry}).
It was proved in~\cite{HoangMutze, DemonetIyamaReadingReitenThomas, BarnardNovelliPilaud} that $\equiv$ is regular if and only if any minimal (in subarc order) arc in the complement of the down set~$\c{A}_\equiv$ is either a left arc~$(i, j, \varnothing, {]i,j[})$ or a right arc~$(i, j, {]i,j[}, \varnothing)$.
For instance, the sylvester, Cambrian, and permutree congruences are regular, while the Baxter, generic rectangulation, and $p$-twist (for~$p > 1$) congruences are not regular.
\label{item:simpleCongruences}

\end{enumerate}


\subsection{Forcing order and subarc order}
\label{subsec:forcingOrder}

The goal of this section is to describe the forcing order on the join irreducible elements of the $\s$-weak order.
Our approach is similar to the one developed for the weak order in~\cite[Thm.~4.4]{Reading-arcDiagrams} and relies on the property that the $\s$-weak order is a {polygonal} lattice with well-understood polygons, described in~\cite[Lems.~3.16~\&~3.18 and Thm.~3.19]{Lacina} and~\cite[Prop.~1.35]{CeballosPons-sWeakOrderI}.

\begin{proposition}[{\cite[Prop.~1.35]{CeballosPons-sWeakOrderI}, from~\cite[Lems.~3.16~\&~3.18 and Thm.~3.19]{Lacina}}]
\label{prop:polygons}
The $\s$-weak order is a polygonal lattice. 
More precisely, let $\tree[T]$ be an $\s$-tree such that $\tree[T]$ is covered by~$\tree[R]$ and $\tree[S]$ by left rotations of the ascents~$(a,b)$ and~$(c,d)$ respectively, with~$a < c$.
Then the interval $[\tree[T], \tree[R] \join \tree[S]]$ is either a square, a pentagon, or an hexagon depending on the following cases:
\begin{enumerate}
\item If $(a,b)$ is an ascent of~$\tree[S]$ and~$(c,d)$ is an ascent of~$\tree[R]$, then $[\tree[T], \tree[R] \join \tree[S]]$ is a square depicted in the first case of Figure~\ref{fig:polygons}.
\label{cond:polygons-square}
\item If $(a,b)$ is an ascent of $\tree[S]$ but $(c,d)$ is not an ascent of $\tree[R]$, then $b = c$ and $[\tree[T], \tree[R] \join \tree[S]]$ is a pentagon depicted in the second case of Figure~\ref{fig:polygons}.
\label{cond:polygons-pentagon-left}
\item If $(a,b)$ is not an ascent of $\tree[S]$ and $(c,d)$ is an ascent of $\tree[R]$, then $b = c$ and $[\tree[T], \tree[R] \join \tree[S]]$ is a pentagon depicted in the third case of Figure~\ref{fig:polygons}.
\label{cond:polygons-pentagon-right}
\item If $(a,b)$ is not an ascent of $\tree[S]$ and $(c,d)$ is not an ascent of $\tree[R]$, then $b = c$, $s_b = 1$, and $[\tree[T], \tree[R] \join \tree[S]]$ is a hexagon depicted in the fourth case of Figure~\ref{fig:polygons}.
\label{cond:polygons-hexagon}
\end{enumerate}
\end{proposition}

\begin{figure}[ht]
	\capstart
	\centerline{\input{figures/polygons}}
	\caption{The four possible intervals between $\tree[T]$ and $\tree[R] \join \tree[S]$ where $\tree[T] \iscovered \tree[R]$ and $\tree[T] \iscovered \tree[S]$ in the $\s$-weak order. Figure adapted from~\cite{CeballosPons-sWeakOrderI}.}
	\label{fig:polygons}
\end{figure}
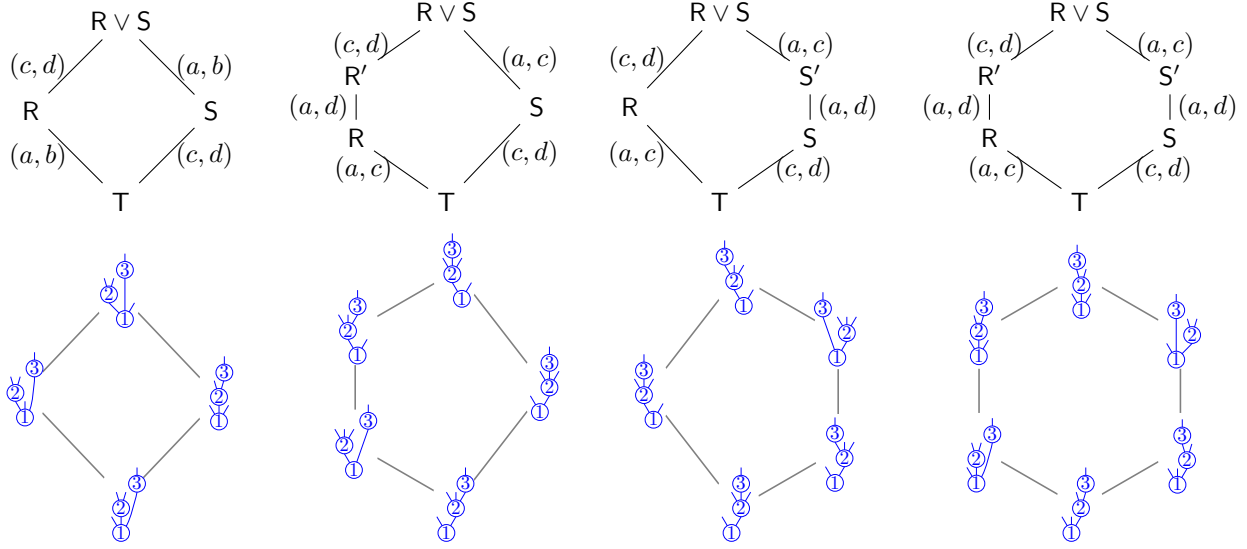

\begin{lemma}
\label{lem:forcing_join-irred_NCAD}
Let $\alpha \eqdef (i, j, A, B, r)$ be an $\s$-arc and $\delta$ be a non-crossing $\s$-arc diagram such that~${\alpha\in \delta}$. 
Denote by $\tree_{\!\join}(\alpha)_\star$ the only $\s$-tree covered by $\tree_{\!\join}(\alpha)$, and by~$\tree_{\!\join}(\delta)_\star$ the $\s$-tree obtained by rotating the descent $(i,j)$ of $\tree_{\!\join}(\delta)$. 
Then, for any congruence $\equiv$ of the $\s$-weak order, $\tree_{\!\join}(\alpha)_\star \equiv \tree_{\!\join}(\alpha)$ if and only if ${\tree_{\!\join}(\delta)_\star} \equiv {\tree_{\!\join}(\delta)} $. 
\end{lemma}

\begin{proof}
Assume first that $\tree_{\!\join}(\alpha)_\star \equiv \tree_{\!\join}(\alpha)$. 
Then ${\tree_{\!\join}(\delta)=\bigJoin_{\beta \in \delta} \tree_{\!\join}(\beta) \equiv \bigJoin_{\beta \in \delta\ssm\{\alpha\}} \tree_{\!\join}(\beta) \join \tree_{\!\join}(\alpha)_\star}$. 
Let us denote $\tree' \eqdef \bigJoin_{\beta \in \delta\ssm\{\alpha\}} \tree_{\!\join}(\beta) \join \tree_{\!\join}(\alpha)_\star$. 
It is clear that $\tree_{\!\join}(\alpha) \nleq \tree'$ in the $\s$-weak order. 
Thus it follows from \cref{lem:prepCanonicalJoinRepresentation} (where we take $\tree \eqdef \tree_{\!\join}(\delta)$) that $\pos(\tree',i,j)<\pos(\tree_{\!\join}(\alpha),i,j)$, and \cref{prop:positionsRotation} implies that $\tree'\leq \tree_{\!\join}(\delta)_\star$.
Hence $\tree_{\!\join}(\delta)_\star$ is contained in the interval $[\tree', \tree_{\!\join}(\delta)]$, and $\tree_{\!\join}(\delta)_\star \equiv \tree_{\!\join}(\delta)$ since congruence classes are intervals.

Assume now that $\tree_{\!\join}(\delta)_\star\equiv \tree_{\!\join}(\delta)$. 
Then we have $\tree_{\!\join}(\alpha)=\tree_{\!\join}(\alpha)\meet \tree_{\!\join}(\delta) \equiv \tree_{\!\join}(\alpha)\meet \tree_{\!\join}(\delta)_\star$.
But since $\tree_{\!\join}(\alpha) \nleq \tree_{\!\join}(\delta)_\star$ in the $\s$-weak order, we have that $\tree_{\!\join}(\alpha)\meet \tree_{\!\join}(\delta)_\star \leq \tree_{\!\join}(\alpha)_\star$. 
We conclude that $\tree_{\!\join}(\alpha) \equiv \tree_{\!\join}(\alpha)_\star$.
\end{proof}

\begin{definition}
\label{def:subarc}
Consider two $\s$-arcs~$\alpha \eqdef (i, j, A, B, r)$ and~$\alpha' \eqdef (i', j', A', B', r')$. 
We say that $\alpha$ is a \defn{subarc} of $\alpha'$ if all the following conditions hold:
\begin{itemize}
\item $i'\leq i < j \leq j'$,
\item $A\subseteq A'$ and $B\subseteq B'$,
\item if $s_j=0$ then $j=j'$,
\item if $i'=i$ then $r=r'$,
\item if $i'<i$ then either $i\in A'$ and $r=1$, or $i\in B'$ and $r=s_i$.
\end{itemize}
\end{definition}

\begin{definition}
\label{def:subarc_cover}
We say that an $\s$-arc $\alpha \eqdef (i, j, A, B, r)$ \defn{is extended} by an $\s$-arc $\alpha' \eqdef (i', j', A', B', r')$ if one of the following conditions holds:
\begin{enumerate}
\item $s_j\neq 0$ and $\alpha'=(i,j+1,A\cup\{j\},B,r)$,
\item $s_j\neq 0$ and $\alpha'=(i,j+1,A,B\cup\{j\},r)$,
\item $r=s_i$ and $\alpha'=(i-1,j,A,B\cup\{i\},r')$,
\item $r=1$ and $\alpha'=(i-1,j,A\cup\{i\},B,r')$.
\end{enumerate}
\end{definition}

It is clear that the subarc relation is the transitive closure of the extension relation.
The subarc order is illustrated in \cref{fig:sSubarcOrder} (note that when~$n = 3$, the subarc and the extension orders coincide).

\begin{figure}[t]
	\capstart
	\centerline{\includegraphics[scale=.8]{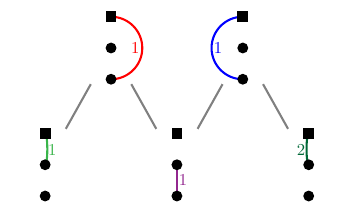} \hspace{1.5cm} \includegraphics[scale=.8]{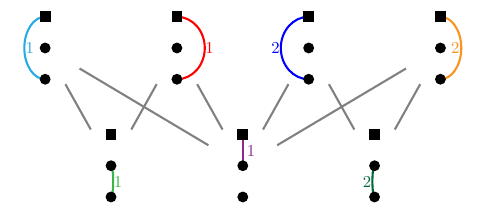}}
	\caption{The subarc order on $\s$-arcs for $\s = (1,2,0)$ (left) and~$\s = (2,1,0)$ (right).}
	\label{fig:sSubarcOrder}
\end{figure}

\begin{theorem}
\label{thm:forcingOrder}
The forcing order on join irreducible $\s$-trees coincide with the subarc order on $\s$-arcs:
if~$\alpha$ and~$\alpha'$ are two~$\s$-arcs, then $\tree_{\!\join}(\alpha)$ forces $\tree_{\!\join}(\alpha')$ if and only if $\alpha$ is a subarc of $\alpha'$.
\end{theorem}

\begin{definition}
\label{def:arrow_join-irred}
Consider two $\s$-arcs~$\alpha \eqdef (i, j, A, B, r)$ and~$\alpha' \eqdef (i', j', A', B', r')$. 
We write \defn{$\alpha \rightarrow \alpha'$} if there are non-crossing $\s$-arc diagrams $\delta$ and $\delta'$ such that $\alpha\in \delta$, $\alpha'\in \delta'$ and the edge $\tree_{\!\join}(\delta)_\star \iscovered \tree_{\!\join}(\delta)$ forces the edge  $\tree_{\!\join}(\delta')_\star \iscovered \tree_{\!\join}(\delta')$ in a polygon of the $\s$-weak order, where we denote by $\tree_{\!\join}(\delta)_\star$ the \mbox{$\s$-tree} obtained by rotating the descent $(i,j)$ of $\tree_{\!\join}(\delta)$ and by $\tree_{\!\join}(\delta')_\star$ the $\s$-tree obtained by rotating the descent $(i',j')$ of $\tree_{\!\join}(\delta')$.
\end{definition}

\begin{proof}[Proof of \cref{thm:forcingOrder}]
It follows from \cref{thm:edge_forcing_polygon}, \cref{lem:forcing_join-irred_NCAD} and \cref{def:arrow_join-irred} that the forcing order on the $\s$-arcs is the transitive closure of the relation $\rightarrow$. 

We can assume that $\alpha$ and $\alpha'$ are distinct.

First, we suppose that $\alpha \rightarrow \alpha'$ with $\delta, \delta', \tree_{\!\join}(\delta)_\star, \tree_{\!\join}(\delta')_\join$ as in \cref{def:arrow_join-irred} and we show that $\alpha$ is a subarc of $\alpha'$.
By considering the polygon in which the edge $\tree_{\!\join}(\delta)_\star \iscovered \tree_{\!\join}(\delta)$ forces the edge  $\tree_{\!\join}(\delta')_\star \iscovered \tree_{\!\join}(\delta')$, \cref{prop:polygons} shows that we are in one of the following cases:
\begin{enumerate}
\item $\tree_{\!\join}(\delta) = \tree[R]$ and $\tree_{\!\join}(\delta') = \tree[R]'$ (cases (\ref{cond:polygons-pentagon-left}) or (\ref{cond:polygons-hexagon}) of \cref{prop:polygons}), then $i=i'$, $s_j\neq 0$.
The $\s$-tree $\tree[R] = \tree_{\!\join}(\delta)=\tree_{\!\join}(\delta')_\star$ has a descent $(i,j)$ and an ascent $(i,j')$ around a same gap of the node $i$ (so that $(j,j')$ is an ascent of $\tree[R]'=\tree_{\!\join}(\delta')$). This implies that $r=r'$. Moreover, all the nodes $i<k<j$ that are weakly on the left, resp.\ strictly on the right, of the path from $i$ to $j$ in $\tree[R]$ are weakly on the left, resp.\ strictly on the right, of the path from $i$ to $j'$ in $\tree[R]'$, thus $A\subseteq A'$ and $B\subseteq B'$. 
(In this case we moreover have $j\in A'$).
\item $\tree_{\!\join}(\delta)=\tree[R]$ and $\tree_{\!\join}(\delta')=\tree[S]'$ (cases (\ref{cond:polygons-pentagon-right}) or (\ref{cond:polygons-hexagon}) of \cref{prop:polygons}), then $i=i'$ and $s_j\neq 0$.
The $\s$-tree $\tree$ has an ascent $(i,j)$ and an ascent $(j,j')$ around the leftmost gap of node $j$, so that $(i,j')$ is an ascent of the $\s$-tree~$\tree[S]$. 
 This implies that $r=r'$. 
 Moreover, all the nodes $i<k<j$ that are weakly on the left, resp.\ strictly on the right, of the path from $i$ to $j$ in $\tree[R]$ are weakly on the left, resp.\ strictly on the right, of the path from $i$ to $j'$ in $\tree[S]'$, thus $A\subseteq A'$ and $B\subseteq B'$. 
(In this case we moreover have $j\in B'$).
\item $\tree_{\!\join}(\delta) = \tree[S]$ and $\tree_{\!\join}(\delta')=\tree[S]'$ (cases (\ref{cond:polygons-pentagon-right}) or (\ref{cond:polygons-hexagon}) of \cref{prop:polygons}), then $j=j'$.
The tree $\tree[S]$ has an ascent $(i',j)$ and a descent $(i,j)$. This implies that $r=s_i$. 
Moreover, it follows from the fact that $(i,j)$ is an ascent of $\tree[S]'$ that there are no nodes $i<k<j$ on the path from $i$ to $j$ in $\tree[S]$.
Hence, all the nodes $i<k<j$ that are weakly on the left, resp.\ strictly on the right, of the path from $i$ to $j$ in $\tree[S]$ are weakly on the left, resp.\ strictly on the right, of the path from $i'$ to $j$ in $\tree[S]'$, thus $A\subseteq A'$ and $B\subseteq B'$. 
The node $i$ is on the right of the path from $i'$ to $j$ in $\tree[S]'$, thus $i\in B'$.
\item $\tree_{\!\join}(\delta)=\tree[S]$ and $\tree_{\!\join}(\delta')=\tree[R]'$ (cases (\ref{cond:polygons-pentagon-left}) or (\ref{cond:polygons-hexagon}) of \cref{prop:polygons}), then $j=j'$. 
The $\s$-tree $\tree$ has an ascent $(i',i)$ and an ascent $(i,j)$. The fact that $(i',j)$ is an ascent of $\tree[R]$ implies that $r=1$ (so that $j$ is not moved during the rotation of $(i',i)$ from $\tree$ to $\tree[R]$).
The nodes $i<k<j$ that are weakly on the left, resp.\ strictly on the right, of the path from $i$ to $j$ in $\tree[S]$ are weakly on the left, resp.\ strictly on the right, of the path from $i'$ to $j$ in $\tree[R]'$, thus $A\subseteq A'$ and $B\subseteq B'$. 
The node $i$ is on the path from $i'$ to $j$ on $\tree[R]'$, thus $i\in A'$.
\end{enumerate}

Then, we suppose that $\alpha$ is extended by $\alpha'$ and we show that $\alpha \rightarrow \alpha'$. 
For each of the four cases of \cref{def:subarc_cover} we indicate which $\delta$ and $\delta'$ we can use to recover the four previously studied cases.

\begin{enumerate}
\item $s_j\neq 0$ and $\alpha'=(i,j+1,A\cup\{j\},B,r)$. We take $\delta=\{\alpha\}$ if $r=1$ or ${\delta=\{\alpha, (i,j+1,A\cup\{j\}},$ $B,r-1)\}$ if $r>1$, and $\delta'=\{\alpha'\}$. Then we have $\tree_{\!\join}(\delta)=\tree[R]$ and $\tree_{\!\join}(\delta')=\tree[R]'$ in case (\ref{cond:polygons-pentagon-left}) of \cref{prop:polygons} if $s_j>1$ or case (\ref{cond:polygons-hexagon}) if $s_j=1$.
\item $s_j\neq 0$ and $\alpha'=(i,j+1,A,B\cup\{j\},r)$. 
If $s_j>1$, we take $\delta=\{\alpha, (j,j+1,\emptyset,\emptyset,s_j)\}$, and~$\delta'=\{\alpha'\}$ if $r=1$ or $\delta'=\{\alpha', (i,j,A,B,r-1) \}$ if $r>1$. Then we have $\tree_{\!\join}(\delta)=\tree[R]$ and $\tree_{\!\join}(\delta')=\tree[S]'$ in case (\ref{cond:polygons-pentagon-right}) of \cref{prop:polygons}.
If $s_j=1$, we take $\delta=\{\alpha\}$ and $\delta'=\{\alpha'\}$ if~$r=1$ or $\delta=\{\alpha, (i,j+1,A\cup\{j\},B,r-1)\}$ and  $\delta'=\{\alpha', (i,j,A,B,r-1) \}$ if $r>1$. Then we have $\tree_{\!\join}(\delta)=\tree[R]$ and $\tree_{\!\join}(\delta')=\tree[S]'$ in case (\ref{cond:polygons-hexagon}) of \cref{prop:polygons}.
\item $r=s_i$ and $\alpha'=(i-1,j,A,B\cup\{i\},r')$. 
If $A=\emptyset$, we take $\delta=\{\alpha\}$ and $\delta'=\{\alpha'\}$ if~$r'=1$, or $\delta=\{\alpha, (i-1, i, \emptyset, \emptyset,r'-1)\}$ and $\delta'=\{\alpha', (i-1,i,\emptyset,\emptyset,r'-1)\}$ if $r'>1$. 
If $A\neq \emptyset$, we denote by $k\in {]i,j[}$ the maximum element of $A$ 
and we add the $\s$-arc $(i-1, k, A\ssm\{k\}, (B\cap {]i,k[}) \cup \{i\}, r')$ to the diagrams $\delta$ previously defined.
Then we have $\tree_{\!\join}(\delta)=\tree[S]$ and $\tree_{\!\join}(\delta')=\tree[S]'$ in case (\ref{cond:polygons-pentagon-right}) of \cref{prop:polygons} if $s_i>1$ or case (\ref{cond:polygons-hexagon}) if $s_i=1$.
\item $r=1$ and $\alpha'=(i-1,j,A\cup\{i\},B,r')$. 
If $A=\emptyset$, we take $\delta=\{\alpha\}$ if $r'=1$, or $\delta=\{\alpha, (i-1, i, \emptyset, \emptyset,r'-1)\}$ if $r'>1$,  and $\delta'=\{\alpha'\}$. 
If $A\neq \emptyset$, we denote by $k\in {]i,j[}$ the maximum element of $A$ 
and we add the $\s$-arc $(i-1, k, A\ssm\{k\}, (B\cap {]i,k[}) \cup \{i\}, r')$ to the diagram $\delta$ and we add $(i,k,A\ssm\{k\},B\cap {]i,k[},s_i)$ to $\delta'$.
Then we have $\tree_{\!\join}(\delta)=\tree[S]$ and~$\tree_{\!\join}(\delta')=\tree[R]'$ in case (\ref{cond:polygons-pentagon-left}) of \cref{prop:polygons} if $s_i>1$ or case (\ref{cond:polygons-hexagon}) if $s_i=1$.
\qedhere
\end{enumerate}
\end{proof}

From~\cref{thm:congruenceLattice,thm:forcingOrder}, we immediately deduce the following statement, illustrated in \cref{fig:sWeakOrderCongruenceLattice}.

\begin{figure}[t]
	\capstart
	\centerline{\raisebox{1.2cm}{\includegraphics[scale=.8]{120-congruenceLattice}} \hspace{.5cm} \includegraphics[scale=.8]{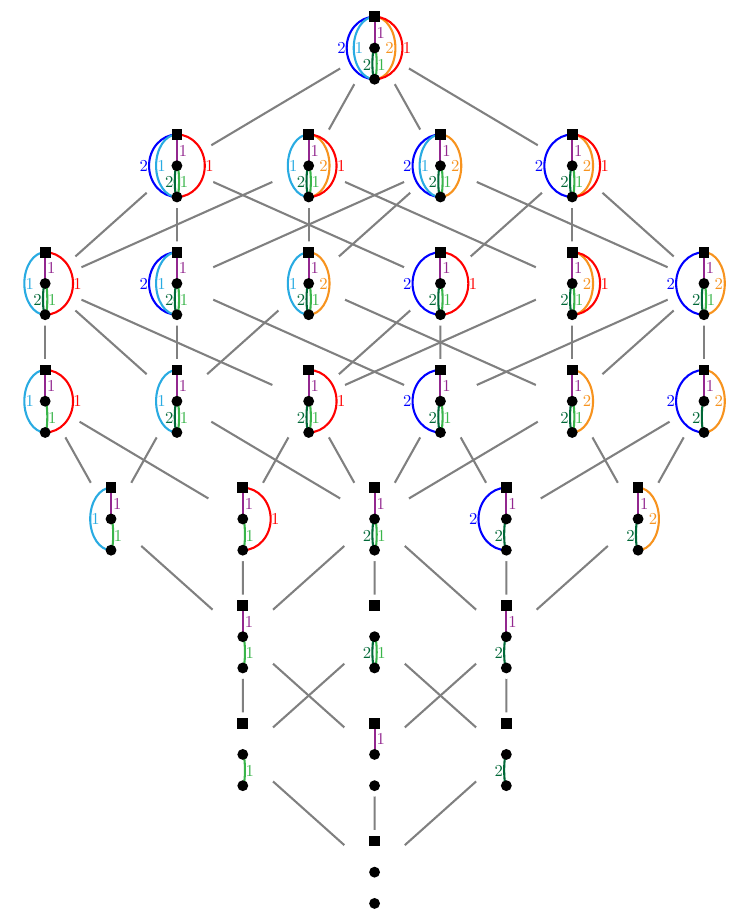}}
	\caption{The congruence lattice of the $\s$-weak order~$W_\s$ for $\s = (1,2,0)$ (left) and~$\s = (2,1,0)$ (right), where each congruence~$\equiv$ is replaced by its $\s$-arc down set~$\c{A}_\equiv$. See also \cref{fig:sQuotientFoamLattice,fig:sQuotientopeLattice}.}
	\label{fig:sWeakOrderCongruenceLattice}
\end{figure}

\begin{corollary}
\label{coro:sWeakOrderCongruenceLattice}
The congruence lattice of the $\s$-weak order~$W_\s$ is isomorphic to the lattice of down sets of the subarc order on $\s$-arcs.

For a congruence~$\equiv$ of~$W_\s$, we denote by~$\c{A}_\equiv$ the corresponding down set of the subarc order on $\s$-arcs.
\end{corollary}


\subsection{Some relevant congruences of the $\s$-weak order}
\label{subsec:relevantQuotients}

In this section, we exploit \cref{coro:sWeakOrderCongruenceLattice} to define particularly relevant congruences, mimicking some congruences of \cref{subsubsec:specialCongruences}.
We present a few conjectures that echo the situation of the similar congruences in the classical weak order.

\subsubsection{The $\s$-sylvester congruence and the $\s$-Tamari lattice}
\label{subsubsec:sylvester}

A \defn{right $\s$-arc} is an $\s$-arc of the form~$(i, j, \set{k \in {]i,j[}}{s_k \ne 0}, \varnothing, r)$ for some~$1 \le i < j \le n$ and~$r\in[s_i]$.
The set of right $\s$-arcs clearly forms a down set~$\c{A}_\mathrm{sylv}$ of the subarc order.
The corresponding congruence of the $\s$-weak order is the \defn{$\s$-sylvester congruence}~$\equiv_\mathrm{sylv}$, and the quotient~$W_\s / {\equiv_\mathrm{sylv}}$  of the $\s$-weak order by the $\s$-sylvester congruence is the \defn{$\s$-Tamari lattice}.
An $\s$-tree~$\tree$ is minimal in its $\s$-sylvester congruence class if and only if~$\pos(\tree, a, b) \le \pos(\tree, a, c)$ for all~$1 \le a < b < c \le n$ with~$s_b \ne 0$.

When~$\s$ contains no~$0$, and up to our convention modifications (see \cref{rem:sTreeVSsDecreasingTrees}), these $\s$-trees are called \defn{$\s$-Tamari trees} in~\cite{CeballosPons-sWeakOrderI,CeballosPons-sWeakOrderII}.
We thus directly obtain that the subposet of the \mbox{$\s$-weak} order induced by the $\s$-Tamari trees is (isomorphic to) the $\s$-Tamari lattice~$W_\s / {\equiv_\mathrm{sylv}}$.
This immediately recovers \cite[Thm.~2.20]{CeballosPons-sWeakOrderI}.
We refer to~\cite[Sect.~2.4]{CeballosPons-sWeakOrderI} for the connection to the $\nu$-Tamari lattice of~\cite{PrevilleRatelleViennot, CeballosPadrolSarmiento-geometryNuTamari}.

When~$\s$ contains some~$0$, the $\s$-trees that are minimal in their $\s$-sylvester congruence classes actually differ from the $\s$-Tamari trees of~\cite{CeballosPons-sWeakOrderI}, because they drop the condition~$s_b \ne 0$.
It was proved in~\cite[Thm.~2.2]{CeballosPons-sWeakOrderI} that the subposet of the $\s$-weak order induced by the $\s$-Tamari trees is still a sublattice of the $\s$-weak order, but is not anymore a lattice quotient of the $\s$-weak order.
(We note that even when~$\s$ contains no~$0$, it follows from the theory of lattice congruences that this subposet is a join sublattice, but it could a priori fail to be a meet sublattice.)
They still call this sublattice the $\s$-Tamari lattice.
We insist that, in the situation that~$\s$ contains some~$0$, our $\s$-Tamari lattice differs from that of~\cite[Sect.~2]{CeballosPons-sWeakOrderI}.

\subsubsection{The $\s$-Cambrian congruences}

For any $\s$-arc~$\alpha$, the \defn{$\alpha$-Cambrian congruence}~$\equiv_\alpha$ is defined by the down set of subarcs of~$\alpha$, and the quotient~$W_\s / {\equiv_\alpha}$ is the \defn{$\alpha$-Cambrian lattice}.
Note that the $\s$-sylvester congruence is not anymore an $\s$-Cambrian congruence (as some right arcs are not subarcs of the arc~$(1, n, \set{k \in {]1,n[}}{s_k \ne 0}, \varnothing, 1)$).
Our main conjecture on $\s$-Cambrian congruences is the following.

\begin{conjecture}
\label{conj:Cambrian}
For any fixed~$\s$, the following only depend on the endpoints of~$\alpha$:
\begin{itemize}
\item the cardinality of the $\alpha$-Cambrian lattice,
\item the $f$-vector of the canonical join complex of the $\alpha$-Cambrian lattice,
\item the (isomorphism class of the) undirected cover graph of the $\alpha$-Cambrian lattice,
\item the (isomorphism class of the) face lattice of the $\alpha$-Cambrian foam~$\quotientFoam[\equiv_\alpha]$ or equivalently of the $\alpha$-Cambrian quotientoplex~$\quotientoplex[\equiv_\alpha]$ (see \cref{sec:sQuotientopes} for the definitions of~$\quotientFoam$ and~$\quotientoplex$).
\end{itemize}
\end{conjecture}

\subsubsection{The $\s$-permutree congruences}

For a map~$\decoration : \set{i \in [n]}{s_i \ne 0} \to \Decorations$, the \defn{$\decoration$-per\-mutree congruence}~$\equiv_\decoration$ is defined by the down set of arcs that do not pass on the right of a point~$j$ with~$\decoration(j) \in \{\upCirc, \upDownCirc\}$ nor on the left of the points~$j$ with~$\decoration(j) \in \{\downCirc, \upDownCirc\}$.
The \defn{$\decoration$-permutree lattice} is the quotient~$W_\s / {\equiv_\decoration}$.
For these congruences, we observed the following behavior.

\begin{conjecture}
\label{conj:permutrees}
For any fixed~$\s$, changing any~$\upCirc$ to~$\downCirc$ in~$\decoration$ does not affect 
\begin{itemize}
\item the cardinality of the $\decoration$-permutree lattice,
\item the $f$-vector of the canonical join complex of the $\decoration$-permutree lattice.
\end{itemize}
\end{conjecture}

Note that in contrast to \cref{conj:Cambrian}, the undirected cover graph of the $\decoration$-permutree lattice and the face lattice of the $\decoration$-permutree foam~$\quotientFoam[\equiv_\decoration]$ and of the $\decoration$-permutree quotientoplex~$\quotientoplex[\equiv_\decoration]$ all depend on the positions of~$\upCirc$ and~$\downCirc$.
For instance, for~$\s \eqdef (1, 1, 2, 2, 1, 1)$, and~$\decoration \eqdef (\upCirc, \upCirc, \upCirc, \upCirc, \upCirc, \upCirc)$ and~$\decoration' \eqdef (\upCirc, \upCirc, \upCirc, \downCirc, \upCirc, \upCirc)$, the $\decoration$- and $\decoration'$-permutree lattices both have cardinality~$331$, the \mbox{$f$-vector} of their canonical complex is~$(1, 20, 93, 139, 69, 9)$, but their undirected cover graphs are not isomorphic.

\subsubsection{Regular congruences}

An interval of a lattice~$L$ is \defn{cellular} if it is of the form~$[x, \bigJoin Y]$ for~$x \in L$ and a non-empty subset~$Y$ of elements of~$L$ covering~$x$, or dually~$[\bigMeet X, y]$ for any~$y \in L$ and a non-empty subset~$X$ of elements of~$L$ covered by~$y$.
We say that~$L$ is \defn{cellularly regular} if the cover graph of any cellular interval of~$L$ is regular.
Note that it obviously does not imply that the cover graph of the lattice~$L$ itself is regular.
The result of~\cite{HoangMutze, DemonetIyamaReadingReitenThomas, BarnardNovelliPilaud} on regular congruences of the weak order motivates the following conjecture.

\begin{conjecture}
\label{conj:simpleCongruences}
The following conditions are equivalent for a congruence~$\equiv$ of the $\s$-weak order~$W_\s$:
\begin{itemize}
\item the quotient~$W_\s / {\equiv}$ is cellularly regular,
\item any minimal (in subarc order) arc of the complement of the down set~$\c{A}_\equiv$ is either a left arc (\ie with~${A = \varnothing}$) or a right arc (\ie with~$B = \varnothing$).
\end{itemize}
\end{conjecture}


\section{Quotient foams and quotientoplexes}
\label{sec:sQuotientopes}

In this section, we construct polyhedral complexes realizing all lattice quotients of the $\s$-weak order~$W_\s$.
Most proofs of the results of this section actually require some basic tropical geometry, and are therefore delayed to \cref{sec:tropicalGeometry}.


\subsection{Recollections \ref{sec:sQuotientopes}: Quotient fans and quotientopes}
\label{subsec:recollectionsGeometry}

As usual now, we start with a recollection of the geometric realizations of the lattice quotients of the classical weak order~\cite{Reading-HopfAlgebras, PilaudSantos-quotientopes, PadrolPilaudRitter}.

\subsubsection{Quotient fan}

We start with polyhedral fan realizations.

\begin{definition}[{\cite[Thm.~1.1]{Reading-HopfAlgebras}}]
\label{def:quotientFan}
The \defn{quotient fan} of a congruence~$\equiv$ of the weak order is the polyhedral fan~$\quotientFan$~where
\begin{enumerate}[(i)]
\item the maximal cones are obtained by gluing together the chambers of the braid arrangement corresponding to permutations in the same congruence class of~$\equiv$,
\item the union of the codimension~$1$ cones is the union of the shards of the arcs of~$\c{A}_\equiv$.
\end{enumerate}
\end{definition}

By construction, the braid fan refines the quotient fan~$\quotientFan$, and the dual graph of the quotient fan~$\quotientFan$, oriented in the direction~$\b{\omega}$, is isomorphic to the Hasse graph of the quotient~$W_n/{\equiv}$.
For instance, \cref{fig:sylvesterFanAssociahedron}\,(left) represents the \defn{sylvester fan}, the quotient fan of the sylvester congruence of \cref{fig:sylvesterCongruence}\,(left), whose dual graph is the cover graph of the Tamari lattice of \cref{fig:sylvesterCongruence}\,(right).

\begin{figure}[b]
	\capstart
	\centerline{\includegraphics[scale=.6]{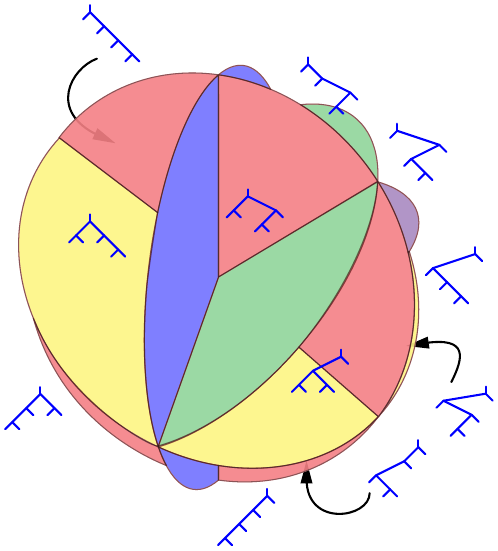} \qquad \raisebox{.1cm}{\includegraphics[scale=.6]{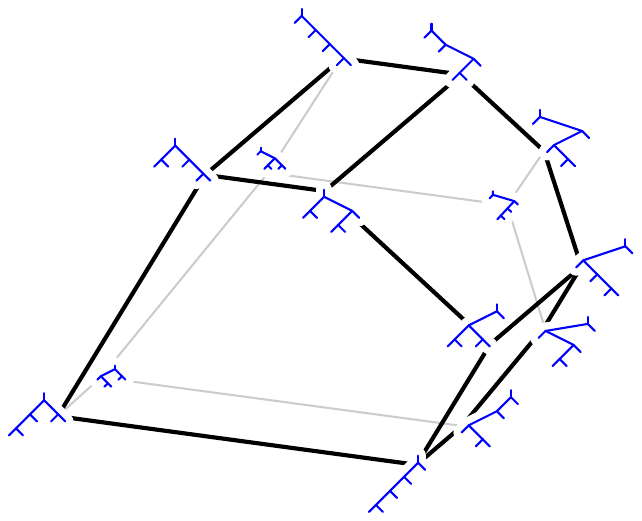}}}
	\caption{The sylvester fan (left) and the associahedron~$\Asso[4]$ (right). \cite[Fig.~5]{PadrolPilaudRitter}}
	\label{fig:sylvesterFanAssociahedron}
\end{figure}

\subsubsection{Shard polytopes and quotientopes}

We now recall the definition and main property of the shard polytopes of~\cite{PadrolPilaudRitter}, which are illustrated in \cref{fig:shardPolytiopes}.

\begin{definition}[{\cite[Defs.~39~\&~40]{PadrolPilaudRitter}}]
\label{def:shardPolytopes}
Fix a classical arc~$\alpha \eqdef (i, j, A, B)$.
An \defn{$\alpha$-alternating matching} $\altmatch$ is a sequence $i \le i_1 < j_1 < i_2 < j_2 < \dots < i_q < j_q \le j$ such that~$i_p \in \{i\} \cup A$ and~$j_p \in \{j\} \cup B$ for all~$p \in [q]$.
Its \defn{characteristic vector} is~$\charvect_{\altmatch} \eqdef \sum_{p \in [q]} \b{e}_{i_p} - \b{e}_{j_p}$. 
We denote by $\altmatchset$ the set of all $\alpha$-alternating matchings. 
The \defn{shard polytope} of~$\alpha$ is the convex hull~$\shardPolytope$ of the characteristic vectors of all $\alpha$-alternating matchings~$\altmatch$ in $\altmatchset$.
\begin{figure}
	\capstart
	\centerline{\includegraphics[scale=.9]{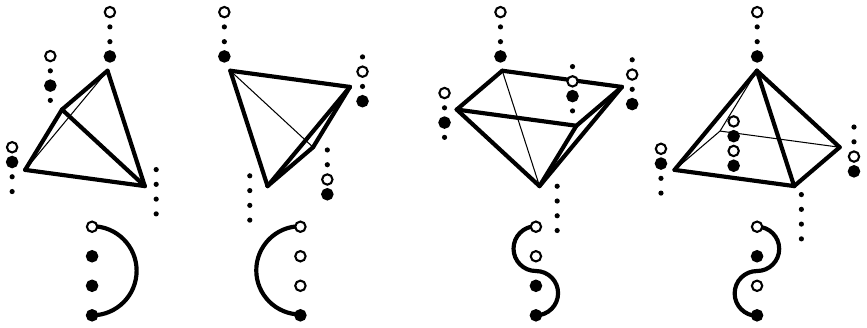}}
	\caption{The shard polytopes of the four arcs of the form~$(1, 4, A, B)$. The vertices of the shard polytope of~$\alpha \eqdef (i, j, A, B)$ are labeled by the corresponding $\alpha$-alternating matchings, where we use solid dots~$\bullet$ for elements in~$\{i\} \cup A$ and hollow dots~$\circ$ for elements in~$B \cup \{j\}$. The corresponding vertex coordinates are directly read replacing~$\bullet$ by~$1$ and~$\circ$ by~$-1$. For instance, the vertex labeled~\raisebox{-.1cm}{\rotatebox{90}{\scalebox{.6}{${\bullet \cdot \cdot \,\circ}$}}} has coordinates~$(1,0,0,-1)$. Adapted from~\cite[Fig.~10]{PadrolPilaudRitter}.}
	\label{fig:shardPolytiopes}
\end{figure}
\end{definition}

For instance, the shard polytope~$\shardPolytope$ corresponding to a right arc~$\alpha \eqdef (i, j, {]i,j[}, \varnothing)$ is just the simplex~$\triangle_{[i,j]} \eqdef \conv\set{\b{e}_k}{i \le k \le j}$ translated by the vector~$-\b{e}_j$.

\begin{proposition}[{\cite[Prop.~48]{PadrolPilaudRitter}}]
\label{prop:shardPolytopes}
For any arc~$\alpha$, the union of the walls of the normal fan of the shard polytope~$\shardPolytope$
\begin{itemize}
\item contains the shard~$\shard[\alpha]$ corresponding to~$\alpha$,
\item is contained in the union of the shards~$\shard[\beta]$ over all subarcs~$\beta$ of~$\alpha$.
\end{itemize}
\end{proposition}

The shard polytopes are the building blocks to construct polytopal realizations of lattice quotients of the weak order.
We note that alternative realizations were constructed by V.~Pilaud and F.~Santos in~\cite{PilaudSantos-quotientopes} using direct but slightly obscure right hand sides to define their inequalities.

\begin{theorem}[{\cite[Coro.~50]{PadrolPilaudRitter}}]
\label{thm:quotientopes}
For any congruence~$\equiv$ of the weak order, the quotient fan~$\quotientFan$ is the normal fan of the \defn{quotientope}~$\quotientope$, obtained as the Minkowski sum of (any positive scaling of) the shard polytopes~$\shardPolytope$, over all arcs~$\alpha$ in~$\c{A}_\equiv$.
\end{theorem}

By construction, the quotientope~$\quotientope$ is a deformed permutahedron (or generalized permutahedron~\cite{Postnikov, PostnikovReinerWilliams}, or polymatroid~\cite{Edmonds}), whose skeleton, oriented in the direction~$\b{\omega}$, is isomorphic to the Hasse diagram of the quotient~$W_n/{\equiv}$.
For instance, \cref{fig:sylvesterFanAssociahedron}\,(right) represents the associahedron~$\Asso[4]$, a quotientope for the sylvester congruence of \cref{fig:sylvesterCongruence}\,(left), whose normal fan is the sylvester fan of \cref{fig:sylvesterFanAssociahedron}\,(left), and whose skeleton is the cover graph of the Tamari lattice of \cref{fig:sylvesterCongruence}\,(right).
It is obtained as the Minkowski sum of the shard polytopes of all right arcs, that is, up to translation, of the faces~$\triangle_{[i,j]}$ of the standard simplex corresponding to all intervals~$1 \le i < j \le n$.

If we want to make explicit the scaling coefficients, we denote by~$\quotientope({\b{\lambda}})$ the quotientope obtained as the Minkowski sum~$\sum_{\alpha \in \c{A}_{\equiv}} \lambda_{\alpha} \, \shardPolytope$ for~${\b{\lambda}} \eqdef (\lambda_{\alpha})_{\alpha \in \c{A}_{\equiv}}$ with~$\lambda_{\alpha} > 0$ for all~$\alpha \in \c{A}_\equiv$.

\begin{figure}[p]
	\capstart
	\centerline{\includegraphics[scale=.6]{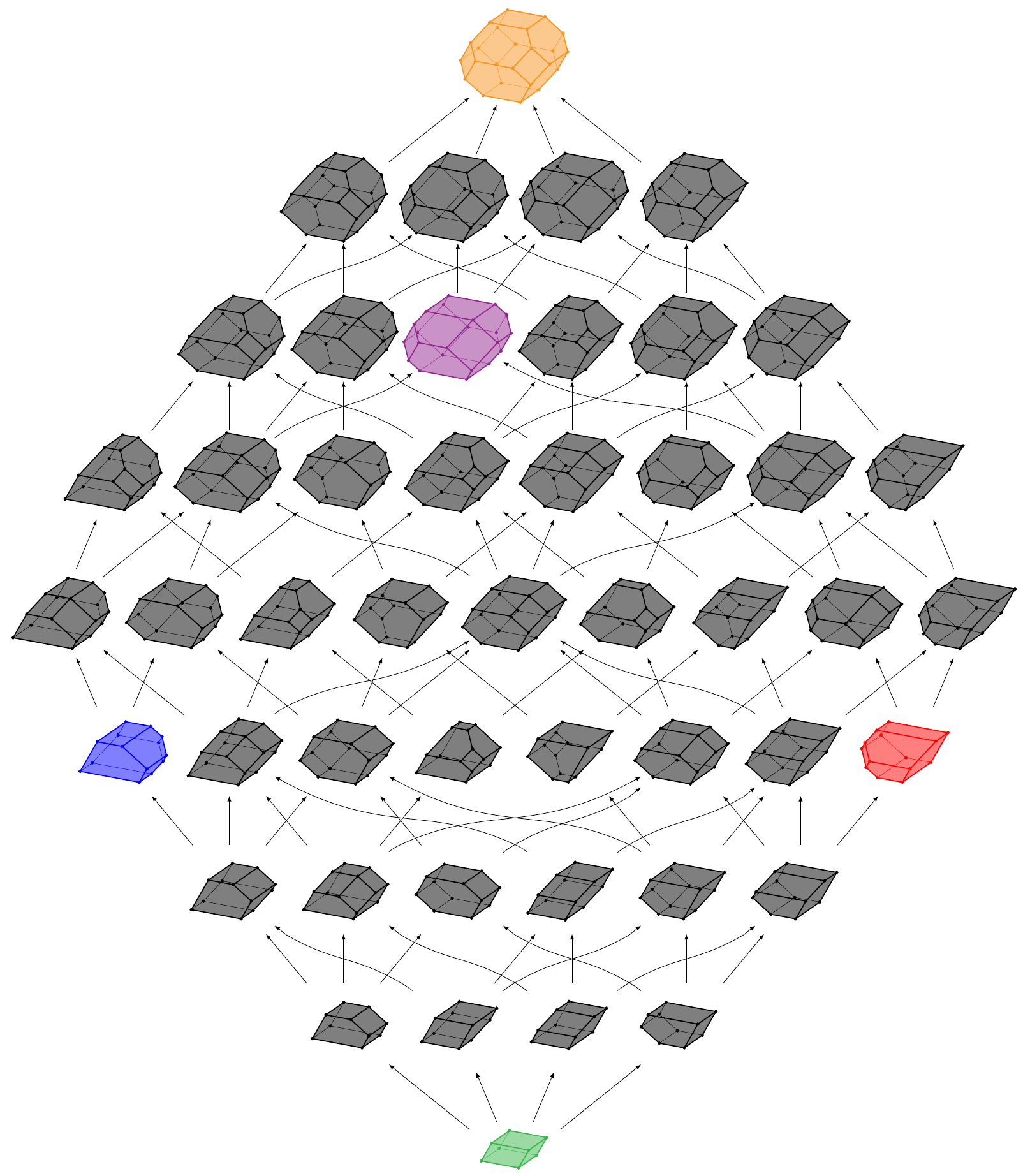}}
	\caption{The congruence lattice of the weak order~$W_4$, where each congruence~$\equiv$ is replaced by its quotientope~$\quotientope$. We have colored in green / blue / red / purple / orange the cube / associahedron / anti-associahedron / Baxter polytope / permutahedron. See also~\cref{fig:weakOrderCongruenceLattice}. Adapted from~\cite[Fig.~9]{PilaudSantos-quotientopes}}
	\label{fig:weakOrderQuotientopeLattice}
\end{figure}

\begin{remark}
In general, the quotient fan~$\quotientFan$ is not simplicial, and the quotientope~$\quotientope$ is not simple.
It is the case for regular congruences, see \cref{subsubsec:specialCongruences}\,\eqref{item:rectangulationCongruence} and~\cite{HoangMutze, DemonetIyamaReadingReitenThomas, BarnardNovelliPilaud}.
\end{remark}


\subsection{Quotient foams}
\label{subsec:sQuotientFoam}

We now generalize \cref{def:quotientFan} to the $\s$-weak order.
For a congruence of the $\s$-weak order~$W_\s$, we construct a polyhedral complex~$\quotientFoam$ coarsening the $\s$-foam~$\sFoam$, and whose oriented dual graph is isomorphic to the Hasse diagram of the quotient~$W_\s / {\equiv}$.
See \cref{fig:sQuotientFoamLattice,fig:sylvesterQuotientFoams} for some illustrations.
Recall from \cref{def:sShard} that we have associated to each $\s$-arc~$\alpha$ a codimension~$1$ polyhedral cone called the $\s$-shard~$\shard$ of~$\alpha$.

\begin{figure}[t]
	\capstart
	\centerline{\includegraphics[scale=.6]{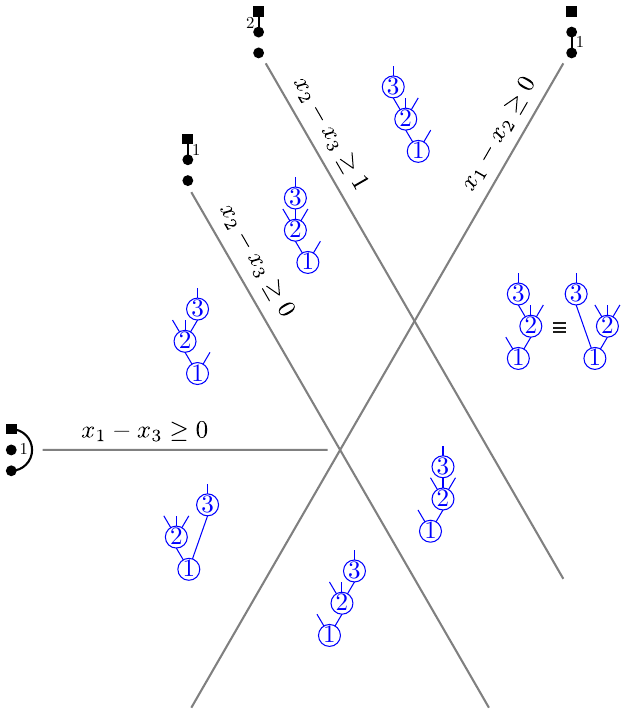} \hspace{.5cm} \includegraphics[scale=.6]{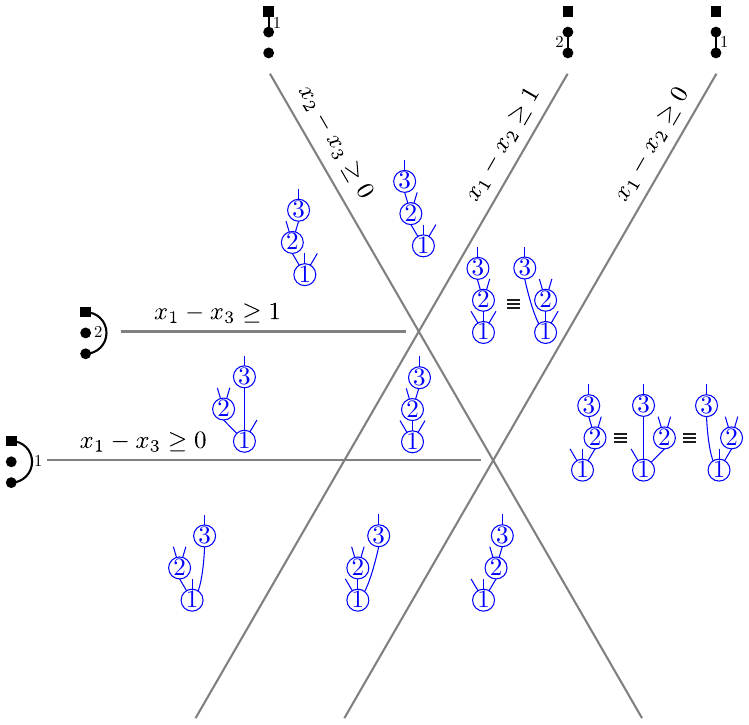}}
	\caption{The quotient foam~$\quotientFoam[\equiv_\mathrm{sylv}]$ for the $\s$-sylvester congruence of the $\s$-weak order of \cref{subsubsec:sylvester}, for $\s = (1,2,0)$ (left) and~$\s = (2,1,0)$ (right).}
	\label{fig:sylvesterQuotientFoams}
\end{figure}

\begin{figure}[t]
	\capstart
	\centerline{\raisebox{1.2cm}{\includegraphics[scale=.8]{120-quotientFoamLattice}} \hspace{.5cm} \includegraphics[scale=.8]{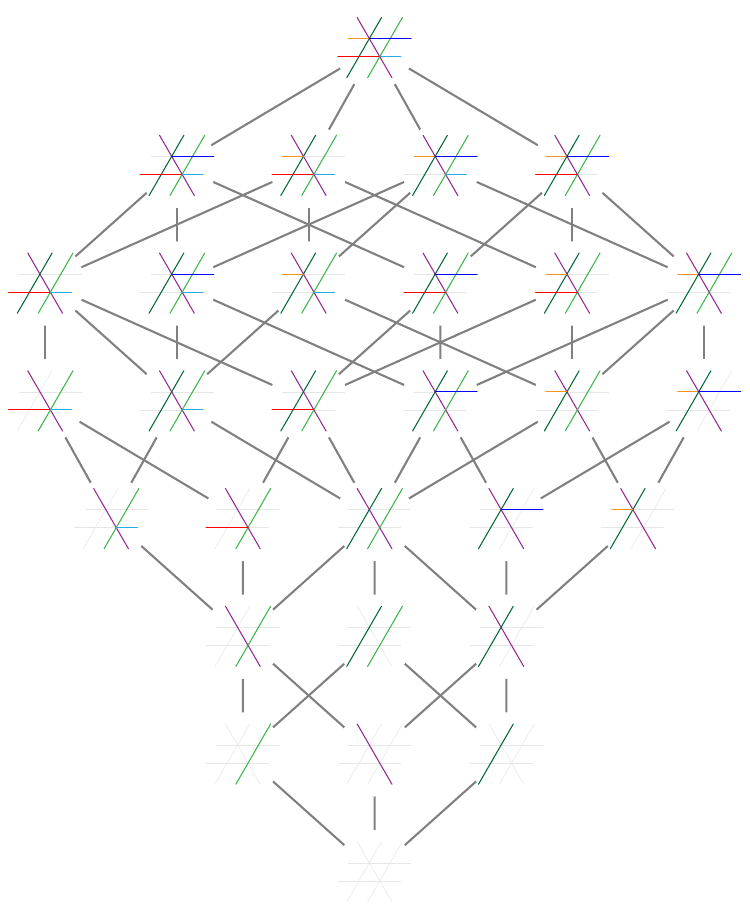}}
	\caption{The congruence lattice of the $\s$-weak order~$W_\s$ for $\s = (1,2,0)$ (left) and~$\s = (2,1,0)$ (right), where each congruence~$\equiv$ is replaced by its quotient foam~$\quotientFoam$. See also \cref{fig:sWeakOrderCongruenceLattice,fig:sQuotientopeLattice}.}
	\label{fig:sQuotientFoamLattice}
\end{figure}

\begin{definition}
\label{def:quotientFoam}
For any congruence~$\equiv$ of the $\s$-weak order~$W_\s$, the \defn{quotient foam}~$\quotientFoam$ is the complete polyhedral complex defined by the following equivalent descriptions:
\begin{enumerate}[(i)]
\item its maximal cells are obtained by gluing together the maximal cells of the $\s$-foam corresponding to $\s$-trees in the same congruence class of~$\equiv$,
\item the union of its codimension~$1$ cells is the union of the $\s$-shards~$\shard$ of the $\s$-arcs~$\alpha$~in~$\c{A}_\equiv$.
\end{enumerate}
\end{definition}

\cref{def:quotientFoam} requires some justifications, that we give in the following statements.

\begin{proposition}
\label{prop:quotientFoam1}
For any congruence~$\equiv$ of the $\s$-weak order~$W_\s$, the two descriptions of \cref{def:quotientFoam}~coincide.
\end{proposition}

\begin{proof}
Consider a cover relation~$\tree \iscovered \tree'$ in the $\s$-weak order, let~$(i,j)$ denote the ascent of~$\tree$ and descent of~$\tree'$ corresponding to this flip, and let~$\alpha \eqdef (i, j, A, B, r)$ denote the corresponding $\s$-arc in~$\delta_\join(\tree')$.
Consider now the $\s$-bush~$\bush$ obtained by stitching~$(i,j)$ in~$\tree$ (or equivalently in~$\tree'$), that is, such that~$\closedFiber{\tree} \cap \closedFiber{\tree'} = \fiber{\bush}$ by \cref{prop:stitchingIncision}.
Then~$\closedFiber{\bush}$ is contained in the shard~$\shard$ by \cref{lem:shard}.
As~$\tree \equiv \tree'$ if and only if~$\alpha \notin \c{A}_\equiv$, we thus obtain that $\tree \iscovered \tree'$ is contracted in the description of \cref{def:quotientFoam}\,(i) if and only if it is contracted in the description of \cref{def:quotientFoam}\,(ii). 
\end{proof}

\begin{proposition}
\label{prop:quotientFoam2}
For any congruence~$\equiv$ of the $\s$-weak order~$W_\s$, the quotient foam~$\quotientFoam$ of \cref{def:quotientFoam} is indeed a polyhedral complex.
\end{proposition}

\begin{proof}
The proof uses tropical geometry to show that the quotient foam is the polyhedral complex induced by a certain arrangement of tropical hypersurfaces, see \cref{thm:quotientFoamTropical}.
\end{proof}

\begin{proposition}
\label{prop:quotientFoam3}
The Hasse diagram of the quotient~$W_\s / {\equiv}$ is isomorphic to the dual graph of the quotient foam~$\quotientFoam$, oriented in the direction~$\b{\omega}$.
\end{proposition}

\begin{proof}
By \cref{def:quotientFoam}\,(i), the maximal cells of the quotient foam~$\quotientFoam$ are labeled by the equivalence classes of~$\equiv$.
Consider now two maximal cells~$\polytope{C}$ and~$\polytope{C}'$ corresponding to two equivalent classes~$\tree[C]$ and~$\tree[C]'$ of~$\equiv$.
Then the following are equivalent: 
\begin{itemize}
\item $\tree[C] \lessdot \tree[C]'$ in the quotient~$W_\s / {\equiv}$,
\item there exist two $\s$-trees~$\tree \in \tree[C]$ and~$\tree' \in \tree[C]'$  such that $\tree \lessdot \tree'$ in the $\s$-weak order, 
\item there exist two $\s$-trees~$\tree \in \tree[C]$ and~$\tree' \in \tree[C]'$ such that~$\closedFiber{\tree}$ and~$\closedFiber{\tree'}$ are adjacent, and $\omega$ points from~$\closedFiber{\tree}$ to~$\closedFiber{\tree'}$,
\item $\polytope{C}$ and~$\polytope{C}'$ are adjacent, and $\omega$ points from~$\polytope{C}$ to~$\polytope{C}'$.
\qedhere
\end{itemize}
\end{proof}


\subsection{Shardoplexes}
\label{subsec:shardPolytopalComplexes}

Generalizing \cref{def:shardPolytopes}, we now associate to each $\s$-arc~$\alpha$ a polytopal complex~$\shardoplex$, that we call the \defn{$\alpha$-shardoplex}.
To construct this polytopal complex, we associate a polytope (not necessarily full-dimensional) to each $\s$-arc and $\s$-trunk, in such a way that the polytopes corresponding to the same $\s$-arc but different \mbox{$\s$-trunks} glue nicely.
These polytopes are constructed as certain faces of shard polytopes of~\cite{PadrolPilaudRitter}, see \cref{def:shardPolytopes}.
Recall from \cref{rem:bush}\,\eqref{item:parametrization} our labeling of the $\s$-trunks by~$\sTrunks$, and from \cref{exm:grid} the geometric description of the corresponding minimal cells of the $\s$-foam~$\sFoam$.

\begin{definition}
\label{def:localShardPolytope}
Fix an $\s$-arc~$\alpha \eqdef (i, j, A, B, r)$, let~$\tilde\alpha \eqdef (i, j, A, B)$ denote the corresponding classical arc, and let~$\b{q} \in \sTrunks$.
The \defn{local shard polytope}~$\localShardPolytope$ is the face of the shard polytope~$\shardPolytope[\tilde\alpha]$ maximizing the scalar product with the vector
\(
\sum_{\ell \in {]i,j]}} \big( q_i - q_\ell + r - 1 + \sum_{k \in B \cap {]i,\ell[}} \max(0, s_k-1) \big) \b{e}_\ell.
\)
\end{definition}

\begin{remark}
\label{rem:abusing1}
In \cref{def:localShardPolytope}, note that $\tilde\alpha \eqdef (i, j, A, B)$ is strictly speaking not really an arc since~$A \sqcup B = \set{k \in {]i,j[}}{s_k \ne 0}$ may differ from~$]i,j[$ if~$\s$ contains some~$0$.
However, \cref{def:shardPolytopes} for the shard polytope~$\shardPolytope[\tilde\alpha]$ still holds, and was already considered in~\cite[Rem.~43]{PadrolPilaudRitter} as ``pseudoshard polytope''.
\end{remark}

\begin{proposition}
\label{prop:shardoplex}
For any $\s$-arc~$\alpha$, the collection of all local shard polytopes~$\localShardPolytope$ for~$\b{q} \in \sTrunks$, together with all their faces, form a polyhedral complex~$\shardoplex$ that we call the \defn{shardoplex} of~$\alpha$.
\end{proposition}

\begin{proof}
The union of the sets of faces of the local shard polytopes~$\localShardPolytope$ for~$\b{q} \in \sTrunks$ is a set of faces of the shard polytope~$\shardPolytope[\tilde\alpha]$ that contains~$\shardPolytope[\tilde\alpha]$ itself. 
Indeed, we can take for example $\b{q}\in\sTrunks$ such that $q_\ell = 1$ for all $\ell\in [i] \cup {]j,n]}$ and $q_\ell = r + \sum_{k \in B \cap {]i,\ell[}} \max(0, s_k-1)$ for all $\ell \in {]i,j]}$ to have~$\localShardPolytope=\shardPolytope[\tilde\alpha]$.
Hence the collection of local shard polytopes with their faces indeed form a polytopal complex.
\end{proof}

Generalizing \cref{prop:shardPolytopes}, the main feature of the $\alpha$-shardoplex~$\shardoplex$ is that the union of the walls of its dual polyhedral complex contains the $\s$-shard~$\shard$ and is contained in the union of the $\s$-shards~$\shard[\beta]$ over all subarcs~$\beta$ of the $\s$-arc~$\alpha$.
This property is properly stated in \cref{prop:shard_trophyp} in terms of tropical geometry.


\subsection{Quotientoplexes}
\label{subsec:sQuotientopes}

We now generalize \cref{thm:quotientopes} to the $\s$-weak order.
For a congruence~$\equiv$ of the $\s$-weak order, we define a polytopal complex~$\quotientoplex$ using Minkowski sums of the $\alpha$-shardoplexes of \cref{prop:shardoplex}, such that the oriented graph of~$\quotientoplex$ is isomorphic to the Hasse diagram of the quotient~$W_\s / {\equiv}$.
See \cref{fig:MinkowskiSums,fig:sQuotientopeLattice,fig:1211-quotientopes,fig:2101-quotientopes,fig:1432-quotientopes,fig:2321-quotientopes} for some illustrations.
We first need to explain what we mean by Minkowski sums here, even if it will be clear from \cref{sec:tropicalGeometry}.

\begin{figure}
	\capstart
	\centerline{\includegraphics[scale=1.5]{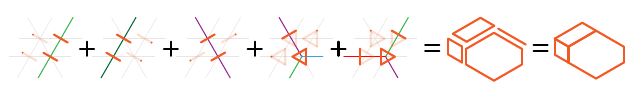}}
	\centerline{\includegraphics[scale=1.5]{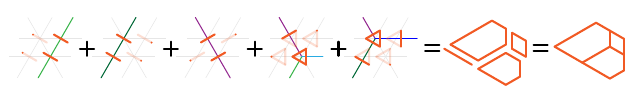}}
	\centerline{\includegraphics[scale=1.5]{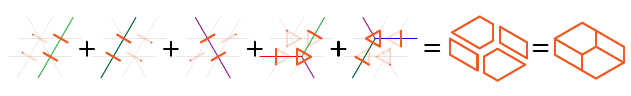}}
	\caption{Some quotientoplexes of the $(2,1,0)$-weak order, obtained as Minkowski sums of shardoplexes.}
	\label{fig:MinkowskiSums}
\end{figure}

\begin{lemma}
\label{lem:MinkowskiSumPolyhedralComplex}
Let~$\polytope{S} \eqdef (\polytope{S}^\b{q})_{\b{q} \in \sTrunks}$ and~$\polytope{T} \eqdef (\polytope{T}^\b{q})_{\b{q} \in \sTrunks}$ be such that for any~$\b{p} \ne \b{q} \in \sTrunks$, there is~$\mu, \nu \in \R$ such that~$\dotprod{\b{q}-\b{p}}{\b{s}} - \mu$ (resp.~$\dotprod{\b{q}-\b{p}}{\b{t}} - \nu$) is non-negative for all~$\b{s} \in \polytope{S}^\b{p}$ (resp.~$\b{t} \in \polytope{T}^\b{p}$) and non-positive for all~$\b{s} \in \polytope{S}^\b{q}$ (resp.~$\b{t} \in \polytope{T}^\b{q}$).
If~$\polytope{S}$ and~$\polytope{T}$, together with all their faces, form two polytopal complexes, then so does their \defn{Minkowski sum}~$\polytope{S} + \polytope{T} \eqdef (\polytope{S}^\b{q} + \polytope{T}^\b{q})_{\b{q} \in \sTrunks}$.
\end{lemma}

\begin{proof}
Let~$\polytope{X}$ and~$\polytope{Y}$ be two faces of~$\polytope{S} + \polytope{T}$ whose intersection is non-empty.
Let~$\b{p}, \b{q} \in \sTrunks$ be such that~$\polytope{X}$ (resp.~$\polytope{Y}$) is a face of~$\polytope{S}^\b{p} + \polytope{T}^\b{p}$ (resp.~of~$\polytope{S}^\b{q} + \polytope{T}^\b{q}$). 
Let~$\mu, \nu \in \R$ be as in the statement.
Then~$\dotprod{\b{q}-\b{p}}{\b{s}+\b{t}} - (\mu + \nu)$ is non-negative for any~$\b{s}+\b{t} \in \polytope{S}^\b{p} + \polytope{T}^\b{p}$ and non-positive for~$\b{s}+\b{t} \in \polytope{S}^\b{q} + \polytope{T}^\b{q}$.
Hence, this dot product must vanish on the intersection~$\polytope{X} \cap \polytope{Y}$.
We conclude that the intersection~$\polytope{X} \cap \polytope{Y}$ lies in the face of~$\polytope{X}$, hence of~$\polytope{S}^\b{p} + \polytope{T}^\b{p}$ (resp.~of~$\polytope{Y}$, hence of~$\polytope{S}^\b{q} + \polytope{T}^\b{q}$) maximizing the direction~$\b{q} - \b{p}$ (resp.~$\b{p} - \b{q}$).
By standard properties of Minkowski sums, the latter is the Minkowski sum of the faces of~$\polytope{S}^\b{p}$ and~$\polytope{T}^\b{p}$ (resp.~of~$\polytope{S}^\b{q}$ and~$\polytope{T}^\b{q}$) maximizing the direction~$\b{q} - \b{p}$ (resp.~$\b{p} - \b{q}$).
Since~$\polytope{S}$ (resp.~$\polytope{T}$) form polytopal complexes, these faces of~$\polytope{S}^\b{p}$ and~$\polytope{S}^\b{q}$ (resp.~of~$\polytope{T}^\b{p}$ and~$\polytope{T}^\b{q}$) intersect along a face~$\polytope{F}$ (resp.~$\polytope{G}$) of both.
We conclude that~$\polytope{X} \cap \polytope{Y} = \polytope{F} + \polytope{G}$ is a face of both~$\polytope{X}$ and~$\polytope{Y}$.
\end{proof}

As the shardoplexes clearly satisfy the conditions of \cref{lem:MinkowskiSumPolyhedralComplex}, we can apply \cref{lem:MinkowskiSumPolyhedralComplex} to define Minkowski sums of shardoplexes.
Note that we will see in \cref{thm:quotientoplexTropical} an alternative and more direct definition for the quotientoplex~$\quotientoplex$ in terms of tropical geometry.

\begin{definition}
\label{def:quotientoplex}
For a congruence~$\equiv$ of the $\s$-weak order~$W_\s$, the \defn{quotientoplex}~$\quotientoplex$ is the Minkowski sum of (any positive scaling of) the shardoplexes~$\shardoplex$ over all $\s$-arcs~$\alpha$ in~$\c{A}_\equiv$.
\end{definition}

If we want to make explicit the scaling coefficients, we denote by~$\quotientoplex(\b{\lambda})$ the quotientoplex obtained as the Minkowski sum~$\sum_{\alpha \in \c{A}_\equiv} \lambda_\alpha \, \shardoplex$ for~$\b{\lambda} \eqdef (\lambda_\alpha)_{\alpha \in \c{A}_\equiv}$ with~${\lambda_\alpha > 0}$.

We now state a generalization of \cref{thm:quotientopes}.
A more refined version will be stated in \cref{thm:quotientFoamTropical} and \cref{thm:quotientoplexTropical} using tropical geometry.

\begin{proposition}
\label{prop:quotientoplex1}
There is an inclusion reversing bijection~$\psi$ from the faces of the quotient foam~$\quotientFoam$ to the faces of the quotientoplex~$\quotientoplex$ such that~$\polytope{F}$ and~$\psi(\polytope{F})$ are orthogonal.
\end{proposition}

\begin{proof}
See \cref{rem:proof_quotientoplex1}.
\end{proof}

\begin{proposition}
\label{prop:quotientoplex2}
For any congruence~$\equiv$ of the $\s$-weak order~$W_\s$, the Hasse diagram of the quotient~$W_\s / {\equiv}$ is isomorphic to the skeleton of the quotientoplex~$\quotientoplex$.
\end{proposition}

\begin{proof}
This follows from \cref{prop:quotientFoam3,prop:quotientoplex1}.
\end{proof}

\begin{proposition}
\label{prop:quotientoplex3}
For an $\s$-arc~$\alpha \eqdef (i, j, A, B, r)$, denote by~$\tilde\alpha \eqdef (i, j, A, B)$ the corresponding classical arc.
For a congruence~$\equiv$ of the $\s$-weak order~$W_\s$, denote by~$\tilde\equiv$ the corresponding congruence of the weak order~$W_n$, with down set of arcs~$\c{A}_{\tilde\equiv} \eqdef \set{\tilde\alpha}{\alpha \in \c{A}_\equiv}$.
Consider~$\b{\lambda} \eqdef (\lambda_\alpha)_{\alpha \in \c{A}_\equiv}$ with~${\lambda_\alpha > 0}$, and let~$\tilde{\b{\lambda}} \eqdef (\tilde\lambda_{\tilde\alpha})_{\tilde\alpha \in \c{A}_{\tilde\equiv}}$ with~$\tilde\lambda_{\tilde\alpha} \eqdef \sum_\alpha \lambda_\alpha$ where the sum ranges over all $\s$-arcs~$\alpha$ which project to~$\tilde\alpha$.
Then the quotientoplex~$\quotientoplex(\b{\lambda})$ is a polytopal subdivision of (a translate of) the quotientope~$\quotientope[\tilde\equiv](\tilde{\b{\lambda}})$.
\end{proposition}

\begin{proof}
For each~$\s$-arc~$\alpha$, the support of the $\alpha$-shardoplex~$\shardoplex$ is (a translate of) the shard polytope~$\shardPolytope[\tilde\alpha]$.
Hence, the support of the Minkowski sum~$\quotientoplex(\b{\lambda}) \eqdef \sum_{\alpha \in \c{A}_\equiv} \lambda_\alpha \shardoplex$ is (a translate of) the Minkowski sum~$\sum_{\alpha \in \c{A}_\equiv} \lambda_\alpha \, \shardPolytope[\tilde\alpha] = \sum_{\tilde\alpha \in \c{A}_{\tilde\equiv}} \big( \sum_{\alpha} \lambda_\alpha \big) \shardPolytope[\tilde\alpha] = \sum_{\tilde\alpha \in \c{A}_{\tilde\equiv}} \tilde\lambda_{\tilde\alpha} \, \shardPolytope[\tilde\alpha] \defeq \quotientope[\tilde\equiv](\tilde{\b{\lambda}})$.
\end{proof}

\begin{remark}
\label{rem:abusing2}
In \cref{prop:quotientoplex3}, note that~$\tilde\equiv$ is strictly speaking not really a congruence of the weak order, for the same reason as \cref{rem:abusing1}.
However, the definition of~$\quotientope[\tilde\equiv](\tilde{\b{\lambda}})$ as a Minkowski sum of pseudoshard polytopes~$\sum_{\tilde\alpha} \tilde\lambda_{\tilde\alpha} \, \shardPolytope[\tilde\alpha]$ is still valid.
\end{remark}

Applying \cref{prop:quotientoplex2,prop:quotientoplex3} to the trivial congruence (where each congruence class contains a single $\s$-tree), we obtain the following statement, answering a question of C.~Ceballos and V.~Pons~\cite{CeballosPons-sWeakOrderI,CeballosPons-sWeakOrderII}.
We note that this question was partially solved in~\cite{DLeonMoralesPhilippeTamayoYip} in the case when~$\s$ contains no~$0$ entry, with a very different method based on a combination of flow polytopes, tropical geometry, and Cayley embedding.

\begin{corollary}
\label{coro:quotientoplexTrivial}
For any $\s$, the Hasse diagram of the $\s$-weak order~$W_\s$ is isomorphic to the oriented skeleton of a polytopal subdivision of a polytope combinatorially equivalent to the zonotope~$\Zono \eqdef \sum_{1 \le i < j \le n} s_i \conv\{\b{e}_i, \b{e}_j\}$.
\end{corollary}

\begin{proof}
By \cref{prop:quotientoplex2,prop:quotientoplex3}, the Hasse diagram of the $\s$-weak order~$W_\s$ is isomorphic to a polyhedral subdivision of~$\quotientope[\tilde\equiv](\tilde{\b{\lambda}})$, where~$\tilde\equiv$ is the projection of the trivial congruence~$\equiv$ of the $\s$-weak order.
The normal fan of~$\quotientope[\tilde\equiv](\tilde{\b{\lambda}})$ is the arrangement of the hyperplanes~$\set{\b{x} \in \R^n}{x_i = x_j}$ for all~$1 \le i < j \le n$ such that there exists an $\s$-arc of the form~$(i, j, A, B, r)$, that is, such that~$s_i \ne 0$.
We conclude that~$\quotientope[\tilde\equiv](\tilde{\b{\lambda}})$ and~$\Zono$ are normally equivalent, hence combinatorially equivalent.
\end{proof}

\begin{remark}
\label{rem:dilationFactors}
In fact, the Minkowski sum of the shardoplexes~$\shardoplex$ for all $\s$-arcs~$\alpha$, with coefficients~$\lambda_\alpha = 1$ is a polytopal subdivision of the zonotope
\[
\sum_{1 \le i < j \le n} \one_{s_i \ne 0}  \Big( \sum_{1 \le g \le i} s_g \, 2^{\#\set{h \in {]g,i[}}{s_h \ne 0}} \Big)  \Big( 1 + \one_{s_j \ne 0}  \sum_{j < \ell \le n} 2^{\#\set{k \in {]j,\ell[}}{s_k \ne 0}} \Big)  \conv\{\b{0}, \b{e}_i -\b{e}_j\}.
\]
For instance, when~$\s = (1, \dots, 1)$, this zonotope
\[
\sum_{1 \le i < j \le n} 2^{n-j+i-1} \conv\{\b{0}, \b{e}_i - \b{e}_j\}
\]
is the sum of all the shard polytopes illustrated in \cite[Fig.~14]{PadrolPilaudRitter}.
\end{remark}

\cref{fig:sQuotientopeLattice} illustrates the quotientoplexes of all congruences of the $(1,2,0)$- and $(2,1,0)$-weak orders.
In \cref{fig:1211-quotientopes,fig:2101-quotientopes,fig:1432-quotientopes,fig:2321-quotientopes}, we have represented the quotientoplexes for the trivial and the sylvester congruences of some $\s$-weak orders.

\begin{figure}[p]
	\capstart
	\centerline{\raisebox{1.2cm}{\includegraphics[scale=.8]{120-quotientopeLattice}} \hspace{.5cm} \includegraphics[scale=.8]{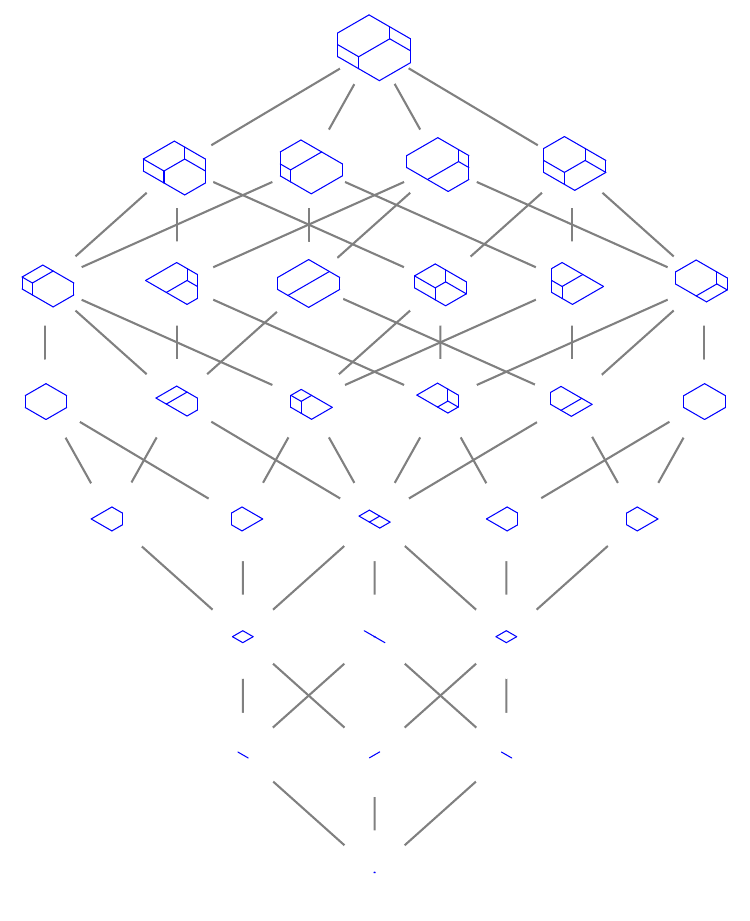}}
	\caption{The congruence lattice of the $\s$-weak order~$W_\s$ for $\s = (1,2,0)$ (left) and~$\s = (2,1,0)$ (right), where each congruence~$\equiv$ is replaced by its quotientoplex~$\quotientope$. See also \cref{fig:sWeakOrderCongruenceLattice,fig:sQuotientFoamLattice}.}
	\label{fig:sQuotientopeLattice}
\end{figure}

\begin{figure}
	\capstart
	\centerline{\includegraphics[scale=.25]{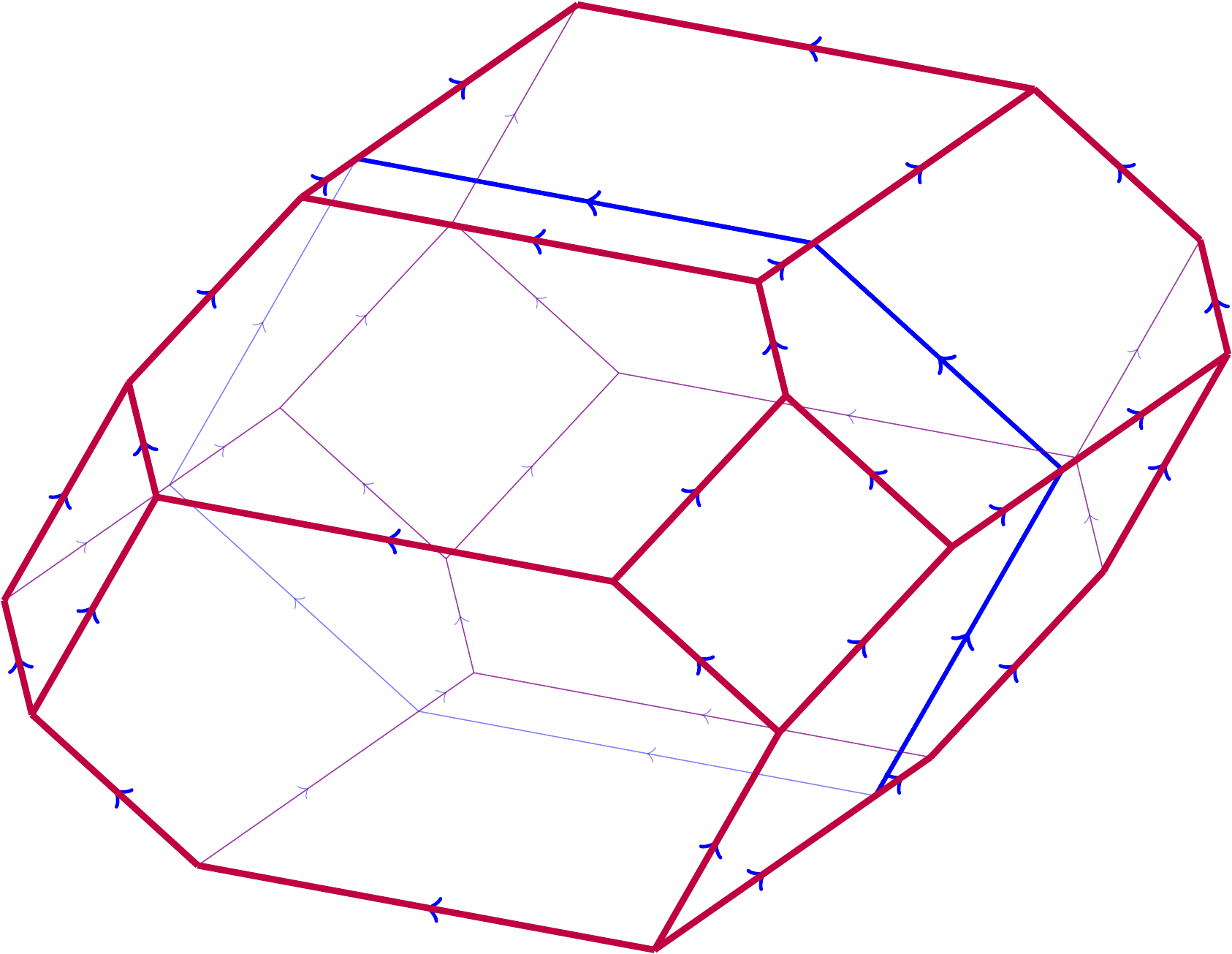}\hspace{-.2cm}\includegraphics[scale=.25]{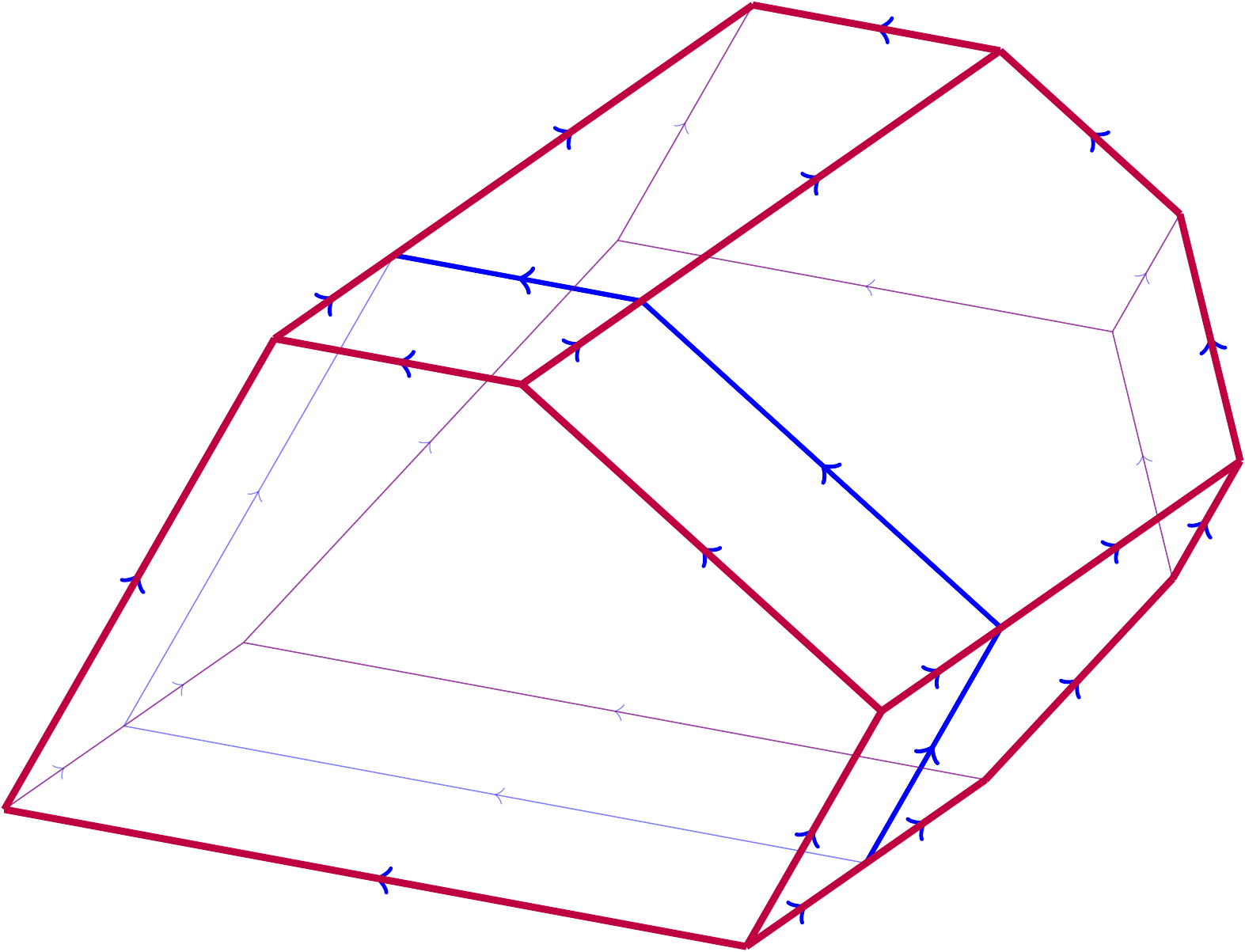}}
	\caption{The $(1,1,2,1)$-permutahedron and the $(1,1,2,1)$-associahedron.}
	\label{fig:1211-quotientopes}
\end{figure}

\begin{figure}
	\capstart
	\centerline{\includegraphics[scale=.22]{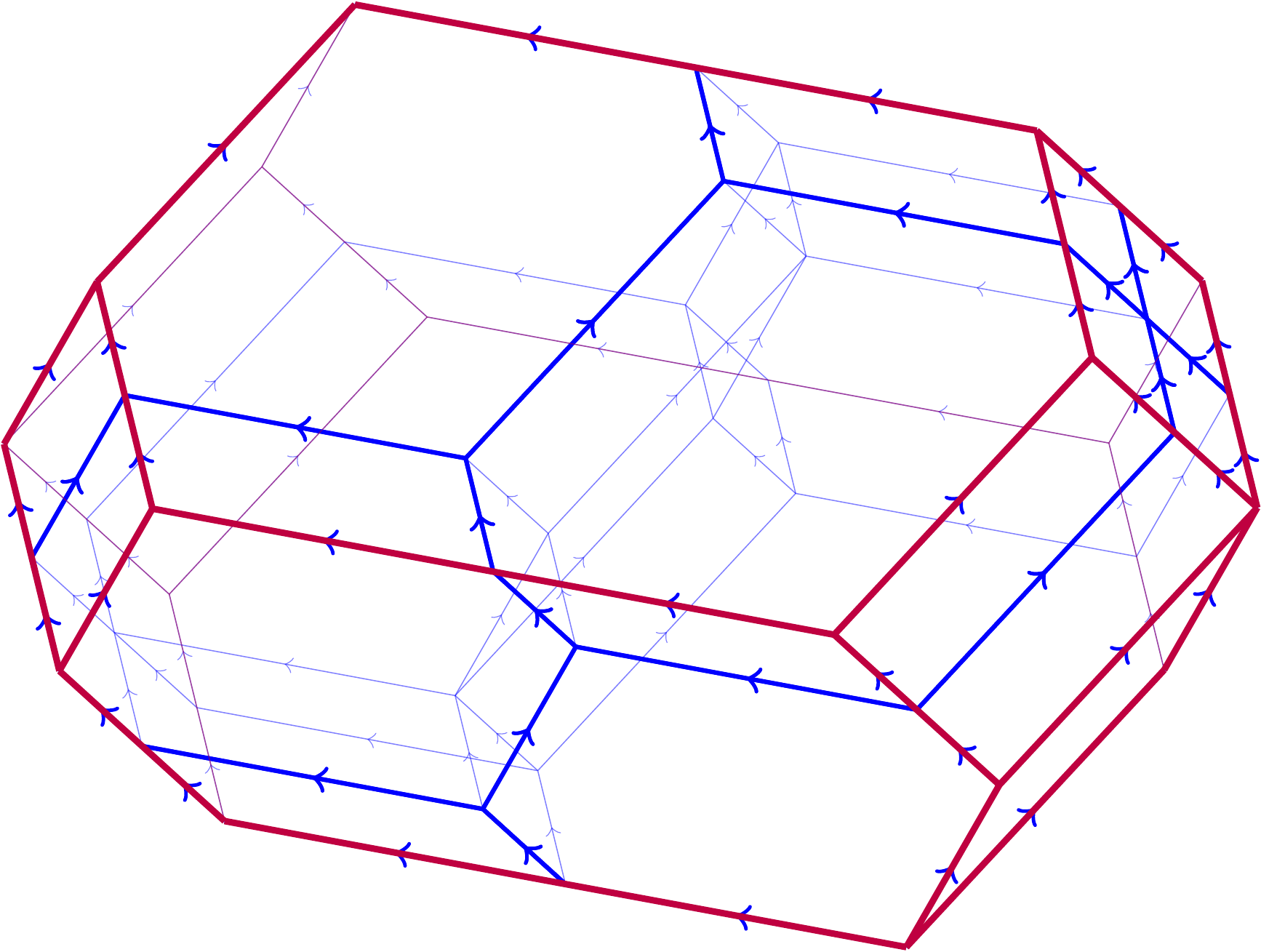}\includegraphics[scale=.22]{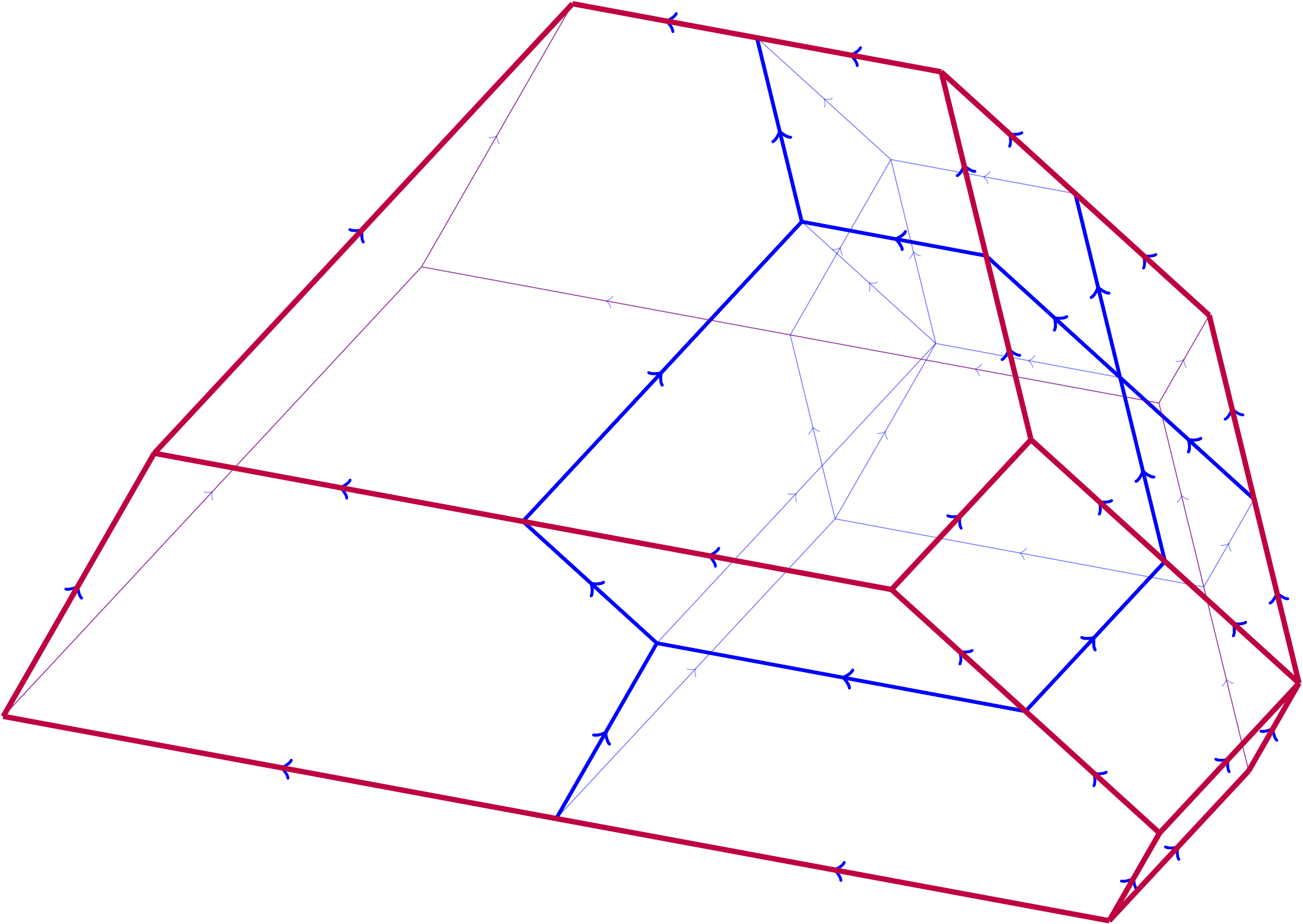}}
	\caption{The $(2,1,0,1)$-permutahedron and the $(2,1,0,1)$-associahedron.}
	\label{fig:2101-quotientopes}
\end{figure}

\begin{figure}
	\capstart
	\centerline{\includegraphics[scale=.13]{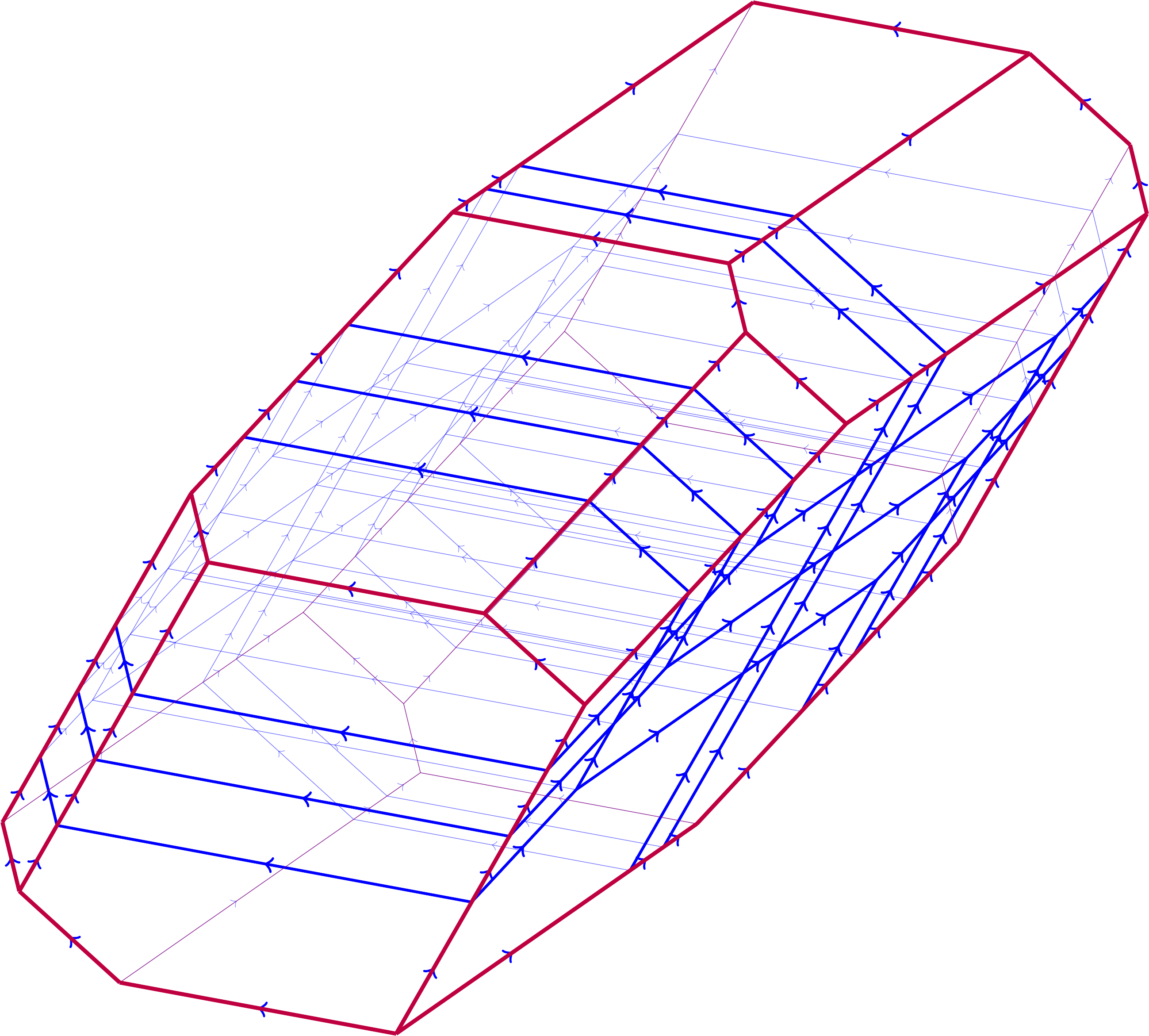}\hspace{-.2cm}\includegraphics[scale=.13]{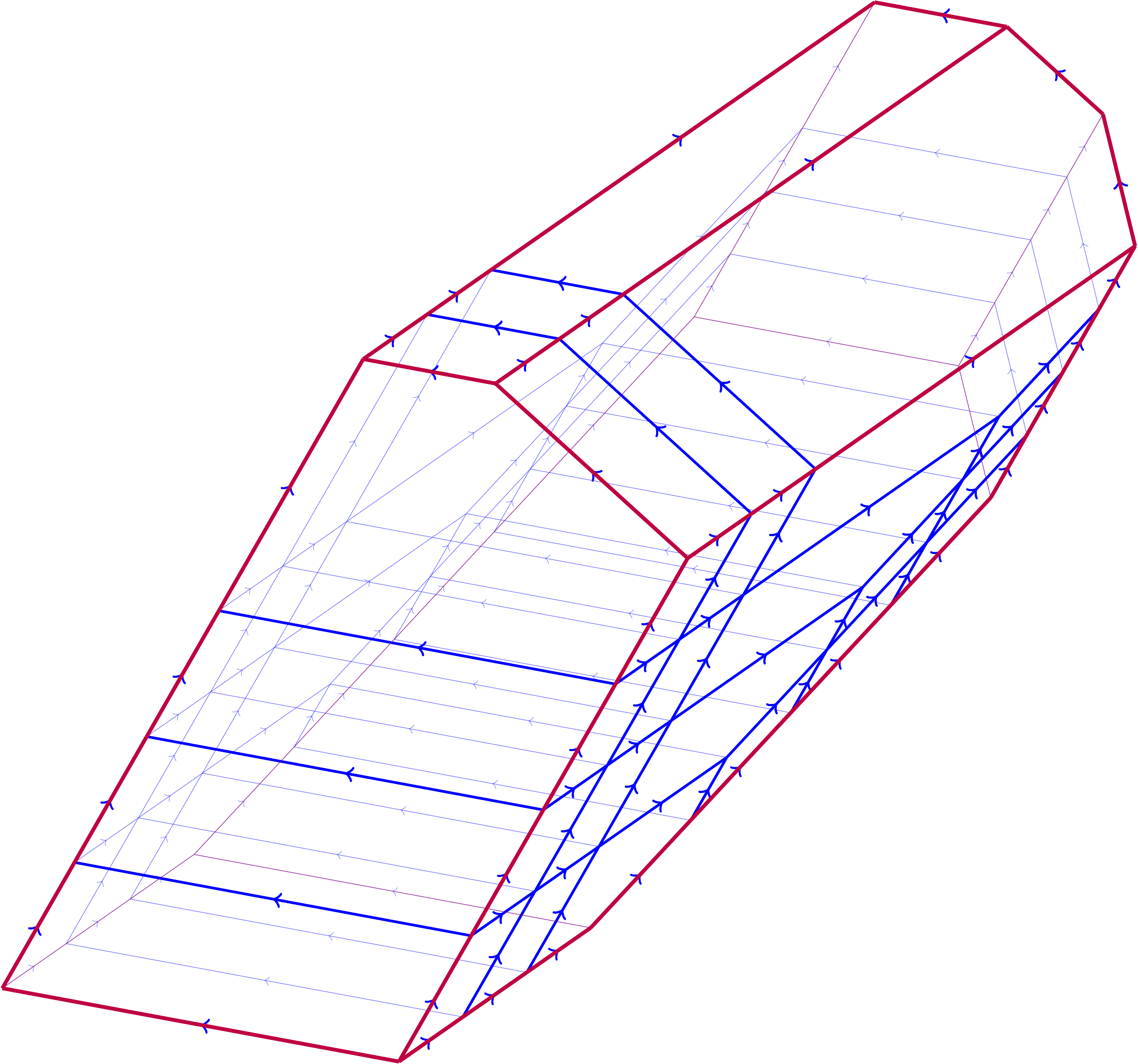}}
	\caption{The $(1,4,3,2)$-permutahedron and the $(1,4,3,2)$-associahedron.}
	\label{fig:1432-quotientopes}
\end{figure}

\begin{figure}
	\capstart
	\centerline{\includegraphics[scale=.13]{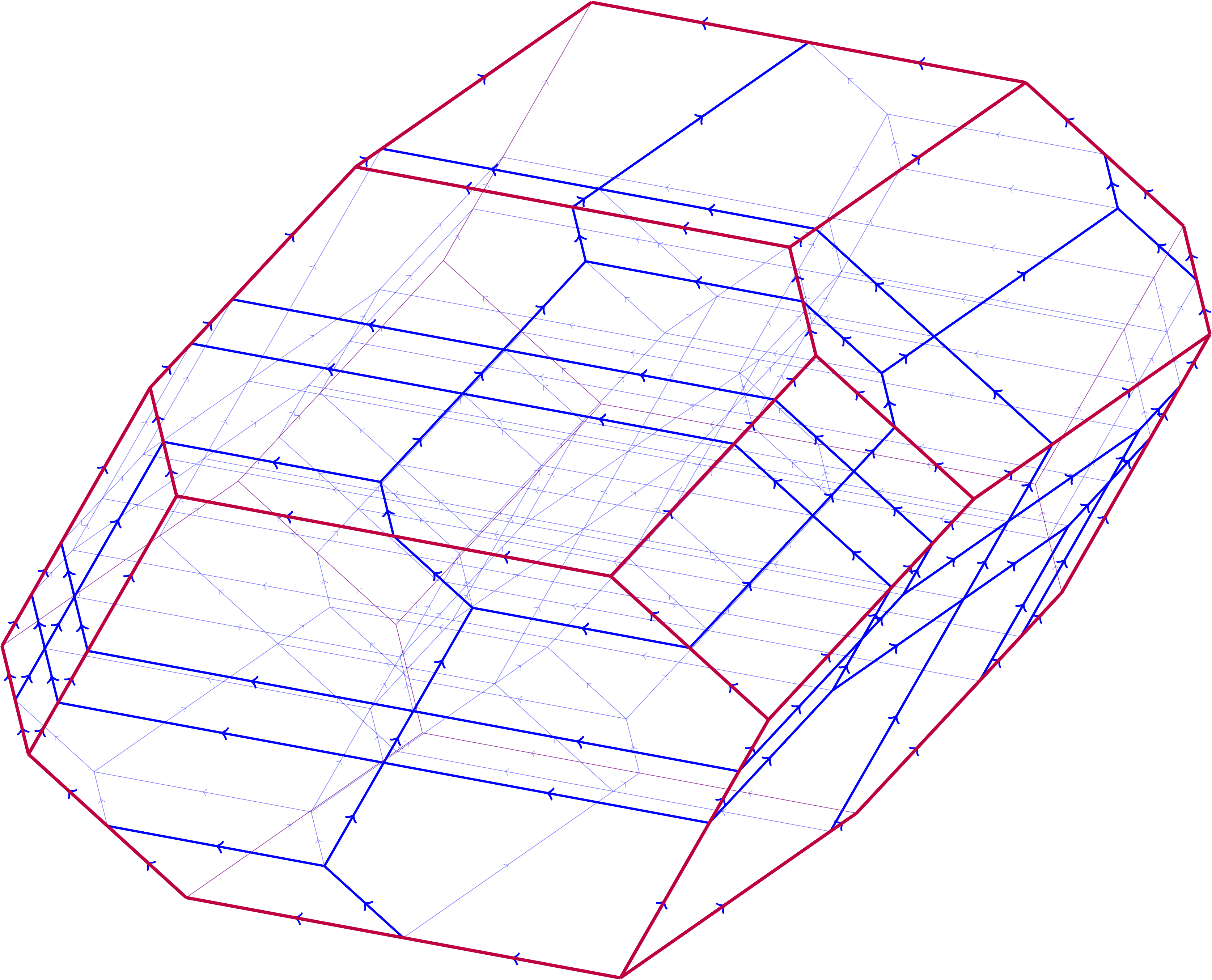}\hspace{-.2cm}\includegraphics[scale=.13]{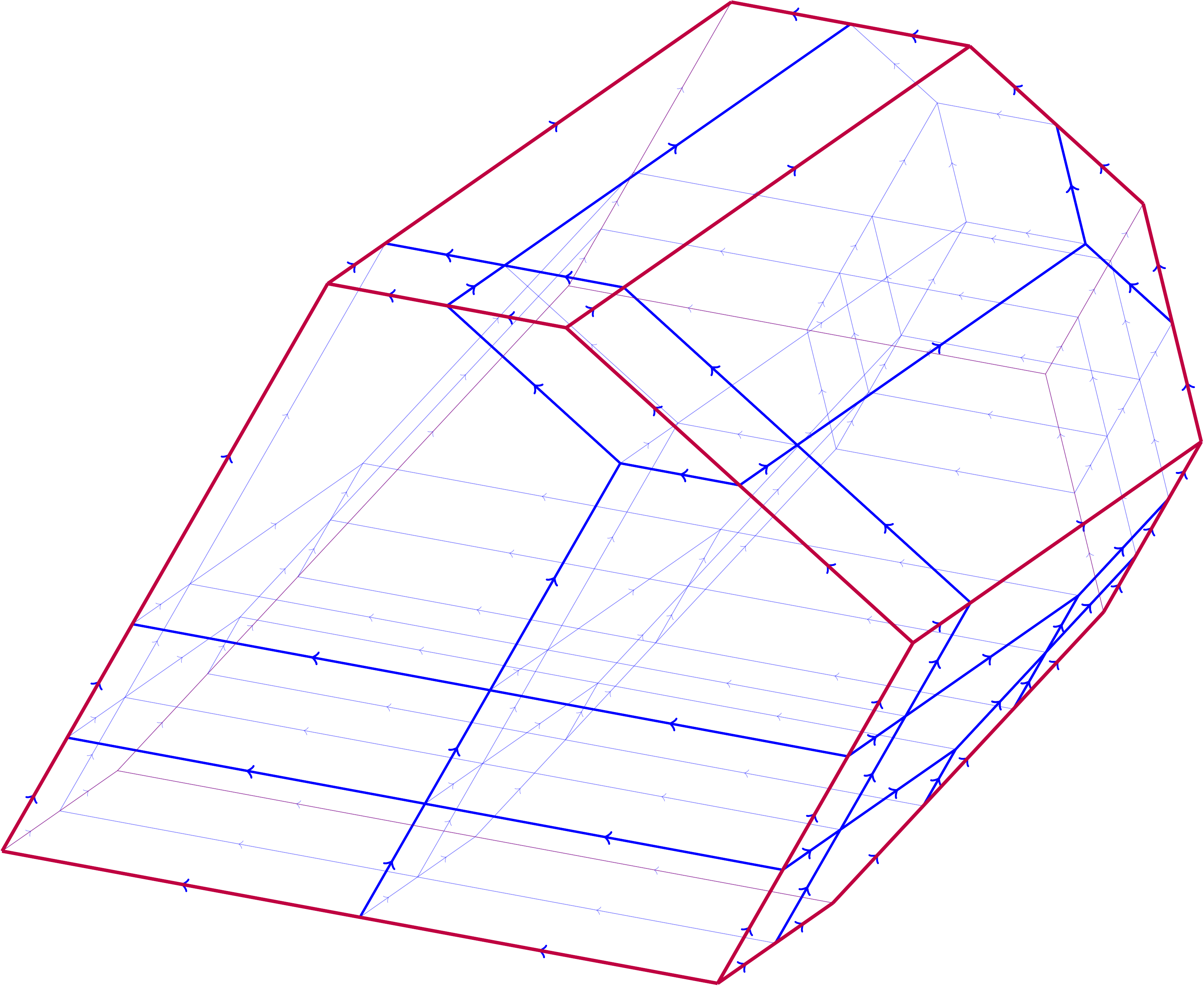}}
	\caption{The $(2,3,2,1)$-permutahedron and the $(2,3,2,1)$-associahedron.}
	\label{fig:2321-quotientopes}
\end{figure}

\begin{remark}
\label{rem:acyclicReorientationLattices}
Consider~$\b{s} \in \{0,1\}^n$, and define the directed graph~$D_\b{s}$ with vertices~$[n]$ and arcs $\set{(i,j)}{1 \le i < j \le n \text{ and } s_i \ne 0}$.
As~$T_j \le 2$ for all~$j \in [n]$, we have~${\prod_{j \in [n]} \max(1, T_{j-1}-1) = 1}$, hence a single $\s$-trunk. Consequently
\begin{itemize}
\item the $\s$-weak order is the acyclic reorientation lattice of~$D_\b{s}$ studied in~\cite{Pilaud-acyclicReorientationLattices} (one can check that~$D_\b{s}$ is indeed a vertebrate and skeletal acyclic directed graph as defined in~\cite{Pilaud-acyclicReorientationLattices}),
\item the $\s$-arcs correspond to the ropes of~$D_\b{s}$, the non-crossing $\s$-arc diagrams correspond to the non-crossing rope diagrams of~$D_\b{s}$, and the subarc order coincide with the subrope order of~$D_\b{s}$ (see~\cite{Pilaud-acyclicReorientationLattices} for the definitions),
\item the $\s$-foam is the graphical arrangement of~$D_\b{s}$, and the $\s$-permutahedron is a graphical zonotope of~$D_\b{s}$ (a Minkowski sum of positive dilates of the segments~$[\b{e}_i, \b{e}_j]$ for all arcs~$(i,j)$ of~$D_\s$, where the dilation factors are directly obtained from \cref{rem:dilationFactors}).
\end{itemize}
\end{remark}


\section{Tropical geometry}
\label{sec:tropicalGeometry}

In this section, we prove the results of \cref{sec:sQuotientopes} via tropical geometry.


\subsection{Recollection \ref{sec:tropicalGeometry}: Polytopal subdivisions and tropical duality}
\label{subsec:tropicalDual}

Tropical geometry offers a convenient setting to dualize regular polyhedral subdivisions, in a sense that we define below. 
This section is based on the work of Joswig in~\cite[Chap.~1]{Joswig21} and~\cite{Joswig16} (except that we define the tropical addition with $\max$ rather than $\min$).

The \defn{tropical semiring} is the set $\tropring\eqdef \R\cup\{\infty\}$ equipped with the \defn{tropical addition} $x\oplus y \eqdef \max(x,y)$ and the \defn{tropical multiplication} $x\odot y \eqdef x+y$.
A \defn{tropical polynomial} on $d$ variables is any function~$\troppol:\R^d\to \R$ of the form
\[
\troppol(\point x) = \bigoplus_{i\in [n]} c_i\odot {\point x}^{\point a_i}
=\max\set{c_i+\sprod{\point a_i}{\point x}}{i\in [n]},
\]
where $n\in \N$ and for all $i\in [n]$, $c_i\in \tropring$ and $\point a_i\in \Z^d$. 
This is a convex piecewise affine function. Note that $\troppol$ is not uniquely determined by the exponents $\point a_i$ and the coefficients $c_i$.

The \defn{tropical hypersurface} defined by~$\troppol$, or \defn{vanishing locus} of $\troppol$, is~the~set
$$\trophyp \eqdef \left\{ \point x\in \R^d\, |\, \text{the maximum of $F(\point x)$ is attained at least twice}\right\}.$$
It is the image codimension-2-skeleton of the \defn{dome} 
$$\tropdome\eqdef \left\{(\point x,y)\in \R^{d+1}\mid \point x\in \R^d, \, y\in \R, \,  y\geq \troppol(\point x) \right\}$$ under the orthogonal projection that omits the last coordinate \cite[Coro.~1.6]{Joswig21}. 

The \defn{cells} of $\trophyp$ are the projections of the faces of $\tropdome$ (here we include the regions of~$\R^d$ delimited by $\trophyp$ as its $d$-dimensional cells.
In fact we are considering the normal complex~$NC(\troppol)$ defined in~\cite[after Exm.~1.7]{Joswig21}).

Note that the cells of $\trophyp$ are invariant under multiplying all $c_i$ and $\point a_i$ by a same scalar $\lambda\in \R$.

Let $\pc=\{\point a_1, \ldots, \point a_n\}$ be a point configuration in $\R^d$ with integer coordinates vertices, 
and $\lift:[n] \to \R$ a \defn{lifting function}.
Such a point configuration together with its lifting function $\lift$ define: 
\begin{itemize}
\item the \defn{regular} subdivision $\subdiv $ of $\pc$ obtained by taking the image of the upper faces of $\conv(\{(\point a_1, \lift(1)), \ldots, (\point a_n, \lift(n))\})$ by the projection that forgets the last coordinate,
\item the \defn{tropical polynomial} $$\troppol(\point x)=\bigoplus_{i\in [n]} \lift(i) \odot \point x^{\point a_i}=\max\left\{\lift(i) + \langle \point a_i, \point x\rangle \, |\, i\in [n] \right\},$$
where $\point x\in \R^d$.
\end{itemize}

We say that $\trophyp$ is the \defn{tropical dual} of the subdivision $\subdiv$ with lifting function~$\lift$ since we have the following theorem.

\begin{theorem}[{\cite[Theorem 1.13]{Joswig21}}]\label{thm:tropical_dual}
There is a bijection between the $k$-dimensional cells of $\subdiv$ and the $(d-k)$-dimensional cells of $\trophyp$, which reverses the inclusion order. 
\end{theorem}

\begin{remark}
\label{rem:strongDuality}
This duality corresponds to what P.~McMullen calls \emph{strong duality} in \cite[Sects.~6~\&~7]{McMullen} (see in particular~\cite[Thm.~7.1]{McMullen}). 
It implies that the affine spans of dual cells are orthogonal.
\end{remark}

We now look at the case of Minkowski sum. 
We consider a family of point configurations $\pc_1, \ldots, \pc_k$ in $\R^d$ such that for all $j\in [k]$, $\pc_j=\{\point a_{j, 1}, \ldots, \point a_{j, m_j}\}$ is a point configuration in $\R^d$ with integer coordinate vertices  and $\lift_j:[m_j] \to \R$ is an associated lifting function.
Their Minkowski sum is the point configuration ${\pc}\eqdef \sum \pc_i = \set{\point a_{1, i_1} + \ldots + \point a_{k, i_k}}{(i_1, \ldots, i_k)\in [m_1]\times \ldots \times [m_k]}$ in  $\R^d$ 
with the lifting function ${\lift}:[m_1]\times\ldots\times [m_k] \to \R$ such that 
${\lift}(i_1, \ldots, i_k)=\sum_{j \in [k]} \lift_j(i_j)$.

The corresponding tropical polynomial is
\begin{align*}
{\troppol}(\point x)
&= \bigoplus_{\substack{(i_1, \ldots, i_k) \\ \in [m_1]\times \ldots \times [m_k]}} {\lift}(i_1, \ldots, i_k)\odot \point x^{\sum_{j \in [k]} \point a_{j, i_j}}
= \bigoplus_{\substack{(i_1, \ldots, i_k) \\ \in [m_1]\times \ldots \times [m_k]}} \bigodot_{j \in [k]} \lift_j(i_j)\odot \point x^{\point a_{j, i_j}}\\
&= \bigodot_{j \in [k]} \bigoplus_{i_j\in [m_j]} \lift_j(i_j)\odot \point x^{\point a_{j, i_j}}
= \bigodot_{j \in [k]} \troppol_j(\point x),
\end{align*}
where $\troppol_j$ is the tropical polynomial associated to $\pc_j$ with lifting $\lift_j$.

\begin{lemma}[{\cite[Lem.~6]{Joswig16}}]
\label{lem:product_troppol}
Let $\troppol$ be a tropical polynomial obtained as a factorization ${\troppol=\bigodot_{j \in [k]} \troppol_j}$. 
Then, the vanishing locus $\trophyp[\troppol]$ is obtained by taking the union of the vanishing loci $\trophyp[\troppol_j]$ for~$j\in [k]$ and the cells of $\trophyp[\troppol]$ are the intersections of the cells of all $\trophyp[\troppol_j]$ for~$ j\in [k]$. 
We say that these cells are \defn{induced} by the arrangement of tropical hypersurfaces $\left\{\trophyp[\troppol_j] \, |\, j\in [k]\right\}$.
\end{lemma}

We have the following statement as a consequence of the previous discussion on Minkowski sum and of Theorem~\ref{thm:tropical_dual}. 

\begin{theorem}\label{thm:arrangement_tropical_hypersurfaces}
The tropical dual of the mixed subdivision ${\subdiv}$ of a Minkowski sum of point configurations $\set{\pc_j}{j\in [k]}$ each with a lifting function $\lift_j$ is the polyhedral complex of cells induced by the arrangement of tropical hypersurfaces $\set{\trophyp[\troppol_j]}{j\in [k]}$.
\end{theorem}


\subsection{Shard tropical hypersurfaces and polynomials}

\begin{definition}
Let $\alpha \eqdef (i,j,A,B,t)$ be an $\s$-arc, and $\tilde{\alpha} \eqdef (i,j,A,B)$ the corresponding classical arc. 
For each $\tilde\alpha$-alternate matching $\altmatch \eqdef \{i_1<j_1<\ldots <i_q<j_q\}$ in $\altmatchset[\tilde\alpha]$, we define the lifting $\lift_\alpha(\altmatch) \eqdef -(r-1)\one_{i_1=i} - \sum_{p\in [q]} \sum_{k\in B\cap {]i_p,j_p[}} \max(0, s_k-1)$. 

This lifting gives rise to the tropical polynomial
\[\troppol_{\alpha} (\point x) = \bigoplus_{\altmatch \in \altmatchset[\tilde\alpha]} \lift_\alpha(\altmatch) \odot \point x^{\charvect_\altmatch},\]
with associated tropical hypersurface $\trophyp[\troppol_{\alpha}]$.
\end{definition}

\begin{lemma}
\label{lem:shardoplexTropical}
The lifting $\lift_\alpha$ induces the trivial subdivision of the shard polytope~$\shardPolytope[\tilde\alpha]$. 
Moreover, the unique minimal cell of~$\trophyp[\troppol_\alpha]$ is the following subspace, of dimension $n-(j-i)$:
\begin{equation*}
\label{eq:trophyp_translation}
\polytope{C}^{\min}_{\alpha} \eqdef \set{\point x \in \R^n}{
\begin{aligned}
x_i - x_j &= r - 1 + \sum_{k \in B} \max(0, s_k-1), \\
x_i - x_a &= r - 1 + \sum_{k \in B \cap {]i,a[}} \max(0, s_k-1) \text{ for all~$a \in A$},\\
x_i - x_b &= r - 1 + \sum_{k \in B \cap {]i,b[}} \max(0, s_k-1) \text{ for all~$b \in B$}.
\end{aligned}
}
\end{equation*}
\end{lemma}

\begin{proof}
A minimal cell of~$\trophyp[\troppol_\alpha]$ is the locus where the maximum of $\troppol_{\alpha}$ is attained at a maximal subset $S$ of $\altmatchset$. 
Let $\point x$ be a point in such a cell. 
We see that 
\begin{itemize}
\item $\left\{\emptyset, \{i<j\}\right\} \subseteq S$ implies that $x_i - x_j = r - 1 + \sum_{k \in B} \max(0, s_k-1)$,  
\item for all $a\in A$, $\left\{\emptyset, \{a<j\}\right\} \subseteq S$ implies that $x_a - x_j = \sum_{k \in B\cap ]a,j[} \max(0, s_k-1)$, 
\item for all $b\in B$, $\left\{\emptyset, \{i<b\}\right\} \subseteq S$ implies that $x_i - x_b = r - 1 + \sum_{k \in B\cap ]i,b[} \max(0, s_k-1)$, 
\end{itemize}
and if all these equations are satisfied then $\left\{\emptyset, \mu \right\} \subseteq S$ for any $\mu \in \altmatchset$. 
Hence there is a unique minimal cell of~$\trophyp[\troppol_\alpha]$, where the maximum of $\troppol_\alpha$ is attained for all $\altmatch\in \altmatchset[\tilde\alpha]$, and it is equal to~$\polytope{C}^{\min}_{\alpha}$.
\end{proof}

\begin{proposition}
\label{prop:shard_trophyp}
For any $\s$-arc $\alpha$, the tropical hypersurface $\trophyp[\troppol_{\alpha}]$
\begin{itemize}
\item contains the shard $\shard$,
\item is contained in the union of the shards $\shard[\beta]$ over all subarcs $\beta$ of $\alpha$.
\end{itemize}
\end{proposition}

\begin{proof}
Note that this proposition is the analogue of \cite[Proposition~48]{PadrolPilaudRitter}, where the normal fan of the shard polytope $\shardPolytope$ is replaced by the tropical hypersurface $\trophyp[\troppol_{\alpha}]$.
The proof is very similar and relies on \cite[Lemma~49]{PadrolPilaudRitter}.

We consider an $\s$-shard $\alpha \eqdef (i,j,A,B,t)$. 
We recall that the {$\s$-shard} of~$\alpha$ is the polyhedron 
\begin{align*}
\shard &= \set{\point x \in \R^n}{
\begin{aligned}
x_i - x_j &= r - 1 + \sum_{k \in B} \max(0, s_k-1), \\
x_i - x_a &\ge r - 1 + \sum_{k \in B \cap {]i,a[}} \max(0, s_k-1) \text{ for all~$a \in A$},\\
x_i - x_b &\le r - 1 + \sum_{k \in B \cap {]i,b[}} \max(0, s_k-1) \text{ for all~$b \in B$}.
\end{aligned}
}\\
&= \set{\point x \in \R^n}{
\begin{aligned}
 - (r - 1) - \sum_{k \in B} \max(0, s_k-1) + x_i-x_j &=0, \\
 - \sum_{k \in B \cap {]a,j[}} \max(0, s_k-1) + x_a - x_j &\le 0 \text{ for all~$a \in A$},\\
 - (r - 1) - \sum_{k \in B \cap {]i,b[}} \max(0, s_k-1) + x_i - x_b &\le 0 \text{ for all~$b \in B$}.
\end{aligned}
}
\end{align*}
This corresponds to the locus where the maximum of $\troppol_{\alpha}$ is attained for both $\mu=\{i<j\}$ and~$\mu=\emptyset$, which shows that $\shard \subseteq \trophyp[\troppol_{\alpha}]$. 
Indeed, for any $\point x \in \shard$ and ${\altmatch \eqdef \{i_1 \! < \! \ldots \! < \! j_q\} \in \altmatchset[\tilde\alpha] \! \ssm \! \{\{i<j\}, \emptyset\}}$, we have
\begin{align*}
\lift_{\alpha}(\altmatch) + \sprod{\charvect_\altmatch}{\point x} 
&= -(r-1)\one_{i_1=i} + \sum_{p\in [q]} \left(- \sum_{k\in B\cap {]i_p,j_p[}} \max(0, s_k-1) + x_{i_p}-x_{j_p} \right)\\
&= -(r-1)\one_{i_1=i} + \sum_{p\in [q]} \left( \begin{array}{c} \displaystyle - \sum_{k\in B\cap {]i_p,j_p[}} \max(0, s_k-1) + x_{i_p}-x_{j_p} \\ \displaystyle -(r-1) - \sum_{k\in B} \max(0, s_k-1) + x_i-x_j \end{array} \right)\\
&= -(r-1)\one_{i_1=i} + \sum_{p\in [q]} \left( \begin{array}{c} \displaystyle -(r-1) - \sum_{k\in B\cap {]i,j_p[}} \max(0, s_k-1) + x_{i}-x_{j_p} \\ \displaystyle - \sum_{k\in B\cap {]i_p,j[}} \max(0, s_k-1) + x_{i_p}-x_{j} \end{array} \right)\\
&\leq 0.
\end{align*}

Now we want to show that any codimension~$1$ cell of $\trophyp[\troppol_{\alpha}]$ is contained in a shard $\shard[\beta]$ for a certain subarc $\beta$ of $\alpha$.
It follows from~\cref{lem:shardoplexTropical} that these cells are in correspondence with the edges of the shard polytope~$\shardPolytope[\tilde\alpha]$. 

Let $\point x \in \trophyp[\troppol_{\alpha}]$, which attains the maximum of $\troppol_\alpha$ at exactly two $\tilde\alpha$-alternate matchings $\altmatch_1$ and~$\altmatch_2$. 
This implies that there is an edge between the vertices $\charvect_{\altmatch_1}$ and $\charvect_{\altmatch_2}$ in the shard polytope~$\shardPolytope[\tilde\alpha]$ and we are in one of the four cases of \cite[Lemma~49]{PadrolPilaudRitter}.
We give the details for the first one and for the others we only specify $\beta$, the computations are similar:
\begin{enumerate}
\item If $\altmatch_1=H<i'<j'<K$ and $\altmatch_2=H<K$, we define $\beta\eqdef (i',j',A',B',r')$ with $A'\eqdef A\cap {]i',j'[}$, $B'\eqdef B\cap {]i',j'[}$ and $r'\eqdef r$ if $i'=i$ or $r'\eqdef 1$ otherwise. 
Then we have:
\begin{itemize}
\item 
$
0 = \lift_\alpha (\altmatch_1) + \sprod{\charvect_{\altmatch_1}}{\point x} - \lift_\alpha (\altmatch_2) + \sprod{\charvect_{\altmatch_2}}{\point x}
= -(r-1)\one_{i'=i}-\sum_{k\in B\cap {]i',j'[}} {\max(0, s_k-1)} + x_{i'}-x_{j'}.
$
\item for any $a\in A'$, 
\begin{align*}
0 &< \lift_\alpha (\altmatch_1) + \sprod{\charvect_{\altmatch_1}}{\point x} - \lift_\alpha (\altmatch_3) + \sprod{\charvect_{\altmatch_3}}{\point x}\\
&= -(r-1)\one_{i'=i}-\sum_{k\in B\cap {]i',a[}} \max(0, s_k-1) + x_{i'}-x_{a},
\end{align*}
where $\altmatch_3$ denotes the $\tilde\alpha$-alternate matching $\altmatch_3 \eqdef H<a<j'<K$.
\item for any $b\in B'$, 
\begin{align*}
0 &< \lift_\alpha (\altmatch_1) + \sprod{\charvect_{\altmatch_1}}{\point x} - \lift_\alpha (\altmatch_3) + \sprod{\charvect_{\altmatch_3}}{\point x}\\
&= -\sum_{k\in B\cap [b,j'[} \max(0, s_k-1) + x_{b}-x_{j'}\\
&= (r-1)\one_{i'=i}+\sum_{k\in B\cap {]i',b[}} \max(0, s_k-1) + x_{b}-x_{i'},
\end{align*}
where $\altmatch_3$ denotes the $\tilde\alpha$-alternate matching $\altmatch_3 \eqdef H<i'<b<K$.
\end{itemize}
\item If $\altmatch_1=H<i'<j_1<K$ and $\altmatch_2=H<i'<j_2<K$ with $j_1<j_2$, we define $\beta\eqdef (j_1,j_2,A',B', s_{j_1})$ with $A'\eqdef A\cap {]j_1,j_2[}$, $B'\eqdef B\cap {]j_1,j_2[}$.
\item If $\altmatch_1=H<i_1<j'<K$ and $\altmatch_2=H<i_2<j'<K$, we define $\beta\eqdef (i_1,i_2, A',B',r')$ with $A'\eqdef A\cap {]i_1,i_2[}$, $B'\eqdef B\cap {]i_1,i_2[}$ and $r'\eqdef r$ if $i_1=i$ or $r'\eqdef 1$ otherwise.
\item If $\altmatch_1=H<i_1<j_1<i_2<j_2<K$ and $\altmatch_2=H<i_1<j_2<K$, we define $\beta\eqdef (j_1,i_2,A',B',s_{j_1})$ with $A'\eqdef A\cap {]i_1,j_2[}$, $B'\eqdef B\cap {]i_1,j_2[}$.
\end{enumerate}

We see that in all cases $\beta$ is a subarc of $\alpha$ and $\point x \in \shard[\beta]$.
\end{proof}

As a corollary of~\cref{prop:shard_trophyp}, \cref{lem:product_troppol} and~\cref{thm:arrangement_tropical_hypersurfaces} we obtain directly the following statement.

\begin{theorem}
\label{thm:quotientFoamTropical}
For any congruence $\equiv$ of the $\s$-weak order, the quotient foam $\quotientFoam$ is the polyhedral complex induced by the arrangement of tropical hypersurfaces $\set{\trophyp[\troppol_{\alpha}]}{\alpha \in \c{A}_{\equiv}}$. 
In particular, it is the tropical dual of the regular subdivision of the Minkowski sum of the point configurations~$\set{\charvect_\altmatch}{\altmatch\in \altmatchset[\tilde\alpha]}$ with lifting function $\lift_\alpha$, over all $\alpha \in \c{A}_\equiv$.
\end{theorem}

\begin{theorem}
\label{thm:quotientoplexTropical}
For any congruence $\equiv$ of the $\s$-weak order, for any $\b{\lambda} \eqdef (\lambda_\alpha)_{\alpha \in \c{A}_\equiv}$ with $\lambda_\alpha >0$, the cells of the quotientoplex~$\quotientoplex(\b{\lambda})$ are the cells of the regular subdivision of the Minkowski sum of the point configurations~$\set{\lambda_\alpha \charvect_\altmatch}{\altmatch\in \altmatchset[\tilde\alpha]}$ with lifting function $\lambda_\alpha \lift_\alpha$, over all $\alpha \in \c{A}_\equiv$.
\end{theorem}

\begin{proof}
Let $\alpha \eqdef (i,j,A,B,r) \in \c{A}$. 
\cref{lem:shardoplexTropical} implies that the cells of $\trophyp[\troppol_\alpha]$ are exactly the translation of the cones of the normal fan of the shard polytope $\lambda_\alpha \shardPolytope[\tilde\alpha]$ by any vector $\point x \in \polytope{C}^{\min}_{\alpha}$. 
This means that the minimal cell of $\trophyp[\troppol_{\alpha}]$ that contains a point $\point y$ is the tropical dual of the face of $\lambda_\alpha \shardPolytope[\tilde\alpha]$ with lifting function $\lambda_\alpha \lift_\alpha$, maximized in the direction $\point y - \point x$ for any $\point x \in \polytope{C}^{\min}_{\alpha}$.

Let $\b{q} \in \sTrunks$.
The tropical dual of the local shard polytope~$\lambda_\alpha \localShardPolytope$ with lifting function $\lambda_\alpha \lift_\alpha$ is the minimal cell of $\trophyp[\troppol_{\alpha}]$ that contains the insertion fiber~$\fiber{\trunk_\b{q}}$ (described in~\cref{exm:grid}).
Indeed, the point $\point y \eqdef -\b{q}$ is in $\fiber{\trunk_\b{q}}$, the point $\point x \eqdef - \sum_{\ell\in ]i,j]} \big(q_i +r-1 +\sum_{k\in B\cap ]i,l[} \max(0,s_k-1)\big) \b{e}_\ell - \sum_{\ell \in [i]\cup ]j,n]} q_\ell \b{e}_\ell$ is in $\polytope{C}^{\min}_{\alpha}$, and $\point y - \point x= \sum_{\ell \in {]i,j]}} \big( q_i - q_\ell + r - 1 + \sum_{k \in B \cap {]i,\ell[}} \max(0, s_k-1) \big) \b{e}_\ell$ is a direction along which the face~$\localShardPolytope$ is maximized in the polytope~$\shardPolytope[\tilde\alpha]$ (\cref{def:localShardPolytope}).

This implies that the tropical dual of the sum $\sum_{\alpha\in \c{A}} \lambda_\alpha \localShardPolytope$ with lifting function $\lambda_\alpha \lift_\alpha$ for each summand is the minimal cell of the arrangement $\set{\trophyp[\troppol_{\alpha}]}{\alpha \in \c{A}_{\equiv}}$ that contains~$\fiber{\trunk_\b{q}}$. 

Reciprocally, we only need to check that each minimal cell $\polytope{C}$ of the arrangement $\set{\trophyp[\troppol_{\alpha}]}{\alpha \in \c{A}_{\equiv}}$ is the tropical dual of a sum $\sum_{\alpha\in \c{A}} \lambda_\alpha \localShardPolytope$ for a certain $\b{q}\in \sTrunks$. 
It follows from~\cref{thm:quotientFoamTropical} that such a~$\polytope{C}$ is a minimal cell of $\quotientFoam$, thus it is also a minimal cell of the $\s$-foam~$\sFoam$, that is of the form $\fiber{\trunk_\b{q}}$ for a $\b{q}\in \sTrunks$. 
It follows from the previous discussion that $\polytope{C}$ is the tropical dual of the sum $\lambda_\alpha \sum_{\alpha\in \c{A}} \localShardPolytope$.
\end{proof}

\begin{remark}
\label{rem:proof_quotientoplex1}
\cref{prop:quotientoplex1} is a consequence of \cref{thm:quotientoplexTropical}, \cref{thm:quotientFoamTropical} and \cref{rem:strongDuality}.
\end{remark}


%


\section*{Acknowledgements}

We are grateful to Daniel Tamayo Jiménez for his collaboration in the early stages of this work, to Arnau Padrol for many suggestions and comments during this project, to Cesar Ceballos and Viviane Pons for various discussions on the $\s$-weak order that clarified the intersection of this paper with their ongoing work, and to Spencer Backman for suggesting \cref{rem:acyclicReorientationLattices}.


\bibliographystyle{alpha}
\bibliography{squotientopes}
\label{sec:biblio}


%
%

\end{document}

%% file: figures/polygons.tex
\begin{tabular}{c@{\quad}c@{\quad}c@{\quad}c}
\begin{tikzpicture}[scale=.6]
\node (T) at (0,0) {$\tree[T]$};
\node (R) at (-2,2) {$\tree[R]$};
\node (S) at (2,2) {$\tree[S]$};
\node (RS) at (0,4) {$\tree[R] \join \tree[S]$};

\draw (T) edge node[xshift=-.5cm] {$(a,b)$} (R);
\draw (T) edge node[xshift=.5cm] {$(c,d)$} (S);
\draw (R) edge node[xshift=-.5cm] {$(c,d)$} (RS);
\draw (S) edge node[xshift=.5cm] {$(a,b)$} (RS);
\end{tikzpicture}
&
\begin{tikzpicture}[scale=.6]
\node (T) at (0,0) {$\tree[T]$};
\node (R) at (-2,1.4) {$\tree[R]$};
\node (R2) at (-2,2.8) {$\tree[R]'$};
\node (S) at (2,2.1) {$\tree[S]$};
\node (RS) at (0,4.2) {$\tree[R] \join \tree[S]$};

\draw (T) edge node[xshift=-.5cm] {$(a,c)$} (R);
\draw (R) edge node[xshift=-.5cm] {$(a,d)$} (R2);
\draw (T) edge node[xshift=.5cm] {$(c,d)$} (S);
\draw (R2) edge node[xshift=-.5cm] {$(c,d)$} (RS);
\draw (S) edge node[xshift=.5cm] {$(a,c)$} (RS);
\end{tikzpicture}
&
\begin{tikzpicture}[scale=.6]
\node (T) at (0,0) {$\tree[T]$};
\node (R) at (-2,2.1) {$\tree[R]$};
\node (S) at (2,1.4) {$\tree[S]$};
\node (S2) at (2,2.8) {$\tree[S]'$};
\node (RS) at (0,4.2) {$\tree[R] \join \tree[S]$};

\draw (T) edge node[xshift=-.5cm] {$(a,c)$} (R);
\draw (T) edge node[xshift=.5cm] {$(c,d)$} (S);
\draw (S) edge node[xshift=.5cm] {$(a,d)$} (S2);
\draw (S2) edge node[xshift=.5cm] {$(a,c)$} (RS);
\draw (R) edge node[xshift=-.5cm] {$(c,d)$} (RS);
\end{tikzpicture}
&
\begin{tikzpicture}[scale=.6]
\node (T) at (0,0) {$\tree[T]$};
\node (R) at (-2,1.4) {$\tree[R]$};
\node (R2) at (-2,2.8) {$\tree[R]'$};
\node (S) at (2,1.4) {$\tree[S]$};
\node (S2) at (2,2.8) {$\tree[S]'$};
\node (RS) at (0,4.2) {$\tree[R] \join \tree[S]$};

\draw (T) edge node[xshift=-.5cm] {$(a,c)$} (R);
\draw (R) edge node[xshift=-.5cm] {$(a,d)$} (R2);
\draw (S) edge node[xshift=.5cm] {$(a,d)$} (S2);
\draw (T) edge node[xshift=.5cm] {$(c,d)$} (S);
\draw (R2) edge node[xshift=-.5cm] {$(c,d)$} (RS);
\draw (S2) edge node[xshift=.5cm] {$(a,c)$} (RS);
\end{tikzpicture}
\\
\includegraphics[scale=.6]{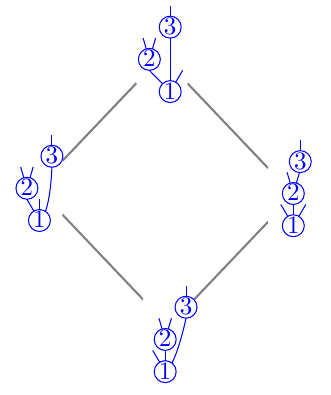}
&
\qquad\includegraphics[scale=.6]{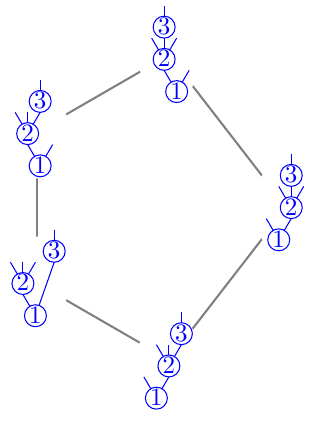}
&
\includegraphics[scale=.6]{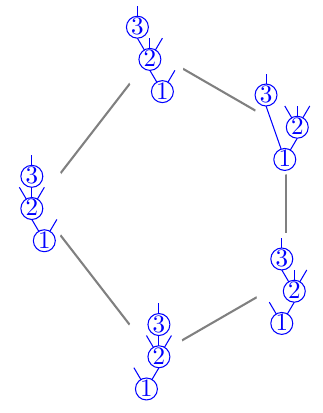}\qquad
&
\includegraphics[scale=.6]{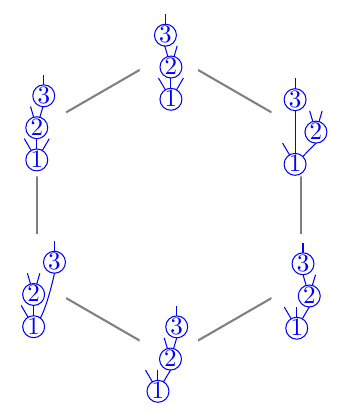}
\end{tabular}